\documentclass{template}

\newunicodechar{≈}{$\bm{\approx}$}
\newunicodechar{Π}{$\mathsf{\Pi}$}
\newunicodechar{π}{$\uppi$}
\newunicodechar{∈}{$\in$}
\newunicodechar{ℕ}{$\N$}
\newunicodechar{₀}{$\mathsf{_0}$}
\newunicodechar{₁}{$\mathsf{_1}$}
\newunicodechar{̃}{\!\!\!\~{}}
\newunicodechar{ε}{$\upvarepsilon$}
\newunicodechar{⋅}{$\cdot$}

\KOMAoptions{twoside = true}

\newtheorem{theorem}{Theorem}[section]
\newtheorem{lemma}[theorem]{Lemma}
\newtheorem{proposition}[theorem]{Proposition}
\newtheorem{corollary}[theorem]{Corollary}
\newtheorem{conjecture}[theorem]{Conjecture}
\theoremstyle{definition}
\newtheorem{definition}[theorem]{Definition}

\newcommand*{\N}{\mathbb{N}}
\newcommand*{\Z}{\mathbb{Z}}

\newcommand*{\R}{\mathbb{R}}
\newcommand*{\C}{\mathbb{C}}
\newcommand*{\F}{\mathbb{F}}

\DeclareMathOperator{\rowsp}{rowsp}
\DeclareMathOperator{\colsp}{colsp}
\DeclareMathOperator{\spn}{span}

\newcommand*{\tran}{^{\mkern-1.5mu{\scriptscriptstyle\mathsf{T}}}}
\newcommand*{\tranconj}{^{\mathsf{H}}}
\newcommand*{\mat}[1]{\mathbf{#1}}
\newcommand*{\txteq}[1]{\overset{\text{\tiny #1}}{=}}
\newcommand*{\txtneq}[1]{\overset{\text{\tiny #1}}{\neq}}
\newcommand*{\txtleq}[1]{\overset{\text{\tiny #1}}{\leq}}
\newcommand*{\txtimplies}[1]{\overset{\text{\tiny #1}}{\implies}}
\newcommand*{\txtin}[1]{\overset{\text{\tiny #1}}{\in}}

\newcommand*{\hpipe}{\rotatebox[origin=c]{90}{$|$}}
\newcommand*\pmat[1]{\begin{pmatrix}#1\end{pmatrix}}

\newcommand*{\macheps}{\varepsilon_{\mathrm{M}}}

\DeclarePairedDelimiter{\abs}{\lvert}{\rvert}
\DeclarePairedDelimiter{\norm}{\lVert}{\rVert}
\DeclarePairedDelimiter{\inp}{\langle}{\rangle}

\begin{document}

\newcommand{\Dname}{Patrick Sonnentag}
\newcommand{\Dnumber}{1233353}
\newcommand{\Dtitel}{Finding the Smallest Possible Exact Aggregation of a Markov Chain}
\newcommand{\Dthesis}{Bachelor's thesis}
\newcommand{\Dprofa}{Prof. Dr. Markus Siegle}
\newcommand{\Dprofb}{Prof. Dr. Maximilian Moll}
\newcommand{\Dadvisora}{Fabian Michel}
\newcommand{\Dday}{25.04.2025}

\frontmatter

\mainmatter

\chapter{Introduction}\label{ch:introduction}
Markov chains are used to model many different types of systems.
Such models include computer systems~\cite{haverkort2000cluster,wang2008rsvp}, biological or chemical processes~\cite{gillespie1977lotka,kierzek2001gene} and economic systems~\cite{simon1961ncdchains}.
These applications can yield very large Markov chains, for which an ordinary numerical analysis may no longer be computationally feasible.
This intractability can be remedied by reducing the size of a Markov chain while keeping important features.
As such, reduction of the computational eﬀort needed for analysis through state space reduction of a Markov chain has been studied in, for example,~\cite{abate2021aggregation,bittracher2021aggregation,bucholz2014aggregation, michel2025errbndmarkovaggr,simon1961ncdchains}.
However, all algorithms presented in the above papers reduce the state space so that the smaller system is again a Markov chain.
We remove this restriction, giving us greater freedom to choose such a reduction.
Specifically, we want the reduced system to, in some way, span the same, or a similar, subspace as the original system.
We achieve this by using the Arnoldi iteration, which was initially intended for approximating eigenvalues and -vectors.

We introduce basic notions such as Markov chains and the Arnoldi iteration in Chapter~\ref{ch:theoretical-preliminaries}.
This chapter serves as an introduction for the uninitiated.
As such, we provide a detailed explanation of existing results with no new content.
Then, in Chapter~\ref{ch:finding-small-aggregations}, we combine the Arnoldi iteration with Markov chains to provide a theoretical description of how we can use the Arnoldi iteration to find aggregations of a given Markov chain.
We further show that the already known error bounds for aggregations can be tight in this special case.
In Chapter~\ref{ch:numerical-implementation}, we analyse the implementation of this novel aggregation method to show where numerical errors might arise, how to combat them, and how it behaves in terms of runtime and memory complexity.
We further give a heuristic to determine when the aggregation best approximates the original Markov chain.
Lastly, we compare this aggregation method to another one on a range of different models and their respective Markov chains.
In the end, in Chapter~\ref{ch:conclusion}, we summarize the results and give an outlook on where this aggregation method might be improved in what ways.

\chapter{Theoretical preliminaries}\label{ch:theoretical-preliminaries}
This chapter serves the purpose of putting together the theoretical framework that will serve as the foundation of our ideas.
In Section~\ref{sec:markov-chains-and-aggr}, we will formally introduce the notion of discrete-time Markov chains and their aggregations.
Section~\ref{sec:error-bounds-on-approximated-transient-distributions} looks at how to determine the error made by aggregating a Markov chain.
Lastly, we look at the so-called Arnoldi iteration in Section~\ref{sec:arnoldi-iteration}, originally devised to calculate eigenvalues and eigenvectors.
However, we will later use it as a method of finding aggregations of Markov chains.
\section{Markov chains and their aggregation}\label{sec:markov-chains-and-aggr}
Roughly speaking, a Markov chain is a random process in which the probability for the next state of this process depends solely on its current state.
It may be easier to view such a process as the random traversal of a weighted and directed graph.
Intuitively, these two situations are equivalent:
At each node of the graph, corresponding to a state in the process, we can only travel along an outgoing edge to the next node or state.
We select which edge to take by looking at its weight.
More precisely, we need the weight of an edge $(a, b)$ to be the probability that we switch from state $a$ to $b$ in the process.
Take, for example:
\begin{figure}[H]
    \centering
    \begin{tikzcd}
        a \arrow[rrr, "\frac{1}{3}"', bend left] \arrow[ddd, "\frac{2}{3}", bend left]                      &  &  & b \arrow[ddd, "1"', bend left]                                          \\
        &  &  &                                                                         \\
        &  &  &                                                                         \\
        d \arrow[uuu, "\frac{3}{4}"', bend left] \arrow["\frac{1}{4}", loop, distance=2em, in=305, out=235] &  &  & c \arrow[lll, "\frac{1}{2}"', bend left] \arrow[llluuu, "\frac{1}{2}"']
    \end{tikzcd}
    \caption{A weighted, directed graph representing a Markov chain with four states}
    \label{fig:ex-markov-chain}
\end{figure}
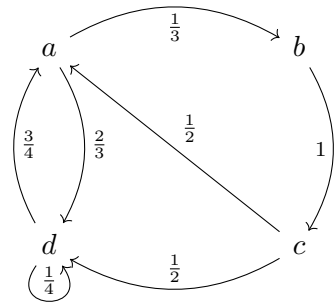
We can see the probability of reaching each state depending on the current state.
If we are in state $a$, we can move to $b$ with probability $\frac{1}{3}$ or to $d$ with $\frac{2}{3}$.
One cannot reach state $c$ in one step if we are in state $a$, as no edge connects the two corresponding nodes.
However, usually, we want to know more than what state we will probably end up in the next step, for example, knowing the same after arbitrarily many steps.
In many applications, this question is central.
For instance, in a variant of the so-called \emph{Ressource Reservation Protocol}, seen in~\cite{wang2008rsvp}, each device in a network is modelled as a Markov chain, with states for it not using any resources, requesting resources, giving them back, and so on.
To be able to make predictions about when a network using this protocol reaches its finite capacity, one needs to determine the likelihood of each device being in a particular state.

For this section, we take the definitions directly related to Markov chains from~\cite{blitzstein2019introductionmarkov} while only partially changing minor things like the naming of indices.
Additionally, a \enquote{marginal distribution} there is called a \enquote{transient distribution} here.
\begin{definition}[\Ac{DTMC}]\label{def:dtmc}
A \emph{discrete-time Markov chain} with finite \emph{state space} $S \coloneqq \set{1,\dots,n}$ is a sequence of random variables $X_0, X_1,\ldots \in S$ where for all $k$ it is
\[
    P(X_{k+1} = x\, |\, X_0 = x_0, X_1 = x_1, \dots, X_k = x_k) = P(X_{k+1} = x\, |\, X_k = x_k).
\]
\end{definition}
From here on, we will always refer to a \ac{DTMC} when mentioning just Markov chains.
\begin{definition}[Transition matrix]\label{def:trns-mtrx}
The \emph{transition matrix} $\mat{P} \in \R^{n \times n}$, holding all \emph{transition probabilities} of a Markov chain, has the entries
\[
    \mat{P}(i, j) \coloneqq P(X_{k+1} = j\, |\, X_k = i).
\]
\end{definition}
Usually, the transition matrix is the sole descriptor of a Markov chain, as the matrix size immediately indicates the state space size.
\begin{definition}[Initial and transient distribution]\label{def:initandtransdistr}
\label{def:ini-trns-distr}
Given a Markov chain and an \emph{initial distribution} $p_0 \in \R^n$, meaning we have $p_0(i) = P(X_0 = i)$, we call
\[
    (P(X_k = 1), \dots, P(X_k = n))\tran \in \R^n
\]
the \emph{$k$-step transient distribution} $p_k$ of the Markov chain with starting distribution $p_0$.
\end{definition}
\begin{proposition}\label{prop:ksteptransdistr}
Given a Markov chain with an initial distribution $p_0 \in \R^n$, we can represent any $k$-step transient distribution as $p_0\tran \mat{P}^k$.
\end{proposition}
\begin{proof}
    We proceed by induction on $k$.
    For the base case $k = 0$, we get
    \[
        p_0\tran \mat{P}^0 = p_0\tran \mat{I} = p_0\tran \txteq{Def. \ref{def:ini-trns-distr}} (P(X_0 = 1), \dots, P(X_0 = n)).
    \]
    In the inductive step, one can reason with the \ac{IH} and $j \in S$ that
    \begin{align*}
        \left( p_0\tran \mat{P}^{k + 1} \right)\!(j) &= \left( p_0\tran \mat{P}^k \mat{P} \right)\!(j)\\
        &\txteq{\ac{IH}} \left( (P(X_k = 1), \dots, P(X_k = n)) \mat{P} \right)\!(j) = \sum_{i \in S} P(X_k = i) \mat{P}(i,j)\\
        &\txteq{Def. \ref{def:trns-mtrx}} \sum_{i \in S} P(X_k = i) P(X_{k + 1} = j\, |\, X_k = i) = P(X_{k + 1} = j).\qedhere
    \end{align*}
\end{proof}
We further define
\begin{definition}[Stationary distribution]\label{def:statdistr}
Given a Markov chain through the transition matrix $\mat{P} \in \R^{n \times n}$, we call a probability distribution $p \in \R^n$ a \emph{stationary distribution} of the Markov chain, if, and only if
\[
    p\tran \mat{P} = p\tran.
\]
\end{definition}
As mentioned before in the introduction, computing transient probabilities of a Markov chain is elementary in many applications.
We also know that the naive way of calculating a $k$-step transient distribution boils down to $k$ vector-matrix multiplications.
So, if we use a Markov chain with $n$ states, the runtime is in $\mathcal{O}(n^2 k)$.
Thus, reducing the state space size is ideal to reduce runtime while still being able to compute the transient distributions.
The so-called exact state space aggregation of a Markov chain enables precisely this.
There, we find some linear function acting on a smaller-dimensional space, which behaves the same in some way.
This means we can transform the results back to the larger original space to get any transient distribution with no loss of information.

The definitions needed for aggregations, as presented here, are taken from~\cite[3, 9]{michel2025errbndmarkovaggr}.
\begin{definition}[State space aggregation of dimension $m$]\label{def:markov-chain-aggr}
Given $n, m \in \N$ with $m \leq n$ and a Markov chain as its transition matrix $\mat{P} \in \R^{n \times n}$ along with an initial distribution $p_0 \in \R$, we call any $\mat{\Pi} \in \R^{m \times m}$ the \emph{aggregated step matrix}, $\mat{A} \in \R^{m \times n}$ the \emph{disaggregation matrix} and arbitrary $\pi_0 \in \R^m$ the \emph{aggregated initial vector}.
\end{definition}
When referring to a state space aggregation, we will mainly just speak of an aggregation.
Even though $\mat{\Pi}$ and $\pi_0$ do not necessarily have to describe a Markov chain or distribution, respectively, we will still treat them like one, as in Proposition~\ref{prop:ksteptransdistr}.
Furthermore, we still say that the state space size of the aggregation is $m$, although technically, we do not explicitly have discrete states.
\begin{definition}[Aggregated transient distribution]\label{def:aggr-trans-distr}
Given an aggregation of a Markov chain and $k \in \N \setminus \set{0}$, we call
\[
    \pi_k\tran \coloneqq \pi_0\tran \mat{\Pi}^k
\]
the \emph{$k$-step aggregated transient distribution}.
\end{definition}
\begin{definition}[Approxmiated transient distribution]\label{def:approx-trans-distr}
Once more, given an aggregation and $k \in \N$,
\[
    \tilde{p}_k\tran \coloneqq \pi_k\tran \mat{A} \in \R^n
\]
denotes the $k$-step \emph{approximated transient distribution}.
\end{definition}
\begin{definition}[Aggregated and approximated stationary distribution]\label{def:aggrstatdistr}
Given an aggregation of a Markov chain, we call $\pi \in \R^m$ with
\[
    \pi\tran \mat{\Pi} = \pi\tran
\]
an \emph{aggregated stationary distribution}, if the \emph{approximated stationary distribution}
\[
    \tilde{p}\tran \coloneqq \pi\tran \mat{A} \in \R^n
\]
has that $\norm{\tilde{p}}_1 = 1$ holds true.
\end{definition}
Importantly, such a $\pi$ must not always exist.
In such cases, it can be useful to merely desire something of the form of $\pi\tran \mat{\Pi} \approx \pi\tran$.
As the notation already suggests, $\tilde{p}_k$ and $\tilde{p}$ should approximate $p_k$ and $p$ respectively.
To further formalize this idea, we now look at ways to restrict an aggregation such that $p_k = \tilde{p}_k$ holds for at least some or maybe even all $k$.
\begin{definition}[(Dynamic-) Exact aggregation]\label{def:dyn-exact-aggr}
The aggregation of a Markov chain is \emph{dynamic-exact} if, and only if,
\[
    \mat{\Pi} \mat{A} = \mat{A} \mat{P}.
\]
If further $\pi_0\tran \mat{A} = p_0$ holds, then the aggregation is \emph{exact}.
\end{definition}
This may initially seem somewhat disconnected from our initial idea of describing the quality of an aggregation.
Nevertheless, closer inspection shows a connection between these two things.
\begin{proposition}\label{prop:exact-aggr-implies-k-exact}
For an aggregation with aggregated stationary distribution, it holds that
\[
    \mat{\Pi} \mat{A} = \mat{A} \mat{P}\ \text{and}\ \pi_0\tran \mat{A} = p_0 \implies \forall k \in \N : \pi_0\tran \mat{\Pi}^k \mat{A} = p_0\tran \mat{P}^k\ \text{and}\ \tilde{p}\tran \mat{P} =\tilde{p}.
\]
\end{proposition}
\begin{proof}
    We proceed by induction on $k$.
    Set $k = 0$ for the base case, yielding
    \[
        \pi_0\tran \mat{\Pi}^0 \mat{A} = \pi_0\tran \mat{A} = p_0,
    \]
    enabling us to move on to the inductive step.
    For the \ac{IH}, assume that we have $\pi_0\tran \mat{\Pi}^k \mat{A} = p_0\tran \mat{P}^k$ for some $k \in \N$.
    Then, with
    \[
        \pi_0\tran \mat{\Pi}^{k + 1} \mat{A} = \pi_0\tran \mat{\Pi}^k \mat{\Pi} \mat{A} = \pi_0\tran \mat{\Pi}^k \mat{A} \mat{P} \txteq{\ac{IH}} p_0\tran \mat{P}^k \mat{P} = p_0\tran \mat{P}^{k + 1},
    \]
    we finish the proof for all transient distributions.
    Note that the assumption of there being a $\pi$ was not needed yet.
    We only need it to show
    \[\pi\tran \mat{\Pi} \mat{A} = \pi\tran \mat{A} \mat{P} \iff \pi\tran \mat{A} = \tilde{p}\tran \mat{P} \iff \tilde{p}\tran = \tilde{p}\tran \mat{P}.\qedhere\]
\end{proof}
This proposition is relevant because it shows us that for any exact aggregation, $p_k = \tilde{p}_k$, which is just what we want from an ideal aggregation.
Of course, it may also be worth looking at aggregations where this holds for only some $k$, as those are probably easier to find.
Furthermore, it would be ideal if Proposition~\ref{prop:exact-aggr-implies-k-exact} were an equivalence, as this would show these concepts to be one and the same.
Unfortunately, the best we can get is the weaker:
\begin{proposition}\label{prop:k-exact-aggr-almost-implies-exact}
If for an aggregation, $\pi_0\tran \mat{\Pi}^k \mat{A} = p_0\tran \mat{P}^k$ holds for all $k$, and if further there are pairwise distinct indices $i_0, \dots, i_{m - 1} \in \N$ with $\pi_{i_0}, \dots, \pi_{i_{m - 1}}$ linearly independent, the aggregation is exact.
\end{proposition}
\begin{proof}
    Through the calculation
    \begin{align*}
        p_0\tran \mat{P}^k &= \pi_0\tran \mat{\Pi}^k \mat{A} = \pi_0\tran \mat{\Pi}^k \mat{A} - \pi_0\tran \mat{\Pi}^{k - 1} \mat{A} \mat{P} + \pi_0\tran \mat{\Pi}^{k - 1} \mat{A} \mat{P}\\
        &= \pi_0\tran \mat{\Pi}^{k - 1} (\mat{\Pi} \mat{A} - \mat{A} \mat{P}) + p_0\tran \mat{P}^k
    \end{align*}
    we see that $\pi_0\tran \mat{\Pi}^{k - 1} (\mat{\Pi} \mat{A} - \mat{A} \mat{P}) = (0, \dots, 0)$.
    Now we observe the resulting linear combination
    \begin{equation}
        \pi_0\tran \mat{\Pi}^{k - 1} (\mat{\Pi} \mat{A} - \mat{A} \mat{P}) \txteq{Def.~\ref{def:aggr-trans-distr}} \sum_{i = 1}^m \pi_{k - 1}(i) (\mat{\Pi} \mat{A} - \mat{A} \mat{P})(i) = (0, \dots, 0),\label{eq:lincombimpliesexact}
    \end{equation}
    where $(\mat{\Pi} \mat{A} - \mat{A} \mat{P})(i)$ is the $i$-th row of $\mat{\Pi} \mat{A} - \mat{A} \mat{P}$.
    Without loss of generality, we may assume that $i_0 < \dots < i_{m - 1}$.
    Equation~\eqref{eq:lincombimpliesexact} holds especially for $i_0 + 1 \leq k \leq i_{m - 1} + 1$ too.
    As $\pi_{i_0}, \dots \pi_{i_{m - 1}}$ are linearly independent,
    \[
        \spn\set{\pi_{i_0}\tran(\mat{\Pi} \mat{A} - \mat{A} \mat{P}), \dots, \pi_{i_{m - 1}}\tran(\mat{\Pi} \mat{A} - \mat{A} \mat{P})} = \rowsp(\mat{\Pi} \mat{A} - \mat{A} \mat{P})
    \]
    But we know each vector in this span to be the zero-vector, thus $\rowsp(\mat{\Pi} \mat{A} - \mat{A} \mat{P}) = \set{\mathbf{0}}$, which implies that $\mat{\Pi} \mat{A} = \mat{A} \mat{P}$.
    Lastly, the second condition for exactness follows immediately from the special case of $k = 0$.
\end{proof}
Note that the added condition of there being pairwise distinct indices $i_0, \dots, i_{m - 1} \in \N$ with $\pi_{i_0}, \dots, \pi_{i_{m - 1}}$ linearly independent in Proposition~\ref{prop:k-exact-aggr-almost-implies-exact} cannot be relaxed.
If there were fewer such vectors $\pi_{i_j}$, they could no longer span $\rowsp(\mat{\Pi} \mat{A} - \mat{A} \mat{P})$.

\section{Error bounding approximated transient distributions}\label{sec:error-bounds-on-approximated-transient-distributions}
As we have seen in Propositions~\ref{prop:exact-aggr-implies-k-exact} and~\ref{prop:k-exact-aggr-almost-implies-exact}, there is a close connection between the exactness of an aggregation and whether $\pi_0 \mat{\Pi}^k \mat{A} = p_0 \mat{P}^k$ holds.
So further examining their connection is of great interest.
Following~\cite[2, 4--5]{michel2025errbndmarkovaggr} we will see that we can indeed make such a connection.

We will start by introducing some preliminaries from linear algebra, which are necessary to understand the error bounds introduced later.
Before that, note that for a vector $v$ or matrix $\mat{M}$, the notation $\abs{v}$ or $\abs{\mat{M}}$ denotes the vector or matrix where each entry is the absolute value of the corresponding entry in $v$ or $\mat{M}$.
Further, $\mathbf{1}_n \in \R^n$ contains only ones.
\begin{definition}[Maximum absolute row sum norm]\label{def:max-row-sum-norm}
As defined in~\cite[222]{beilina2017numlinalg}, let
\[\norm{\,\cdot\,}_\infty : \R^{n \times m} \to \R,\ \mat{M} \mapsto \max_{1 \leq i \leq n} \sum_{j = 1}^m \abs*{\mat{M}(i, j)}\]
be the so-called \emph{maximum absolute row sum norm} of $\mat{M} \in \R^{n \times m}$.
\end{definition}
\begin{proposition}\label{prop:inf-norm-is-norm}
The function $\norm{\,\cdot\,}_\infty$ from Definition~\ref{def:max-row-sum-norm} presents a norm on $\R^{n \times m}$, which is also submultiplicative, meaning $\norm{\mat{M}_1 \mat{M}_2}_\infty \leq \norm{\mat{M}_1}_\infty \cdot \norm{\mat{M}_2}_\infty$ for $\mat{M}_1 \in \R^{n \times m}$ and $\mat{M}_2 \in \R^{m \times k}$.
\end{proposition}
\begin{proof}
    We proceed in this proof, by following the definition of a norm for arbitrary matrices $\mat{M}_{1}, \mat{M}_3 \in \R^{n \times m}$ and $\mat{M}_2 \in \R^{m \times k}$.
    \begin{itemize}
        \item Positive definiteness:
        This follows from the positive definiteness of the absolute value used in the summation within $\norm{\,\cdot\,}_\infty$.
        \item Multiplicativity:
        This is concluded through the multiplicativity of the $\max$ function, summation, and that of $\abs{\,\cdot\,}$.
        \item Triangle inequality: Through calculation, we can also show this condition to be true:
        \begin{align*}
            \norm*{\mat{M}_1 + \mat{M}_3}_\infty &\txteq{Def.~\ref{def:max-row-sum-norm}} \max_{1 \leq i \leq n} \sum_{j = 1}^m \abs*{\mat{M}_1(i,j) + \mat{M}_3(i, j)}\\
            &\leq \max_{1 \leq i \leq n} \sum_{j = 1}^m \abs*{\mat{M}_1(i, j)} + \sum_{j = 1}^m \abs*{\mat{M}_3(i, j)}\\
            &\leq \max_{1 \leq i \leq n} \sum_{j = 1}^m \abs*{\mat{M}_1(i, j)} + \max_{1 \leq i \leq n} \sum_{j = 1}^m \abs*{\mat{M}_3(i, j)} \txteq{Def.~\ref{def:max-row-sum-norm}} \norm*{\mat{M}_1}_\infty + \norm*{\mat{M}_3}_\infty.
        \end{align*}
        \item Submultiplicativity: Again, a calculation suffices to finish the proof:
        \begin{align*}
            \norm{\mat{M}_1 \mat{M}_2}_\infty &\txteq{Def.~\ref{def:max-row-sum-norm}} \max_{1 \leq i \leq n} \sum_{j = 1}^k \abs*{\mat{M}_1\mat{M}_2(i, j)} = \max_{1 \leq i \leq n} \sum_{j = 1}^k \abs*{\sum_{\ell = 1}^m \mat{M}_1(i, \ell) \mat{M}_2(\ell, j)}\\
            &\leq \max_{1 \leq i \leq n} \sum_{j = 1}^k \sum_{\ell = 1}^m \abs*{\mat{M}_1(i, \ell)} \abs*{\mat{M}_2(\ell, j)} = \max_{1 \leq i \leq n} \sum_{\ell = 1}^m \abs*{\mat{M}_1(i, \ell)} \sum_{j = 1}^k \abs*{\mat{M}_2(\ell, j)}\\
            &\leq \left( \max_{1 \leq i \leq n} \sum_{\ell = 1}^m \abs*{\mat{M}_1(i, \ell)} \right) \left( \max_{1 \leq \ell \leq m} \sum_{j = 1}^k \abs*{\mat{M}_2(\ell, j)}  \right)\\
            &\txteq{Def.~\ref{def:max-row-sum-norm}} \norm{\mat{M}_1}_\infty \cdot \norm{\mat{M}_2}_\infty.\qedhere
        \end{align*}
    \end{itemize}
\end{proof}
Finishing the additional linear algebra, we get from~\cite[2]{michel2025errbndmarkovaggr}
\begin{proposition}\label{prop:1-norm-inequ}
With $v \in \R^n$ and $\mat{M} \in \R^{n \times m}$ it holds true that
\[
    \norm*{v\tran \mat{M}}_1 \leq
    \begin{cases}
        \inp*{\abs*{v}, \abs{\mat{M}} \cdot \mathbf{1}_m}\\
        \norm{v}_1 \cdot \norm{\mat{M}}_\infty
    \end{cases}
\]
\end{proposition}
\begin{proof}
    We can simply proceed with calculations, achieving
    \begin{align*}
        \norm*{v\tran \mat{M}}_1 &= \sum_{i = 1}^m \abs*{v \mat{M}(i)} = \sum_{i = 1}^m \abs*{\sum_{j = 1}^n v(j) \mat{M}(j, i)} \leq \sum_{i = 1}^m \sum_{j = 1}^n \abs*{v(j)} \abs*{\mat{M}(j, i)} = \sum_{j = 1}^n \abs*{v(j)} \sum_{i = 1}^m \abs*{\mat{M}(j, i)}\\
        &\begin{cases}
             = \inp*{\abs{v}, \abs{\mat{M}} \cdot \mathbf{1}_m},\\
             \leq \left( \sum_{j = 1}^n \abs*{v(j)} \right) \left( \max_{1 \leq j \leq n} \sum_{i = 1}^m \abs*{\mat{M}(j, i)} \right) \txteq{Def.~\ref{def:max-row-sum-norm}} \norm{v}_1 \cdot \norm{\mat{M}}_\infty.
        \end{cases}\qedhere
    \end{align*}
\end{proof}
Now, equipped with a matrix norm and Proposition~\ref{prop:1-norm-inequ}, we can start looking at how we can define and then bound the error made by an aggregation.
\begin{definition}[Error (vector) after $k$ steps]\label{def:error-k-vec}
Given an arbitrary aggregation, we call $e_k$ the \emph{error vector after $k$ steps}, where
\[
    e_k\tran \coloneqq \tilde{p}_k\tran - p_k\tran \txteq{Def.~\ref{def:approx-trans-distr}, Prop.~\ref{prop:ksteptransdistr}} \pi_0\tran \mat{\Pi}^k \mat{A} - p_0 \mat{P}^k.
\]
Then, $\norm{e_k}_1$ denotes the \emph{error after $k$ steps}.
\end{definition}
Now, we will start by showing a subsidiary proposition analogous to the start of the proof of~\cite[Thrm. 4]{michel2025errbndmarkovaggr}.
\begin{proposition}\label{prop:bounds-support}
In any arbitrary aggregation we have that
\[
    \norm{e_{k + 1}}_1 \leq \norm{e_k}_1 +
    \begin{cases}
        \inp*{\abs{\pi_k}, \abs{\mat{\Pi} \mat{A} - \mat{A} \mat{P}} \cdot \mathbf{1}_n},\\
        \norm{\pi_k}_1 \cdot \norm*{\mat{\Pi} \mat{A} - \mat{A} \mat{P}}_\infty.
    \end{cases}
\]
\end{proposition}
\begin{proof}
    This can be shown by computing
    \begin{align*}
        \norm*{e_{k + 1}}_1 &\txteq{Def.~\ref{def:error-k-vec}, Prop.~\ref{prop:ksteptransdistr}} \norm*{\pi_0\tran \mat{\Pi}^{k + 1} \mat{A} - p_0\tran \mat{P}^{k + 1}}_1\\
        &= \norm*{\pi_0\tran \mat{\Pi}^k \mat{A} \mat{P} - p_0\tran \mat{P}^{k + 1} + \pi_0\tran \mat{\Pi}^{k + 1} \mat{A} - \pi_0\tran \mat{\Pi}^k \mat{A} \mat{P}}_1\\
        &= \norm*{ \left( \pi_0\tran \mat{\Pi}^k \mat{A} - p_0\tran \mat{P}^k \right)  \mat{P} + \pi_0\tran \mat{\Pi}^k \left( \mat{\Pi} \mat{A} - \mat{A} \mat{P} \right)}_1\\
        &\txteq{Def.~\ref{def:aggr-trans-distr},~\ref{def:error-k-vec}} \norm*{e_k\tran  \mat{P} + \pi_k\tran \left( \mat{\Pi} \mat{A} - \mat{A} \mat{P} \right)}_1 \leq \norm*{e_k\tran  \mat{P}}_1 + \norm*{\pi_k\tran \left( \mat{\Pi} \mat{A} - \mat{A} \mat{P} \right)}_1\\
        &\txtleq{Prop.~\ref{prop:1-norm-inequ}} \norm*{e_k}_1 \cdot \underbrace{\norm*{\mat{P}}_\infty}_{= 1} + \norm*{\pi_k (\mat{\Pi} \mat{A} - \mat{A} \mat{P})}_1 \txteq{Def.~\ref{def:trns-mtrx}} \norm*{e_k}_1 + \norm*{\pi_k (\mat{\Pi} \mat{A} - \mat{A} \mat{P})}_1\\
        &\txtleq{Prop.~\ref{prop:1-norm-inequ}} \norm*{e_k}_1 +
        \begin{cases}
            \inp*{\abs{\pi_k}, \abs{\mat{\Pi} \mat{A} - \mat{A} \mat{P}} \cdot \mathbf{1}_n},\\
            \norm{\pi_k}_1 \cdot \norm*{\mat{\Pi} \mat{A} - \mat{A} \mat{P}}_\infty.
        \end{cases}\qedhere
    \end{align*}
\end{proof}
Having concluded this auxiliary step, we go on as in~\cite[4--5]{michel2025errbndmarkovaggr}, to show
\begin{proposition}\label{prop:specific-error-bound}
In any arbitrary aggregation we can bound the error after $k$ steps through
\[
    \norm{e_k}_1 \leq \norm{e_0}_1 + \sum_{j = 0}^{k - 1} \inp*{\abs{\pi_j}, \abs{\mat{\Pi} \mat{A} - \mat{A} \mat{P}} \cdot \mathbf{1}_n}.
\]
\end{proposition}
\begin{proof}
    This is a standard case of infinite induction on $k$.
    The starting case of $k = 0$ is clear.
    As the \ac{IH} we may assume that
    \[
        \norm{e_k}_1 \leq \norm{e_0}_1 + \sum_{j = 0}^{k - 1} \inp*{\abs{\pi_j}, \abs{\mat{\Pi} \mat{A} - \mat{A} \mat{P}} \cdot \mathbf{1}_n}
    \]
    holds for some $k$.
    We then get in the inductive step
    \begin{align*}
        \norm{e_{k + 1}}_1 &\txtleq{Prop~\ref{prop:bounds-support}} \norm*{e_k}_1 + \inp*{\abs{\pi_k}, \abs{\mat{\Pi} \mat{A} - \mat{A} \mat{P}} \cdot \mathbf{1}_n}\\
        &\txtleq{\ac{IH}} \norm{e_0}_1 + \sum_{j = 0}^{k - 1} \inp*{\abs{\pi_j}, \abs{\mat{\Pi} \mat{A} - \mat{A} \mat{P}} \cdot \mathbf{1}_n} + \inp*{\abs{\pi_k}, \abs{\mat{\Pi} \mat{A} - \mat{A} \mat{P}} \cdot \mathbf{1}_n}\\
        &= \norm{e_0}_1 + \sum_{j = 0}^k \inp*{\abs{\pi_j}, \abs{\mat{\Pi} \mat{A} - \mat{A} \mat{P}} \cdot \mathbf{1}_n}.\qedhere
    \end{align*}
\end{proof}
Again, by using Proposition~\ref{prop:bounds-support}, we can also show a less precise bound:
\begin{proposition}\label{prop:general-error-bound}
In any arbitrary aggregation we can bound the error after $k$ steps through
\[
    \norm*{e_k}_1 \leq \norm*{e_0}_1 + \norm*{\pi_0}_1 \cdot \norm*{\mat{\Pi} \mat{A} - \mat{A} \mat{P}}_\infty \cdot
    \begin{cases}
        \frac{\norm{\mat{\Pi}}_\infty^k - 1}{\norm{\mat{\Pi}}_\infty - 1}, & \text{if $\norm{\mat{\Pi}}_\infty \neq 1$},\\
        k, & \text{else.}
    \end{cases}
\]
\end{proposition}
\begin{proof}
    For a formal proof, we proceed by induction on $k \in \N$.
    The starting case of $k = 0$ is clear, as we have
    \[
        \norm*{e_0}_1 \leq \norm*{e_0}_1 + 0 = \norm*{e_0}_1 + \norm*{\pi_0}_1 \cdot \norm*{\mat{\Pi} \mat{A} - \mat{A} \mat{P}}_\infty \cdot
        \begin{cases}
            \frac{\norm{\mat{\Pi}}_\infty^0 - 1}{\norm{\mat{\Pi}}_\infty - 1}, & \text{if $\norm{\mat{\Pi}}_\infty \neq 1$},\\
            0, & \text{else.}
        \end{cases}
    \]
    Then, assume for the \ac{IH} with fixed $k \in \N$, that we have
    \[
        \norm*{e_k}_1 \leq \norm*{e_0}_1 + \norm*{\pi_0}_1 \cdot \norm*{\mat{\Pi} \mat{A} - \mat{A} \mat{P}}_\infty \cdot
        \begin{cases}
            \frac{\norm{\mat{\Pi}}_\infty^k - 1}{\norm{\mat{\Pi}}_\infty - 1}, & \text{if $\norm{\mat{\Pi}}_\infty \neq 1$},\\
            k, & \text{else.}
        \end{cases}
    \]
    The following inductive step is concluded by calculating
    \begin{align*}
        \norm*{e_{k + 1}}_1 &\txtleq{Prop.~\ref{prop:bounds-support}} \norm*{e_k}_1 + \norm*{\pi_k}_1 \cdot \norm*{\mat{\Pi} \mat{A} - \mat{A} \mat{P}}_\infty\\
        &\txtleq{Def.~\ref{def:aggr-trans-distr}, Prop.~\ref{prop:inf-norm-is-norm},~\ref{prop:1-norm-inequ}} \norm*{e_k}_1 + \norm*{\pi_0}_1 \cdot \norm*{\mat{\Pi}}_\infty^k \cdot \norm*{\mat{\Pi} \mat{A} - \mat{A} \mat{P}}_\infty\\
        &\txtleq{\ac{IH}} \norm*{e_0}_1 + \norm*{\pi_0}_1 \cdot \norm*{\mat{\Pi} \mat{A} - \mat{A} \mat{P}}_\infty \cdot
        \begin{cases}
            \frac{\norm{\mat{\Pi}}_\infty^k - 1}{\norm{\mat{\Pi}}_\infty - 1} + \norm*{\mat{\Pi}}_\infty^k, & \text{if $\norm{\mat{\Pi}}_\infty \neq 1$},\\
            k + \norm*{\mat{\Pi}}_\infty^k, & \text{else,}
        \end{cases}\\
        &= \norm*{e_0}_1 + \norm*{\pi_0}_1 \cdot \norm*{\mat{\Pi} \mat{A} - \mat{A} \mat{P}}_\infty \cdot
        \begin{cases}
            \frac{\norm{\mat{\Pi}}_\infty^k - 1}{\norm{\mat{\Pi}}_\infty - 1} + \frac{\norm*{\mat{\Pi}}_\infty^k \left( \norm{\mat{\Pi}}_\infty - 1 \right)}{\norm{\mat{\Pi}}_\infty - 1}, & \text{if $\norm{\mat{\Pi}}_\infty \neq 1$},\\
            k + 1, & \text{else,}
        \end{cases}\\
        &= \norm*{e_0}_1 + \norm*{\pi_0}_1 \cdot \norm*{\mat{\Pi} \mat{A} - \mat{A} \mat{P}}_\infty \cdot
        \begin{cases}
            \frac{\norm{\mat{\Pi}}_\infty^{k + 1} - 1}{\norm{\mat{\Pi}}_\infty - 1}, & \text{if $\norm{\mat{\Pi}}_\infty \neq 1$},\\
            k + 1, & \text{else.}
        \end{cases}\qedhere
    \end{align*}
\end{proof}

\section{Arnoldi iteration}\label{sec:arnoldi-iteration}
If we now want an aggregation of a Markov chain where $p_k$ and $\tilde{p}_k$ match, we can see that $p_k\tran$ has to be in the vector space spanned by the rows of the disaggregation matrix $\mat{A}$ for all $k \in \N$.
Choosing elements of $\spn\set{p_0\tran, p_1\tran, \dots}$ as the rows of $\mat{A}$ seems practical.
The so-called Arnoldi iteration fulfils this criterion while providing further amenities, which we will see.
With it, we can find a basis of $\spn\set{p_0\tran, p_1\tran, \dots}$, for which we further need the Gram-Schmidt procedure as it enables us to compute a basis where all elements are pairwise orthogonal with length one.
Importantly, such bases are numerically well-behaved.

The definitions in this section regarding the two Gram-Schmidt procedures and their preliminaries are taken from~\cite[75, 324]{beilina2017numlinalg}.
\begin{definition}[Orthonormal] \label{def:orthonormal}
We call a set of vectors $\set{v_1, \dots, v_n}$ in a vector space $V$ with some inner product $\inp{\cdot, \cdot}$ \emph{orthonormal} if, and only if,
\[
    \forall j, k \in \set{1, \dots, n} : \inp{v_j, v_k} =
    \begin{cases}
        1, & \text{if}\ j = k,\\
        0, & \text{if}\ j \neq k.
    \end{cases}
\]
\end{definition}
This concept allows us now to make sense of the first two algorithms, two variants of the Gram-Schmidt procedure.
We will see that they are equivalent, meaning, in theory, the same input yields the same output.
Still, in practical applications, it makes sense to differentiate between these two, as we will see later on.
\begin{algorithm}[H]
    \caption{Classical Gram-Schmidt procedure}\label{alg:classical-gs}
    \begin{algorithmic}[1]
        \State Let $n \in \N \setminus \set{0}$ and $v_1, \dots, v_n \in \R^m$ linearly independent
        \ForAll{$k = 1, \dots, n$}
            \State $q_k \coloneqq v_k$
            \ForAll{$j = 1, \dots, k - 1$}
                \State $r_{j,k} \coloneqq \inp*{q_j, v_k}$
                \State $q_k = q_k - r_{j,k} q_j$
            \EndFor
            \State $r_{k,k} \coloneqq \norm{q_k}_2$
            \State $q_k = \frac{q_k}{r_{k,k}}$
        \EndFor
        \State \Return $q_1, \dots q_n$
    \end{algorithmic}
\end{algorithm}
\begin{proposition}\label{prop:closed-form-cgs}
In Algorithm~\ref{alg:classical-gs}, after returning, we have for $i \in \N$, $1 \leq i \leq n$ that
\[
    q_i = \frac{v_i - \sum_{j = 1}^{i - 1} \inp{q_j, v_i} q_j}{\norm{v_i - \sum_{j = 1}^{i - 1} \inp{q_j, v_i} q_j}_2}.
\]
\end{proposition}
\begin{proof}
    We can see this by unfolding the pseudocode inside the outer for-loop.
    Initially, we set $q_i = v_i$ to then subtract some term from it, before normalizing it at the end of the for-loop.
    Further, any specific $q_i$ is only ever modified in the singular pass of the outer for-loop with $k = i$.
    So we need to check now, whether this term subtracted from $v_i$ is of the correct form.
    As addition is associative, we can at first sum up the $r_{j,k}$ before subtracting this sum from $v_i$.
    So we get $q_i = v_i - \sum_{j = 1}^{i - 1} \inp{q_j, v_i} q_j$ after the inner for-loop has passed.
    The remaining lines 7 and 8 normalize the above described $q_i$, resulting in the initially proposed value for $q_i$.
\end{proof}
\begin{proposition}\label{prop:CGS-orthonorm-basis}
For a linearly independent input $v_1, \dots, v_n \in \R^m$, the \ac{CGS} returns an output $q_1, \dots, q_n$ where with $i \in \N$, $1 \leq i \leq n$ the vectors $q_1, \dots, q_i$ form an orthonormal basis of $\spn\set{v_1, \dots, v_i}$.
\end{proposition}
\begin{proof}
    We will be using the equivalent closed-form formula from Proposition~\ref{prop:closed-form-cgs} to proceed by finite induction on $i$.
    In the starting case of $i = 1$, we have $q_1 = \frac{v_1}{\norm{v_1}_2}$.
    Thus, we have $\spn\set{q_1} = \spn\set{v_1}$ holds, and $q_1$ alone is also orthonormal.
    For the \ac{IH}, assume that $q_1, \dots, q_i$ is orthonormal and $\spn\set{q_1, \dots, q_i} = \spn\set{v_1, \dots, v_i}$ with $i \leq n - 1$.
    As $v_1, \dots, v_n$ are linearly independent, we can conclude that $v_1, \dots, v_i$ are also.
    So we get
    \[
        v_{i + 1} \notin \spn\set{v_1, \dots v_i} = \spn\set{q_1, \dots q_i},
    \]
    resulting in
    \[
        v_{i + 1} - \underbrace{\sum_{j = 1}^i \inp{v_{i + 1}, q_j} q_j}_{\in \spn\set{q_1, \dots q_i}} \neq 0.
    \]
    Thus, $q_{i + 1}$ is well-defined and
    \[
        \spn\set{v_1, \dots v_i, v_{i + 1}} = \spn\set{q_1, \dots q_i, q_{i + 1}},
    \]
    concluding the inductive step here.
    The inductive step for orthonormality is still necessary.
    As $\norm{q_{i + 1}}_2^2 = \inp{q_{i + 1}, q_{i + 1}} \txteq{Def.~\ref{def:orthonormal}} 1$, we look at orthogonality.
    Choose $\ell \in \set{1, \dots, i}$:
    \begin{align*}
        \inp{q_{i + 1}, q_\ell} &= \inp*{\frac{v_{i + 1} - \sum_{j = 1}^i \inp{v_{i + 1}, q_j} q_j}{\norm{v_{i + 1} - \sum_{j = 1}^i \inp{v_{i + 1}, q_j} q_j}_2}, q_\ell}\\
        &= \frac{1}{\norm{v_{i + 1} - \sum_{j = 1}^i \inp{v_{i + 1}, q_j} q_j}_2} \inp*{v_{i + 1} - \sum_{j = 1}^i \inp{v_{i + 1}, q_j} q_j, q_\ell}\\
        &= \frac{1}{\norm{v_{i + 1} - \sum_{j = 1}^i \inp{v_{i + 1}, q_j} q_j}_2} \left( \inp{v_{i + 1}, q_\ell} - \sum_{j = 1}^i \inp{v_{i + 1}, q_j} \inp{q_j, q_\ell} \right)\\
        &\txteq{\ac{IH}, Def.~\ref{def:orthonormal}} \frac{1}{\norm{v_{i + 1} - \sum_{j = 1}^i \inp{v_{i + 1}, q_j} q_j}_2} \left( \inp{v_{i + 1}, q_\ell} - \inp{v_{i + 1}, q_\ell} \right) = 0.\qedhere
    \end{align*}
\end{proof}
Now, by changing the order of calculation in the classical Gram-Schmidt procedure, we get a very similar algorithm, known as the \ac{MGS}:
\begin{algorithm}[H]
    \caption{Modified Gram-Schmidt procedure}\label{alg:modified-gs}
    \begin{algorithmic}[1]
        \State Let $n \in \N \setminus \set{0}$ and $v_1, \dots, v_n \in \R^m$ linearly independent
        \ForAll{$k = 1, \dots, n$}
            \State $r_{k,k} \coloneqq \norm{v_k}_2$
            \State $q_k \coloneqq \frac{v_k}{r_{k,k}}$
            \ForAll{$j = k + 1, \dots, n$}
                \State $r_{k,j} \coloneqq \inp*{q_k, v_j}$
                \State $v_j = v_j - r_{k,j} q_k$
            \EndFor
        \EndFor
        \State \Return $q_1, \dots, q_n$
    \end{algorithmic}
\end{algorithm}
\begin{proposition}\label{prop:cgs-equiv-mgs}
\Ac{CGS} and \ac{MGS} return the exact same output for the same linearly independent input $v_1, \dots, v_n \in \R^m$.
\end{proposition}
\begin{proof}
    Denote the new value of $v_j$ after the $k$-th iteration of the outer for loop with $v_j^{(k)}$.
    We will proceed by induction on the in- and output size $n \in \N$.
    The case $n = 1$ is clear, as in both Algorithm~\ref{alg:classical-gs} and~\ref{alg:modified-gs}, we just have that $q_1 = \frac{v_1}{\norm{v_1}_2}$.
    For the \ac{IH}, assume that for $n \in \N$ being the input size of both Algorithm~\ref{alg:classical-gs} and~\ref{alg:modified-gs}, their respective outputs are the same.
    Thus, we can apply Proposition~\ref{prop:closed-form-cgs} and~\ref{prop:CGS-orthonorm-basis} to the size $n$ output of \ac{MGS}.

    Now as an auxiliary result, using the above \ac{IH}, we will show for $i \in \N$, $1 \leq i \leq n + 1$ and $k \in \N$, $1 \leq k \leq i - 1$ that
    \begin{equation}
        v_i^{(k)} = v_i - \sum_{\ell = 1}^k \inp*{q_\ell, v_i} q_\ell.\label{eq:partial-gs}
    \end{equation}
    The base case of $k = 1$ is done with $v_i^{(1)} = v_i - \inp*{q_1, v_i} q_1$.
    As the \ac{IH} of this finite induction, set $v_i^{(k)} = v_i - \sum_{\ell = 1}^k \inp*{q_\ell, v_i} q_\ell$ for $k \in \N$, $1 \leq k \leq i - 2$.
    Through line 6 and 7 of \ac{MGS} we get
    \begin{align*}
        v_i^{(k + 1)} &= v_i^{(k)} - q_{k + 1}\tran v_i^{(k)} q_{k + 1}\\
        &\txteq{Aux. \ac{IH}} v_i - \sum_{\ell = 1}^k \inp*{q_\ell, v_{i + 1}} q_\ell - q_{k + 1}\tran \left( v_i - \sum_{\ell = 1}^k \inp*{q_\ell, v_i} q_\ell \right) q_{k + 1}\\
        &= v_i - \sum_{\ell = 1}^k \inp*{q_\ell, v_i} q_\ell - \inp*{q_{k + 1}, v_i} q_{k + 1} - \sum_{\ell = 1}^k q_{k + 1}\tran \inp*{q_\ell, v_i} q_\ell q_{k + 1}\\
        &= v_i - \sum_{\ell = 1}^k \inp*{q_\ell, v_i} q_\ell - \inp*{q_{k + 1}, v_i} q_{k + 1} - \sum_{\ell = 1}^k \inp*{q_\ell, v_i} \underbrace{\inp*{q_{k + 1}, q_\ell}}_{= 0} q_{k + 1}\\
        &\txteq{Main \ac{IH}, Prop.~\ref{prop:CGS-orthonorm-basis}} v_i - \sum_{\ell = 1}^{k + 1} \inp*{q_\ell, v_i} q_\ell,
    \end{align*}
    finishing up this auxiliary result.

    Now we can finally determine $q_{n + 1}$.
    As this is fully determined by $v_{n + 1}^{(n)}$, we use the auxiliary result~(\ref{eq:partial-gs}), giving us
    \[
        v_{n + 1}^{(n)} = v_{n + 1} - \sum_{\ell = 1}^n \inp*{q_\ell, v_{n + 1}} q_\ell.
    \]
    Combining this with
    \[
        q_{n + 1} = \frac{v_{n + 1}^{(n)}}{\norm{v_{n + 1}^{(n)}}_2}
    \]
    from lines 3 and 4, concludes this proof.
\end{proof}
\begin{corollary}\label{cor:MGS-orthonorm-basis}
For a linearly independent input $v_1, \dots, v_n \in \R^m$, \Ac{MGS} returns $q_1, \dots, q_n$ where the vectors $q_1, \dots, q_i$ form an orthonormal basis of $\spn\set{v_1, \dots, v_i}$.
\end{corollary}
\begin{proof}
    This proof follows by applying Proposition~\ref{prop:cgs-equiv-mgs} to Proposition~\ref{prop:CGS-orthonorm-basis}.
\end{proof}
For the upcoming part about the Arnoldi iteration, the preliminary definition and the definition of the Arnoldi iteration are taken from~\cite[164,173--174]{arbenz2016lecturenotes}.
Again, we only made minor changes, such as renaming and transposing, for notational ease later on.
\begin{definition}[Krylov subspace]\label{def:krylov-subspace}
Let $\F$ be a field with $x \in \F^n$ and $\mat{M} \in \F^{n \times n}$.
We call
\[\mathcal{K}^m(x, \mat{M}) \coloneqq \spn\Set{x\tran, x\tran \mat{M}, x\tran \mat{M}^2, \dots, x\tran \mat{M}^{m - 1}}\]
a \emph{Krylov subspace} with respect to $x$ and $\mat{M}$.
\end{definition}
The Arnoldi iteration aims to compute an orthonormal basis of $\mathcal{K}^{k + 1}(x, \mat{M})$.
This is done by simply using the modified Gram-Schmidt procedure to iteratively go through $x\tran \mat{M}^j$ with increasing $j$ until we either have $k + 1$ basis elements and are done, or we terminate sooner if we have found an invariant subspace.
Along the way, we collect some supplementary terms we need to compute anyway, so they can be used later for an interesting relation relevant to our application of aggregating Markov chains exactly.
\begin{algorithm}[H]
    \caption{Arnoldi Iteration}\label{alg:arnoldi-iteration}
    \begin{algorithmic}[1]
        \State Let $m \in \N$ and $x \in \R^n$, $\mat{M} \in \R^{n \times n}$.
        \State $q_1 \coloneqq \frac{x}{\norm{x}_2}$
        \For{$j = 1, \dots, m$}
            \State $r_1 \coloneqq q_j\tran \mat{M}$
            \ForAll{$i = 1, \dots, j$}
                \State $h_{j, i} \coloneqq \inp{r_i\tran, q_i}$
                \State $r_{i + 1} \coloneqq r_i - h_{j, i} q_i\tran$
            \EndFor
            \State $h_{j, j + 1} \coloneqq \norm{r_{j + 1}}_2$
            \If{$h_{j, j + 1} = 0$}
                \Return $q_1, \dots, q_j,\ h_{1,1}, \dots, h_{j, j}$
            \EndIf
            \State $q_{j + 1} \coloneqq \frac{r_{j + 1}\tran}{h_{j, j + 1}}$
        \EndFor
        \State \Return $q_1, \dots, q_{m + 1},\ h_{1, 1}, \dots, h_{m, m + 1}$
    \end{algorithmic}
\end{algorithm}
\begin{proposition}\label{prop:arnoldi-output-orthonorm-krylov}
At line 8 in Algorithm~\ref{alg:arnoldi-iteration}, we have that $q_1\tran, \dots, q_j\tran$ is an orthonormal basis of $\mathcal{K}^j(x, \mat{M})$.
\end{proposition}
\begin{proof}
    We observe that Algorithm~\ref{alg:arnoldi-iteration} almost mirrors Algorithm~\ref{alg:modified-gs} with input vectors $x\tran, x\tran \mat{M}, \dots, x\tran \mat{M}^n$.
    Except for the possible early return in Algorithm~\ref{alg:arnoldi-iteration}, these two are effectively structurally the same, safe for the different ranges of the inner for-loops.
    But even this difference can be remedied.

    By using the exact same technique as in the proof of Proposition~\ref{prop:cgs-equiv-mgs} and~(\ref{eq:partial-gs}), one can show that at line 8 we have that
    \begin{equation}
        r_\ell = q_j\tran \mat{M} - \sum_{i = 1}^{\ell - 1} \inp*{\mat{M}\tran q_j, q_i} q_i\tran,\label{eq:r-j-gram-schmidt}
    \end{equation}
    mirroring~(\ref{eq:partial-gs}).
    Further, with $\ell = j$ and Proposition~\ref{prop:closed-form-cgs} one can show that the resulting vectors $q_1, \dots q_j$ are the result of a Gram-Schmidt-like procedure, meaning with Proposition~\ref{prop:CGS-orthonorm-basis}, that they are pairwise orthonormal.
    More precisely, we orthonormalize $q_j \mat{M}$ against $q_1, \dots q_j$.

    So we need to show $\spn\set{q_1, \dots q_j, q_j \mat{M}} = \mathcal{K}^{j + 1}(x, \mat{M})$.
    Again we do this by induction, starting with $j = 1$.
    By line two in Algorithm~\ref{alg:arnoldi-iteration}, we have
    \[
        \spn\Set{q_1\tran, q_1\tran \mat{M}} = \spn\Set{\frac{x\tran}{\norm{x}_2}, \frac{x\tran}{\norm{x}_2} \mat{M}} = \spn\Set{x\tran, x\tran \mat{M}} = \mathcal{K}^2(x, \mat{M}).
    \]
    Assume $\spn\set{q_1\tran, \dots, q_j\tran, q_j\tran \mat{M}} = \mathcal{K}^{j + 1}(x, \mat{M})$ holds with fixed $j \in \N \setminus \set{0}$ as the \ac{IH}.
    As such, the linear combination of $q_j\tran \mat{M}$, with $\lambda_\ell \in \R$ and $\lambda_{j + 1} \neq 0$ will be denoted as
    \[
        q_j\tran \mat{M} \eqqcolon \sum_{\ell = 1}^{j + 1} \lambda_\ell x\tran \mat{M}^{\ell - 1}.
    \]
    This allows us to see that
    \begin{align*}
        q_{j + 1}\tran \mat{M} &\txteq{Prop.~\ref{prop:closed-form-cgs}} \underbrace{\frac{1}{\norm{q_j\tran \mat{M} - \sum_{i = 1}^j \inp{\mat{M}\tran q_j, q_i} q_i\tran}_2}}_{\eqqcolon \mu \in \R} \left( q_j\tran \mat{M} - \sum_{i = 1}^j \inp{\mat{M}\tran q_j, q_i} q_i\tran \right) \mat{M}\\
        &= \mu (q_j\tran \mat{M}) \mat{M} - \mu \sum_{i = 1}^j \inp*{\mat{M}\tran q_j, q_i} q_i\tran \mat{M}\\
        &= \underbrace{\mu \sum_{\ell = 1}^{j + 1} \lambda_\ell x\tran \mat{M}^\ell}_{\in \mathcal{K}^{j + 2}(x, \mat{M})} - \underbrace{\mu \sum_{i = 1}^j \inp*{\mat{M} q_j, q_i} q_i\tran \mat{M}}_{\txtin{\ac{IH}} \mathcal{K}^{j + 1}(x, \mat{M})} \in \mathcal{K}^{j + 2}(x, \mat{M}).\qedhere
    \end{align*}
\end{proof}
\begin{corollary}\label{cor:arnoldi-iter-line11-span}
If Algorithm~\ref{alg:arnoldi-iteration} returns at line 11, we have that
\[
    \spn\Set{q_1\tran, \dots, q_{m + 1}\tran} = \mathcal{K}^{m + 1}(x, \mat{M}).
\]
\end{corollary}
\begin{proof}
    Viewing Algorithm~\ref{alg:arnoldi-iteration} with an input of $m + 1$ instead of $m$, not returning at line 9 for $1 \leq j \leq m$, we know from Proposition~\ref{prop:arnoldi-output-orthonorm-krylov} that for $j = m + 1$ at line 8 we have $\spn\set{q_1\tran, \dots, q_{m + 1}\tran} = \mathcal{K}^{m + 1}(x, \mat{M})$.
    If we go back to Algorithm~\ref{alg:arnoldi-iteration} with an input of $m$ returning at line 11, we can see that their $q_1, \dots, q_{m + 1}$ are the same, concluding the proof.
\end{proof}
\begin{proposition}\label{prop:arnoldi-iter-l9-invariant}
If Algorithm~\ref{alg:arnoldi-iteration} returns at line 9, we have that
\[
    \forall m \geq j : \mathcal{K}^m(x, \mat{M}) = \mathcal{K}^j(x, \mat{M}).
\]
\end{proposition}
\begin{proof}
    As we return in line 9, it must be that $\norm{r_{j + 1}}_2 = 0$.
    Thus, as similarly seen in the proof of Proposition~\ref{prop:arnoldi-output-orthonorm-krylov}
    \[
        r_{j + 1}~\txteq{(\ref{eq:r-j-gram-schmidt})} \underbrace{q_j\tran \mat{M}}_{\in \mathcal{K}^{j + 1(x, \mat{M})}} - \underbrace{\sum_{i = 1}^j \inp*{\mat{M}\tran q_j, q_i}q_i\tran}_{\in \mathcal{K}^j(x, \mat{M})} = 0.
    \]
    So, $\mathcal{K}^j(x, \mat{M}) = \mathcal{K}^{j + 1}(x, \mat{M})$, meaning $\mathcal{K}^j(x, \mat{M})$ is invariant under $\mat{M}$, which concludes the proof for all $m \geq j$.
\end{proof}
\begin{proposition}\label{prop:arnoldi-k-n1-terminate-9}
With $m \geq n$ in Algorithm~\ref{alg:arnoldi-iteration}, it must always return at line 9.
\end{proposition}
\begin{proof}
    It suffices to show that the proposition holds for $m = n$.
    If we return at line 9 for some $1 \leq j < m$, we are done.
    So suppose now we have not, setting $j = n$.
    Then, by Proposition~\ref{prop:arnoldi-output-orthonorm-krylov}, $q_1\tran, \dots q_n\tran$ is an orthonormal basis of $\mathcal{K}^{n - 1}(x, \mat{M})$.
    Note that each element of this basis is of size $n$, meaning it cannot grow any larger while retaining linear independence.
    With this, if $q_{n + 1}$ exists, $q_1, \dots, q_{n + 1}$ could not be linearly independent, as Proposition~\ref{prop:arnoldi-output-orthonorm-krylov} would suggest.
    We must conclude that the precondition for this proposition is never met, meaning $q_{n + 1}$ is never defined.
    Thus, if $j = n$, we must return at line 9 before the assignment of $q_{n + 1}$ at line 10.
\end{proof}
We have now seen that the Arnoldi iteration from Algorithm~\ref{alg:arnoldi-iteration} always produces a basis of the Krylov subspace $\mathcal{K}^{m + 1}(x, \mat{M})$.
Note that we have so far never used the terms of the form $h_{i, j}$.
This will change as they serve a central role in the following relation, which looks, and as we will later see, also is similar to the criterion for exactness of an aggregation.
\begin{lemma}[Arnoldi relation]\label{lem:arnoldi-relation}
As in~\cite[174]{arbenz2016lecturenotes}, we have
\[
    \underbrace{\pmat{
        h_{1,1}     & h_{1,2}       & 0             & \cdots    & 0 \\
        h_{2,1}     & h_{2,2}       & h_{2,3}       & \ddots    & \vdots \\
        h_{3,1}     & h_{3,2}       & h_{3,3}       & \ddots    & 0 \\
        \vdots      & \vdots        & \vdots        & \ddots    & h_{j-1,j} \\
        h_{j,1}     & h_{j,2}       & h_{j,3}       & \cdots    & h_{j,j}
    }}_{\eqqcolon \mat{H}_j} \underbrace{\pmat{
        \hpipe \, q_1\tran \, \hpipe \\
        \vdots \\
        \hpipe \, q_j\tran \, \hpipe
    }}_{\eqqcolon \mat{Q}_j} + \underbrace{\left(
    \begin{array}{c}
        \scalebox{1.33}{$\mathbf{0}$}\\
        \hpipe \, h_{j, j + 1}q_{j + 1}\tran \, \hpipe
    \end{array}
    \right)}_{\in \R^{j \times n}} = \mat{Q}_j \mat{M}
\]
in Algorithm~\ref{alg:arnoldi-iteration} at line 10.
\end{lemma}
\begin{proof}
    We start by recognizing for indices $i \leq j$ that
    \begin{equation}
        h_{j, i} \txteq{Alg.~\ref{alg:arnoldi-iteration},~(\ref{eq:r-j-gram-schmidt})} q_j\tran \mat{M} q_i - \sum_{\ell = 1}^{i - 1} \inp*{\mat{M}\tran q_j, q_\ell} \underbrace{\inp*{q_\ell, q_i}}_{= 0} \txteq{Def.~\ref{def:orthonormal}, Prop.~\ref{prop:arnoldi-output-orthonorm-krylov}} \inp*{\mat{M}\tran q_j, q_i}.\label{eq:h-j-i-closed-form}
    \end{equation}
    We can use this to then get
    \begin{align*}
        & h_{j, j + 1} q_{j + 1}\tran \txteq{Alg.~\ref{alg:arnoldi-iteration}} h_{j, j + 1} \frac{r_{j + 1}}{h_{j, j + 1}} = r_{j + 1} \txteq{(\ref{eq:r-j-gram-schmidt})} q_j\tran \mat{M} - \sum_{i = 1}^j \inp*{\mat{M}\tran q_j, q_i} q_i\tran\\
        \iff & q_j\tran \mat{M} = \sum_{i = 1}^j \inp*{\mat{M}\tran q_j, q_i} q_i\tran + h_{j, j + 1} q_{j + 1}\tran \txteq{(\ref{eq:h-j-i-closed-form})} \sum_{i = 1}^j h_{j, i} q_i\tran + h_{j, j + 1} q_{j + 1}\tran = \sum_{i = 1}^{j + 1} h_{j, i} q_i\tran.
    \end{align*}
    This is already enough, as $q_j\tran \mat{M}$ is the $j$-th row of the right-hand side and $\sum_{i = 1}^{j + 1} h_{j, i} q_i\tran$ the $j$-th row of the left-hand side in the original identity.
\end{proof}
\begin{corollary}\label{cor:arnoldi-iteration-zero-case}
If we return at line 9 in Algorithm~\ref{alg:arnoldi-iteration}, we have
\[\mat{H}_j \mat{Q}_j = \mat{Q}_j \mat{M}.\]
\end{corollary}
\begin{proof}
    We return at line 9 if, and only if, $h_{j, j + 1} = 0$, so we get $\mat{H}_j \mat{Q}_{j} \txteq{Lem.~\ref{lem:arnoldi-relation}} \mat{Q}_j \mat{M}$.\qedhere
\end{proof}
\begin{corollary}
    If $m \geq n$ in Algorithm~\ref{alg:arnoldi-iteration}, we have $\mat{Q}_j$ and $\mat{H}_j$ with $\mat{H}_j \mat{Q}_j = \mat{Q}_j \mat{M}$.
\end{corollary}
\begin{proof}
    From Proposition~\ref{prop:arnoldi-k-n1-terminate-9}, we know that $m \geq n$ assures we will return at line 9.
    So, combining this with Corollary~\ref{cor:arnoldi-iteration-zero-case}, we conclude that we indeed get $\mat{H}_j \mat{Q}_j = \mat{Q}_j \mat{M}$.
\end{proof}

\chapter{Finding small aggregations}\label{ch:finding-small-aggregations}
In this chapter, we will see how the Arnoldi iteration from Section~\ref{sec:arnoldi-iteration} can be used to construct aggregations for arbitrary transition matrices and initial distributions.
These are either exact with minimal state space size or their error  for the first few steps is zero, again with minimal state space size.
The formal definitions and proofs about this construction can be seen in Section~\ref{sec:exact-aggregations}.
Following, in Section~\ref{sec:aggregations-under-error-bounds}, we will see that the error bound from Section~\ref{sec:error-bounds-on-approximated-transient-distributions} cannot be improved in the general case of the previously defined aggregations.
Lastly, we are going to look at a special case where we will see that no improvement to the bound is possible, again.
\section{Exact and initially exact aggregations}\label{sec:exact-aggregations}
If possible, finding an exact aggregation of a Markov chain is ideal, as this allows us to keep the error regarding transient distributions at precisely zero.
We know that we must have $\mat{\Pi} \mat{A} = \mat{A} \mat{P}$ and $p_0 = \tilde{p}_0$ for any exact aggregation.
One can see that the Arnoldi iteration from Algorithm~\ref{alg:arnoldi-iteration} yields the Arnoldi relation from Lemma~\ref{lem:arnoldi-relation}, which looks pretty similar to the first condition of aggregation to be exact.
Thus, this chapter looks at how the Arnoldi iteration can generate exact aggregations or at least aggregations with no error in the initial $k$ steps.
Furthermore, it will become apparent that those aggregations have the smallest possible state space size regarding these properties.
\begin{definition}[$k$-exact aggregation]\label{def:k-exact-aggr}
With $k \in \N$, we call an aggregation of a Markov chain \emph{$k$-exact}, or informally \emph{initially exact}, if, and only if, we have with $i \in \N$, $i \leq k$, that
\[
    \pi_0\tran \mat{\Pi}^i \mat{A} = p_0\tran \mat{P}^i.
\]
\end{definition}
\begin{definition}[Arnoldi aggregation]\label{def:arnoldi-aggr}
Given $n, m \in \N \setminus \set{0}$ and a Markov chain with transition matrix $\mat{P} \in \R^{n \times n}$ and initial distribution $p_0 \in \R^n$, run the Arnoldi iteration from Algorithm~\ref{alg:arnoldi-iteration} with $m$, $p_0$ and $\mat{P}$ as inputs to receive $\mat{H}_j$ and $\mat{Q}_j$.

We then call the aggregation with $\mat{H}_j$ as the aggregated step matrix, $\mat{Q}_j$ as the disaggregation matrix and with $\pi_0 \coloneqq (\norm{p_0}_2, 0, \dots, 0)\tran \in \R^j$ the \emph{Arnoldi aggregation}.
\end{definition}
Although the choice of $\pi_0$ might seem arbitrary, we will later see that this choice guarantees $\norm{e_0}_1 = 0$.
Now that we have defined the central concept of an Arnoldi aggregation, we move on to examine how this aggregation behaves.
\begin{proposition}\label{prop:arnoldi-aggr-pik-form}
Given an Arnoldi aggregation, we have for all $k \in \N$, $1 \leq k \leq j - 1$
\[
    \pi_k = (\underbrace{\lambda_1, \dots, \lambda_{k + 1},}_{\text{arbitrary reals}} \underbrace{0, \dots, 0}_{j - k -1 \times})\tran.
\]
\end{proposition}
\begin{proof}
    Proceed by finite induction on $k$.
    The base case $k = 0$ is clear by Definition~\ref{def:arnoldi-aggr} of $\pi_0$.
    For the \ac{IH}, assume that with $1 \leq k \leq j - 2$, we always have $\pi_k = (\lambda_1, \dots, \lambda_{k + 1}, 0, \dots, 0)\tran$.
    We then see
    \begin{align*}
        \pi_{k + 1}\tran &\txteq{Def.~\ref{def:aggr-trans-distr}} \pi_k\tran \mat{H}_j \txteq{\ac{IH}} (\lambda_1, \dots, \lambda_{k + 1}, 0, \dots, 0)
        \pmat{
            h_{1,1}     & h_{1,2}       & 0             & \cdots    & 0 \\
            h_{2,1}     & h_{2,2}       & h_{2,3}       & \ddots    & \vdots \\
            h_{3,1}     & h_{3,2}       & h_{3,3}       & \ddots    & 0 \\
            \vdots      & \vdots        & \vdots        & \ddots    & h_{j-1,j} \\
            h_{j,1}     & h_{j,2}       & h_{j,3}       & \cdots    & h_{j,j}
        }\\
        &= \sum_{i = 1}^{k + 1} \lambda_i (h_{i, 1}, \dots, h_{i, i + 1}, \underbrace{0, \dots, 0}_{j - i - 1\times}).
    \end{align*}
    The last summand $\lambda_{k + 1} (h_{k + 1, 1}, \dots, h_{k + 1, k + 2}, 0, \dots, 0)$ is precisely of the wanted form.
    As such, the entire sum is of that form.
\end{proof}
\begin{proposition}\label{prop:formula-arnoldi-aggr-error-basic}
In an Arnoldi aggregation with some number of steps $k \in \N$ we have
\[
    \pi_0\tran \mat{H}_j^k \mat{Q}_j = \pi_0\tran \mat{Q}_j \mat{P}^k - \sum_{i = 0}^{k - 1} \pi_0\tran \mat{H}_j^i
    \left( \begin{array}{c}
               \scalebox{1.33}{$\mathbf{0}$}\\
               \hpipe \, h_{j, j + 1}q_{j + 1}\tran \, \hpipe
    \end{array} \right)
    \mat{P}^{k - 1 - i}.
\]
\end{proposition}
\begin{proof}
    Again, this proof is done by induction on $k$.
    Starting with $k = 0$, it is true:
    \[
        \pi_0\tran \mat{H}_j^0 \mat{Q}_j = \pi_0\tran \mat{Q}_j = \pi_0\tran \mat{Q}_j \mat{P}^0 - 0 = \pi_0\tran \mat{Q}_j \mat{P}^0 - \sum_{i = 0}^{-1} \pi_0\tran \mat{H}_j^i
        \underbrace{\left( \begin{array}{c}
                               \scalebox{1.33}{$\mathbf{0}$} \\
                               \hpipe \, h_{j, j + 1}q_{j + 1}\tran \, \hpipe
        \end{array} \right)}_{\eqqcolon \mat{E}}
        \mat{P}^{k - 1 - i}.
    \]
    For some fixed $k$, set the \ac{IH} as
    \[
        \pi_0\tran \mat{H}_j^k \mat{Q}_j = \pi_0\tran \mat{Q}_j \mat{P}^k - \sum_{i = 0}^{k - 1} \pi_0\tran \mat{H}_j^i \mat{E} \mat{P}^{k - 1 - i}.
    \]
    Lastly, we conclude this proof by showing
    \begin{align*}
        \pi_0\tran \mat{H}_j^{k + 1} \mat{Q}_j &\txteq{Lem.~\ref{lem:arnoldi-relation}} \pi_0\tran \mat{H}_j^k \left(\mat{Q}_j \mat{P} - \mat{E} \right)\\
        &=  \pi_0\tran \mat{H}_j^k\mat{Q}_j \mat{P} - \pi_0\tran \mat{H}_j^k \mat{E}\\
        &\txteq{\ac{IH}} \left(\pi_0\tran \mat{Q}_j \mat{P}^k - \sum_{i = 0}^{k - 1} \pi_0\tran \mat{H}_j^i \mat{E}
        \mat{P}^{k - 1 - i}\right)\mat{P} - \pi_0\tran \mat{H}_j^k \mat{E}\\
        &= \pi_0\tran \mat{Q}_j \mat{P}^{k + 1} - \sum_{i = 0}^{k - 1} \pi_0\tran \mat{H}_j^i \mat{E}
        \mat{P}^{k - i} - \pi_0\tran \mat{H}_j^k \mat{E}\\
        &= \pi_0\tran \mat{Q}_j \mat{P}^{k + 1} - \sum_{i = 0}^k \pi_0\tran \mat{H}_j^i \mat{E}
        \mat{P}^{k - i}.\qedhere
    \end{align*}
\end{proof}
\begin{lemma}\label{lem:arnoldi-aggr-k0-exact}
An Arnoldi aggregation with state space size $j \in \N \setminus \set{0}$ is $(j - 1)$-exact.
\end{lemma}
\begin{proof}
    For any $i \in \N$, $i \leq j - 2$ we have by Proposition~\ref{prop:arnoldi-aggr-pik-form} $\pi_i(j) = 0$.
    So we get
    \begin{equation}
        \pi_i\tran
        \left( \begin{array}{c}
                   \scalebox{1.33}{$\mathbf{0}$}\\
                   \hpipe \, h_{j, j + 1}q_{j + 1}\tran \, \hpipe
        \end{array} \right)
        = (0, \dots, 0).\label{eq:pi-k-E-is-zero}
    \end{equation}
    Using Proposition~\ref{prop:formula-arnoldi-aggr-error-basic}, we get for any $k \in \N$, $k \leq j -1$ that
    \begin{align*}
        \pi_0\tran \mat{H}_j^k \mat{Q}_j &= \pi_0\tran \mat{Q}_j \mat{P}^k - \sum_{i = 0}^{k - 1} \pi_0\tran \mat{H}_j^i
        \left( \begin{array}{c}
                   \scalebox{1.33}{$\mathbf{0}$}\\
                   \hpipe \, h_{j, j + 1}q_{j + 1}\tran \, \hpipe
        \end{array} \right)
        \mat{P}^{k - 1 - i}\\
        &\txteq{Def.~\ref{def:aggr-trans-distr},~(\ref{eq:pi-k-E-is-zero})} \pi_0\tran \mat{Q}_j \mat{P}^k - \sum_{i = 0}^{k - 1} (0, \dots, 0) \mat{P}^{k - 1 - i}\\
        &= \pi_0\tran \mat{Q}_j \mat{P}^k \txteq{Def.~\ref{def:arnoldi-aggr}} \norm{p_0}_2 \cdot q_1\tran \mat{P}^k \txteq{Alg.~\ref{alg:arnoldi-iteration}} p_0\tran \mat{P}^k.\qedhere
    \end{align*}
\end{proof}
\begin{lemma}\label{lem:arnoldi-aggr-k0-geq-n-exact}
With $n \in \N \setminus \set{0}$, any Arnoldi aggregation where the inputs are $p_0 \in \R^n$, $\mat{P} \in \R^{n \times n}$ and $m \in \N$, $m \geq n$ is exact.
\end{lemma}
\begin{proof}
    We know from Proposition~\ref{prop:arnoldi-k-n1-terminate-9} that by choosing $m \geq n$, we will return at line 9 in Algorithm~\ref{alg:arnoldi-iteration}.
    So, with Corollary~\ref{cor:arnoldi-iteration-zero-case}, we have the special case of $\mat{H}_j \mat{Q}_j = \mat{Q}_j \mat{A}$ for the Arnoldi relation.
    With this, we have already shown the first condition by Definition~\ref{def:dyn-exact-aggr} for the Aggregation to be exact.
    The second condition is shown through
    \[
        \pi_0\tran \mat{Q}_j \txteq{Def.~\ref{def:arnoldi-aggr}} \norm{p_0}_2 \cdot q_1\tran \txteq{Alg.~\ref{alg:arnoldi-iteration}} \norm{p_0}_2 \frac{p_0\tran}{\norm{p_0}_2} = p_0\tran.\qedhere
    \]
\end{proof}
\begin{lemma}\label{lem:k0-arnoldi-aggr-is-minimal}
Given an initial distribution $p_0 \in \R^n$, transition matrix $\mat{P} \in \R^{n \times n}$ and $j \in \N \setminus \set{0}$, there is no $(j - 1)$-exact aggregation of $p_0$ and $\mat{P}$ with a smaller state space than an Arnoldi aggregation with state space size $j$.
\end{lemma}
\begin{proof}
    Assume that an aggregation given through $\mat{\Pi} \in \R^{m \times m}$, $\mat{A} \in \R^{m \times n}$ and $\pi_0 \in \R^m$ is $(j - 1)$-exact.
    By definition, this means that for $k \in \N$, that
    \begin{align*}
        \forall k \leq j - 1 :\ & \pi_0\tran \mat{\Pi}^k \mat{A} = p_0\tran \mat{P}^k\\
        \iff & \pi_k\tran \mat{A} = p_k\tran\\
        \iff & \sum_{i = 1}^m \pi_k(i) (\mat{A}(i,1), \dots, \mat{A}(i,n)) = p_k\tran.
    \end{align*}
    This linear combination means $p_0\tran, \dots, p_{j - 1}\tran$ are all elements of the row space of $\mat{A}$, thus with \enquote{$\leq$} denoting the subspace relation,
    \begin{align*}
        & \spn\Set{p_0\tran, \dots, p_{j - 1}\tran} \leq \rowsp(\mat{A})\\
        \txtimplies{Prop.~\ref{prop:arnoldi-output-orthonorm-krylov}} & \spn\Set{q_1\tran, \dots, q_j\tran} \leq \rowsp(\mat{A}).
    \end{align*}
    As such, the state space size, being the same as the dimension of $\rowsp(\mat{A})$, is minimal if, and only if, $\spn\set{q_1\tran, \dots, q_j\tran} = \rowsp(\mat{A})$.
    Again, through Proposition.~\ref{prop:arnoldi-output-orthonorm-krylov} we know that $q_1, \dots, q_j$ are pairwise linearly independent, meaning we must have $m = \dim(\rowsp(\mat{A})) = j$, which the disaggregation matrix $\mat{Q}_j$ of the Arnoldi aggregation precisely achieves.
    Thus, the row space of the disaggregation matrix in the Arnoldi aggregation is minimal, meaning the state space is also.
\end{proof}
\begin{lemma}\label{lem:k0-geq-n-arnoldi-aggr-is-minimal}
Given an initial distribution $p_0 \in \R^n$, transition matrix $\mat{P} \in \R^{n \times n}$ and $m \geq n$, there is no exact aggregation with a smaller state space than the state space of the Arnoldi aggregation of size $j \in \N$ with integer input $m$.
\end{lemma}
\begin{proof}
    Almost analogous to the proof of Lemma~\ref{lem:k0-arnoldi-aggr-is-minimal}, we choose any exact aggregation with $\mat{\Pi} \in \R^{\ell \times \ell}$, $\mat{A} \in \R^{\ell \times n}$ and $\pi_0 \in \R^\ell$ as its aggregated step matrix, disaggregation matrix, and aggregated initial distribution.
    Now, as we have exactness and not just $k$-exactness, we know by Proposition~\ref{prop:exact-aggr-implies-k-exact} that $\spn\set{p_0\tran, p_1\tran, \dots} \leq \rowsp(\mat{A})$.
    Again, by using Proposition.~\ref{prop:arnoldi-output-orthonorm-krylov} and additionally Propositions~\ref{prop:arnoldi-iter-l9-invariant} and~\ref{prop:arnoldi-k-n1-terminate-9}, we know that $q_1\tran, \dots q_j\tran$ form a basis of $\rowsp(\mat{A})$.
    With the same argument as in Lemma~\ref{lem:k0-arnoldi-aggr-is-minimal}, we have at least $j$ states in an exact aggregation.
\end{proof}
\begin{theorem}\label{thrm:arnoldi-aggr-is-smallest-k-1-exact}
If the Arnoldi iteration with integer input $m \in \N \setminus \set{0}$, used for computing the Arnoldi aggregation of a given Markov chain returns at line 11 in Algorithm~\ref{alg:arnoldi-iteration}, we get a $(m - 1)$-exact aggregation with the smallest possible state space size.
\end{theorem}
\begin{proof}
    If we return at line 11, we have a state space size of $m$ in the Arnoldi aggregation by Corollary~\ref{cor:arnoldi-iter-line11-span}.
    By Lemma~\ref{lem:arnoldi-aggr-k0-exact}, this aggregation is $(m - 1)$ exact.
    Lastly, Lemma~\ref{lem:k0-arnoldi-aggr-is-minimal} tells us that the Arnoldi aggregation also has minimal state space.
\end{proof}
\begin{theorem}\label{thrm:arnoldi-aggr-is-smallest-exact}
For $n \in \N \setminus \set{0}$, the Arnoldi iteration with the inputs $p_0 \in \R^n$, $\mat{P} \in \R^{n \times n}$, and $m \in \N$, $m \geq n$, used to compute the Arnoldi aggregation of the given Markov chain, yields an exact aggregation with the smallest possible state space size.
\end{theorem}
\begin{proof}
    Here, we combine Lemma~\ref{lem:arnoldi-aggr-k0-geq-n-exact} and~\ref{lem:k0-geq-n-arnoldi-aggr-is-minimal}.
\end{proof}

\section{Aggregations under error bounds}\label{sec:aggregations-under-error-bounds}
Now that we have introduced the Arnoldi aggregation in Section~\ref{sec:exact-aggregations}, we want to look at how the error bounds from Section~\ref{sec:error-bounds-on-approximated-transient-distributions} behave.
We are going to see that we cannot improve the bound shown in Proposition~\ref{prop:general-error-bound} in the general case, as well as in a specific case of Arnoldi aggregations.
A Markov chain from the class of nearly completely decomposable Markov chains will provide an example.
Intuitively, such a Markov chain is formed by considering multiple separate Markov chains linked together through transitions with very low probability.

\begin{definition}[\Ac{NCD} Markov chains]
    A \emph{\ac{NCD} Markov chain} is a \ac{DTMC} represented by the transition matrix $\mat{P} \in \R^{n \times n}$, which is of the form
    \[
        \mat{P} \coloneqq \pmat{
            \mat{P}_1 & & & & \\
            & \ddots & & \scalebox{1.33}{$\mathbf{0}$} & \\
            & & \mat{P}_I & & \\
            & \scalebox{1.33}{$\mathbf{0}$} & & \ddots &\\
            & & & & \mat{P}_N
        }
        + \varepsilon \mat{C}.
    \]
    Here, $\mat{P}_1, \dots, \mat{P}_N$ are again transition matrices representing the individual Markov chains.
    Further, let $\mat{C} \in \R^{n \times n}$ and $\varepsilon \in \R_+$, which represent the transitions between these individual Markov chains\cite[114--115]{simon1961ncdchains}.
    We further require all row sums of $\mat{C}$ to be zero with $\mat{C}$ containing $-1$, $0$ or $1$ and $\varepsilon < 1$.
    Additionally, we can have any $n \in \Z$ on the diagonal of $\mat{C}$.
\end{definition}
Usually, $\varepsilon$ is chosen relatively small to preserve the idea of a small transition probability between any two of the Markov chains making up $\mat{P}$.
Naturally; this definition lends itself to the intuitive understanding of an aggregation: Each smaller Markov chain $\mat{P}_I$ can be seen as an aggregate, and two such aggregates can have transition probability $\varepsilon$ between each other.
Still, the error bound presented in Proposition~\ref{prop:general-error-bound} cannot be improved.
\begin{proposition}\label{prop:ncd-bound-counterex}
Using the Arnoldi iteration to aggregate \ac{NCD} Markov chains, Proposition~\ref{prop:general-error-bound} can be fulfilled with equality.
\end{proposition}
\begin{proof}
    Let $\mat{P_1} \coloneqq (1) \in \R^{1 \times 1}$, $\mat{P}_2 \in \R^{2 \times 2}$, $p_0 \coloneqq (0, 0, 1)\tran \in \R^3$, $\mat{C} \in \R^{3 \times 3}$ and $\varepsilon \in (0, 1)$.
    We then define
    \[
        \mat{P}_2 = \pmat{1 & 0\\ 1&0},\ \mat{C} \coloneqq \pmat{-1 & 1 & 0\\ 0 & 0 & 0\\ 0 & 0 & 0},\ \text{resulting in}\ \mat{P} = \pmat{1 - \varepsilon & \varepsilon & 0\\ 0 & 1 & 0\\ 0 & 1 & 0}.
    \]
    We now show for $k \geq 1$ that $p_k = (0, 1, 0)\tran$.
    This is clear for $k = 1$ by simple matrix-vector multiplication.
    Then in the inductive step we see
    \[
        p_{k + 1}\tran = p_k\tran \pmat{1 - \varepsilon & \varepsilon & 0\\ 0 & 1 & 0\\ 0 & 1 & 0} = (0, 1, 0) \pmat{1 - \varepsilon & \varepsilon & 0\\ 0 & 1 & 0\\ 0 & 1 & 0} = (0, 1, 0).
    \]
    Now, when again using one as the aggregated state space size, we get, similar to the previous counterexample, that
    \begin{align*}
        &\mat{H}_1 = (0),\ \pi_0 = (1),\ \tilde{p}_k\tran =\begin{cases}
        (0, 0, 1), & \text{if $k = 0$,}\\
        (0, 0, 0), & \text{else,}
        \end{cases}\\
        &\text{and}\ \norm{e_k}_1 =
        \begin{cases}
            0, & \text{if $k = 0$,}\\
            1, & \text{else.}
        \end{cases}
    \end{align*}
    This finally enables us to see the tightness of the error bound:
    \begin{align*}
        \norm*{e_k}_1 &\txteq{Def.~\ref{def:error-k-vec}} \norm*{\tilde{p}_k - p_k}_1 =
        \begin{cases}
            0, & \text{if $k = 0$,}\\
            1, & \text{else.}
        \end{cases} = 0 + 1 \cdot 1 \cdot
        \begin{cases}
            0, & \text{if $k = 0$,}\\
            1, & \text{else.}
        \end{cases}\\
        &= \norm*{e_0}_1 + \norm*{\pi_0}_1 \cdot \norm*{\mat{H}_1 \mat{Q}_1 - \mat{Q}_1 \mat{P}}_\infty \cdot
        \begin{cases}
            \frac{\norm{\mat{H}_1}_\infty^k - 1}{\norm{\mat{H}_1}_\infty - 1}, & \text{if $\norm{\mat{H}_1}_\infty \neq 1$},\\
            k, & \text{else.}
        \end{cases}\qedhere
    \end{align*}
\end{proof}
So, we have seen now in Proposition~\ref{prop:ncd-bound-counterex} that the special case of aggregating NCD Markov chains cannot yield improved error bounds.
Thus, no improvement is possible in the general case either.
Of course, there might still be cases where the error bound can be improved, but we can see with the counterexample that doing so requires stricter conditions.
For example, we chose the counterexample so that $\mat{P}$ maps $p_0$ to an eigenvector orthogonal to $p_0$.
Preventing this is necessary to improve the error bound.
Still, there might be even more requirements.

\chapter{Numerical implementation}\label{ch:numerical-implementation}
Now that we have finished the theoretical framework for Arnoldi aggregations, we examine such aggregations with an implementation of the underlying Arnoldi iteration.
To understand how such an implementation might look, we look at sources of numerical instabilities, ways to combat them, and runtime and memory complexity in Section~\ref{sec:computing-arnoldi-aggregations}.
Then, in Section~\ref{sec:models-and-methodology}, we briefly introduce different Markov chain models, with which we will evaluate the implementation as devised in the previous sections.
Finishing the chapter with Section~\ref{sec:review}, we present the results of the previously defined
experiments and assess them when compared to theoretical results from Chapter~\ref{ch:finding-small-aggregations} or other aggregation methods.
\section{Computing Arnoldi aggregations}\label{sec:computing-arnoldi-aggregations}
To analyse potential improvements, problems and pitfalls when implementing the Arnoldi iteration, we start in Section~\ref{subsec:normalization-of-the-approximated-transient-distribution} by showing that normalizing $\tilde{p}_k$ with respect to the 1-norm can improve $\norm{e_k}_1$.
Then, in Section~\ref{subsec:floating-point-errors}, we examine problems arising from numerical errors in floating-point arithmetic.
Namely, this concerns the gradual loss of orthogonality in a basis built with \ac{CGS}.
We also show how to dampen this loss of orthogonality.
This results in the condition for an early return in Algorithm~\ref{alg:arnoldi-iteration} not to trigger.
Although we cannot prevent this, we present an alternative with some theoretical framework in Section~\ref{subsec:determining-convergence}.
Finally, in Section~\ref{subsec:runtime-and-memory-complexity} we analyse the runtime and memory complexity of computing an Arnoldi aggregation.
\subsection{Normalizing approximated transient distributions}\label{subsec:normalization-of-the-approximated-transient-distribution}
Here, we will introduce a method to improve the error $\norm{e_k}_1$ for arbitrary aggregations with very little computational overhead.
By definition, we want $\tilde{p}_k$ to approximate the probability distribution $p_k$.
So, intuitively speaking, $\norm{\tilde{p}_k}_1$ should be close to one to be a good approximation of $p_k$.
As such, we will look at when normalization of $\tilde{p}_k$ with respect to the 1-norm improves the error, or more formally, we will look at when
\begin{equation}
    \norm*{\frac{\tilde{p}_k}{\norm{\tilde{p}_k}_1} - p_k}_1 \leq \norm{\tilde{p}_k - p_k}_1\label{eq:normalization-helps}
\end{equation}
is fulfilled without knowing more about $p_k$ besides that it is a probability distribution.

Five different possible conditions for $\tilde{p}_k$ are presented, such that the inequality from~(\ref{eq:normalization-helps}) holds for all $p_k$.
Note that only two are proven here; the rest remain as mere conjectures.
\begin{proposition}
    Consider an aggregation of a Markov chain and $k \in \N$.
    If for the approximated transient distribution $\tilde{p}_k$ we have $\norm{\tilde{p}_k}_1 \geq 2$, then~(\ref{eq:normalization-helps}) holds true.
\end{proposition}
\begin{proof}
    Assume that $\norm{\tilde{p}_k}_1 \geq 2$.
    The result follows from
    \begin{align*}
        2 \norm*{\frac{\tilde{p}_k}{\norm*{\tilde{p}_k}_1} - p_k}_1 &\leq \norm*{\tilde{p}_k}_1 \cdot \norm*{\frac{\tilde{p}_k}{\norm*{\tilde{p}_k}_1} - p_k}_1 = \norm*{\tilde{p}_k - \norm{\tilde{p}_k}_1 p_k}_1 = \norm*{\tilde{p}_k - p_k + p_k - \norm{\tilde{p}_k}_1 p_k}_1\\
        &\leq \norm*{\tilde{p}_k - p_k}_1 + \norm*{p_k - \norm{\tilde{p}_k}_1 p_k}_1 = \norm*{\tilde{p}_k - p_k}_1 + \underbrace{\norm*{p_k}_1}_{= 1} \cdot \abs*{1 - \norm{\tilde{p}_k}_1}\\
        &= \norm*{\tilde{p}_k - p_k}_1 + \abs*{1 - \norm{\tilde{p}_k}_1} = \norm*{\tilde{p}_k - p_k}_1 + \abs*{\norm{p_k}_1 - \norm{\tilde{p}_k}_1}\\
        &\txtleq{Inv. triangle inequ.} \norm*{\tilde{p}_k - p_k}_1 + \norm*{\tilde{p}_k - p_k}_1 = 2 \norm*{\tilde{p}_k - p_k}_1.\qedhere
    \end{align*}
\end{proof}
\begin{proposition}
    Consider an aggregation of a Markov chain and $k \in \N$.
    If for the approximated transient distribution $\tilde{p}_k$ we have $\norm{\tilde{p}_k}_1 \geq 1$ and for all $i \in \N$ with $1 \leq i \leq n$ we have $\tilde{p}_k(i) \leq 0$, then~(\ref{eq:normalization-helps}) holds true.
\end{proposition}
\begin{proof}
    This can be shown by breaking down the term into sums concerning the individual components, giving us
    \begin{align*}
        \norm*{\frac{\tilde{p}_k}{\norm*{\tilde{p}_k}_1} - p_k}_1 = \sum_{i = 1}^n \abs*{\frac{\tilde{p}_k(i)}{\norm*{\tilde{p}_k}_1} - p_k(i)}.
    \end{align*}
    Now as $\tilde{p}_k(i)$ is always negative and $\norm{\tilde{p}_k}_1 \geq 1$, we know that $\tilde{p}_k(i) \leq \frac{\tilde{p}_k(i)}{\norm*{\tilde{p}_k}_1}$.
    Combining this with the fact that $p_k(i) \geq 0$, results in
    \[
        \sum_{i = 1}^n \abs*{\frac{\tilde{p}_k(i)}{\norm*{\tilde{p}_k}_1} - p_k(i)} \leq \sum_{i = 1}^n \abs*{\tilde{p}_k(i) - p_k(i)} = \norm*{\tilde{p}_k - p_k}_1.\qedhere
    \]
\end{proof}
Now, the following three conjectures remain unproven as pure postulations.
\begin{conjecture}\label{conj:unproven1}
Consider an aggregation of a Markov chain and $k \in \N$.
If for the approximated transient distribution $\tilde{p}_k$ there is some $i \in \N$ with $1 \leq i \leq n$ and $\tilde{p}_k(i) \leq -1$, then~(\ref{eq:normalization-helps}) holds true.
\end{conjecture}
\begin{conjecture}\label{conj:unproven2}
Consider an aggregation of a Markov chain and $k \in \N$.
If for the approximated transient distribution $\tilde{p}_k$ there is some $i \in \N$ with $1 \leq i \leq n$ and $\tilde{p}_k(i) \geq \frac{9}{8}$, then~(\ref{eq:normalization-helps}) holds true.
\end{conjecture}
\begin{conjecture}\label{conj:unproven3}
Consider an aggregation of a Markov chain and $k \in \N$.
If for the approximated transient distribution $\tilde{p}_k$ we have pairwise different indices $i_1, \dots i_{n - 1}$ with $\sum_{j = 1}^{n - 1} \tilde{p}_k(i_j) \leq -1$, then~(\ref{eq:normalization-helps}) holds true.
\end{conjecture}
These last three conjectures originate from their numerical \enquote{confirmation} in two and three dimensions.
Further, note that even when assuming all propositions to be true, these bounds are still not tight.
For example, with $n = 3$, consider $\tilde{p}_k = \left( \frac{1}{3}, \frac{1}{3}, \frac{1}{3} \right)\tran$.
This vector is covered by none of the above propositions or conjectures, yet~({\ref{eq:normalization-helps}}) still holds with equality.

\subsection{Floating-point errors}\label{subsec:floating-point-errors}
As we have established in Chapter~\ref{ch:theoretical-preliminaries}, the Arnoldi iteration is equivalent to the Gram-Schmidt procedure, where we save some intermediary results.
Thus, we can understand the floating point errors arising in the Arnoldi iteration by studying those in the Gram-Schmidt procedure, which have already been studied widely.
This enables us to understand where different types of floating-point errors might arise, and how we can combat them.
As such, we will compare \ac{CGS} and \ac{MGS}, and offer further variants of both with greater numerical stability.
For this, we will refer to~\cite{björck2013100years}.
More precisely, we will use pages 500 to 504 to explain the numerical loss of orthogonality with \ac{CGS} juxtaposed against \ac{MGS}.
Furthermore, on pages 506 to 508, we will see that we can improve both \ac{CGS} and \ac{MGS} numerically by repeating the orthogonalization of the vectors.

We will just consider the case of applying the Gram-Schmidt procedure to a basis $v_1, \dots, v_n \in \R^m$.
We can assume that the normalization of vectors comes at a negligible impact on numerical stability.
As such, we have that $q_1$ is the same in floating-point arithmetic in \ac{CGS} and \ac{MGS}.
Even further, $q_2$ is too.
This is because the exact same arithmetic operations are executed in the same order in \ac{CGS} and \ac{MGS} in the case of $k = 2$.
Now, we have that $\inp{q_1, q_2} \eqqcolon \varepsilon > 0$ as in most cases, floating point arithmetic will cause $q_1$ and $q_2$ to be close to orthonormal, but not exactly so.
We denote this small error made in the first two steps by $\varepsilon$.
From this point onwards, the error made by \ac{CGS} and \ac{MGS} will differ, as this $\varepsilon$ will propagate differently.
Additionally, we assume that no further errors are made beyond this first error $\varepsilon$.
Of course, in reality this is not the case, but it can be shown that all further errors propagate similarly to the initial $\varepsilon$.

As a formal analysis of the numerical error would far exceed the scope of this thesis, we will just look at the intuitive reasoning behind the numerical superiority of \ac{MGS} over \ac{CGS}.
We firstly process a basis $v_1, \dots, v_n$ of $\R^n$ with \ac{CGS}.
There, we can use the closed-form formula from Proposition~\ref{prop:closed-form-cgs} to see that for $i \neq j$, we have that
\begin{equation}
    \begin{aligned}
        \inp*{q_i, q_j} &= \frac{\inp*{q_i, v_j - \sum_{\ell = 1}^{j - 1} \inp{q_\ell, v_j} q_\ell}}{\norm{v_j - \sum_{\ell = 1}^{j - 1} \inp{q_\ell, v_j} q_\ell}_2}\\
        &= \frac{\inp*{q_i, v_j} - \sum_{\ell = 1}^{j - 1} \inp{q_\ell, v_j}\inp*{q_i, q_\ell}}{\norm{v_j - \sum_{\ell = 1}^{j - 1} \inp{q_\ell, v_j}q_\ell}_2}\\
        &= \frac{-\sum_{\ell = 1, \ell \neq i}^{j - 1} \inp{q_\ell, v_j}\inp*{q_i, q_\ell}}{\norm{v_j - \sum_{\ell = 1}^{j - 1} \inp{q_\ell, v_j}q_\ell}_2}.
    \end{aligned}\label{eq:cgserr}
\end{equation}
We can then proceed by induction on $j \geq 2$ to show that $\inp*{q_1, q_j}$ will never be zero.
As such, by Equation~\eqref{eq:cgserr} it is clear that $\inp*{q_1, q_j}$ will always contain the errors made by $\inp*{q_1, q_2}, \dots, \inp*{q_1, q_{j - 1}}$ as long as these individual errors don't cancel and $\inp{q_\ell, v_j} \neq 0$ for some $j$.
In general, these problems will persist.
But cases where they do are quite atypical and as such ignored.

Moving on, we process the same basis as above, but with \ac{MGS}.
We will further use the notation of $v_j^{(i)}$ denoting the modified $v_j$ after $i$ steps in \ac{MGS}, as introduced in the proof of Proposition~\ref{prop:cgs-equiv-mgs}.
Firstly, we will see by induction on $j$ that here $\inp{q_i, q_j} = 0$, while $i \neq j$, $i < j$ and $i \neq 1$.
Starting with the base case $j = i + 1$, we get
\begin{align*}
    \inp{q_i, q_{i + 1}} &= \frac{\inp{q_i, v_{i + 1}^{(i)}}}{\norm{v_{i + 1}^{(i)}}_2} = \frac{\inp{q_i, v_{i + 1}^{(i - 1)} - \inp{q_i, v_{i + 1}^{(i - 1)}}q_i}}{\norm{v_{i + 1}^{(i)}}_2}\\
    &= \frac{\inp{q_i, v_{i + 1}^{(i - 1)}} - \inp{q_i, v_{i + 1}^{(i - 1)}}\inp{q_i, q_i}}{\norm{v_{i + 1}^{(i)}}_2} = \frac{0}{\norm{v_{i + 1}^{(i)}}_2} = 0,
\end{align*}
finishing this base case.
We then proceed to the inductive step of the first statement with
\begin{align*}
    \inp{q_i, q_{j + 1}} &= \frac{\inp{q_i, v_{j + 1}^{(j)}}}{\norm{v_{j + 1}^{(j)}}_2} = \frac{\inp{q_i, v_{j + 1}^{(j - 1)} - \inp{q_j, v_{j + 1}^{(j - 1)}}q_j}}{\norm{v_{j + 1}^{(j)}}_2}\\
    &= \frac{\inp{q_i, v_{j + 1}^{(j - 1)}} - \inp{q_j, v_{j + 1}^{(j - 1)}}\inp{q_i, q_j}}{\norm{v_{j + 1}^{(j)}}_2} \txteq{\ac{IH}} \frac{\inp{q_i, v_{j + 1}^{(j - 1)}}}{\norm{v_{j + 1}^{(j)}}_2}.
\end{align*}
Here, we can continue to apply the recursive formula $v_j^{(i)} = v_j^{(i - 1)} - \inp{q_i, v_j^{(i - 1)}}q_i$ to the numerator, until this entire term becomes zero too, as it will eventually be of the form
\[
    \frac{\inp{q_i, v_j^{(1)}} - \inp{q_i, v_j^{(1)}}}{\norm{v_{j + 1}^{(j)}}_2} = 0,
\]
because $\inp{q_i, q_i} = 1$, finishing the inductive step.
This technique no longer works for $\inp{q_1, q_j}$, where we will see for $j \geq 2$ that $\inp{q_1, q_j} \neq 0$.
Again, proceeding by induction on $j$, the base case $j = 2$ is clear.
By using the recursive formula above, we get
\begin{align*}
    \inp{q_1, q_{j + 1}} &= \frac{\inp{q_1, v_{j + 1}^{(j)}}}{\norm{v_{j + 1}^{(j)}}_2} = \frac{\inp{q_1, v_{j + 1}^{(j - 1)} - \inp{q_j, v_{j + 1}^{(j - 1)}}q_j}}{\norm{v_{j + 1}^{(j)}}_2}\\
    &= \frac{\inp{q_1, v_{j + 1}^{(j - 1)}} - \inp{q_j, v_{j + 1}^{(j - 1)}}\inp{q_1, q_j}}{\norm{v_{j + 1}^{(j)}}_2} \txtneq{\ac{IH}} 0
\end{align*}
in the inductive step, as long as $\inp{q_j, v_{j + 1}^{(j - 1)}} \neq 0$.
Again, this cannot be assumed in general but in most cases it will hold.
For a general numerical analysis of \ac{CGS} and \ac{MGS}, see~\cite{björck1994numericana}.
There, one can find references to other, more in-depth sources.

Now we have seen how the floating-point error made by $\inp{q_1, q_2}$ propagates in both \ac{CGS} and \ac{MGS}.
In \ac{MGS}, $\inp{q_i, q_j}$ is unaffected for $i \neq 1$, contrasting \ac{CGS} where this error appears in all terms of the form $\inp{q_i, q_j}$.
This gives an intuition as to why \ac{MGS} is numerically more stable than \ac{CGS}.

Now, lastly, we will consider a simple option on how to combat the error as described above.
This can be done through so-called \emph{reorthogonalization}.
We will only describe this method very superficially.
As such, the \ac{CGS2} is the same as Algorithm~\ref{alg:classical-gs} with lines $r_{j,k} \coloneqq \inp{q_j, q_k}$ and $q_k = q_k - r_{j,k} q_j$ added in the innermost loop after line 6.
Analogously, the \ac{MGS2} is just as Algorithm~\ref{alg:modified-gs} with lines $r_{j,k} \coloneqq \inp{q_k, v_j}$ and $v_j = v_j - r_{j,k} q_k$ added in the innermost loop after line 7.

Of course, this severely impacts the runtime negatively.
As such, there are different conditions which can be checked to see whether these additional lines should be executed.
These variants of \ac{CGS} and \ac{MGS} are called the \ac{CGSIR} and \ac{MGSIR}.
See~\cite[506--508]{björck2013100years} for a more detailed overview on reorthogonalization in the Gram-Schmidt procedure.
Usually, \ac{CGSIR} offers the best compromise between speed and accuracy, as the loss of orthogonality is much smaller than in \ac{CGS} while being comparable to the slower alternatives of \ac{CGS2}, \ac{MGSIR} or \ac{MGS2}.
We will later confirm this in our use-case.

\subsection{Determining convergence}\label{subsec:determining-convergence}
In the context of using the Arnoldi iteration as an aggregation method, determining convergence refers to determining at what aggregated state space size the aggregation reaches a sufficient level of exactness.

The first idea is, by following Algorithm~\ref{alg:arnoldi-iteration}, to check for which $j$, we would have $h_{j,j+1} \approx 0$.
This would allow us to detect when an exact aggregation is found.
Unfortunately, we will later see in experiments that this does not work due to numerical errors.

Thus, we look at the error bound
\[
    \norm{e_k}_1 \leq \norm{e_0}_1 + \sum_{j = 0}^{k - 1} \inp*{\abs*{\pi_j}, \abs*{\mat{\Pi} \mat{A} - \mat{A} \mat{P}} \cdot \mathbf{1}_n}
\]
from Proposition~\ref{prop:specific-error-bound}.
Now, we can check the quality of an aggregation by determining the above bound for some fixed $k$.
This is done in every successive aggregated state space size.
Once the error bound has fallen below a given size for some $k$, we can stop expanding the aggregated state space.
While this idea allows us to properly assess how \enquote{good} an aggregation is, it is only useful for large $k$, which takes too long to compute in practice.
Thus, under the assumption that $\pi_j$ converges towards an eigenvector or stationary distribution of $\mat{\Pi}$ for sufficiently large $j$, we look at the value of $\inp*{\abs*{\pi}, \abs*{\mat{\Pi} \mat{A} - \mat{A} \mat{P}} \cdot \mathbf{1}_n}$, where $\pi$ is an eigenvector with associated eigenvalue one.
We will now look at how such a $\pi$ may be determined.

Determining an eigenvector of the aggregated step matrix is already done in the Krylov-Schur variant of the Arnoldi iteration presented in~\cite{stewart2002krylovschur}.
This variation in itself is not directly suited to our application, and as such, we will not explore it further here.
Still, computing the so-called Schur decomposition of the aggregated step matrix to get its eigenvalues and -vectors is applicable to our situation too.
At first, we need to understand how this works:
\begin{definition}[QR decomposition]\label{def:qrdecomp}
Following~\cite[305]{beilina2017numlinalg}, let $\mat{M} \in \C^{n \times m}$ with $n \geq m$ and full column rank.
Then we call $\mat{Q} \in \C^{n \times m}$ with $\mat{Q} \mat{Q}\tranconj = \mat{I}_n$ and $\mat{R} \in \C^{m \times m}$ upper triangular, a \emph{QR decomposition} of $\mat{M}$, if, and only, if $\mat{M} = \mat{Q} \mat{R}$ holds.
\end{definition}
\begin{proposition}\label{prop:qrdecompworks}
For any given matrix $\mat{M} \in \C^{n \times m}$ with $n \geq m$ and full column rank, there exists a QR decomposition.
\end{proposition}
\begin{proof}
    Apply the Gram-Schmidt procedure from Algorithm~\ref{alg:classical-gs} or~\ref{alg:modified-gs} to the columns of $\mat{M}$, starting at the first column $c_1$.
    Note that we only defined the Gram-Schmidt procedure for real inputs, nonetheless, all proofs are easily transferred by substituting the transpose with the conjugate transpose and using the complex inner product $\inp{x, y} = x\tran \overline{y}$.
    By Proposition~\ref{prop:CGS-orthonorm-basis} and Corollary~\ref{cor:MGS-orthonorm-basis} respectively, we get an orthonormal basis $q_1, \dots, q_m$ of the space $\colsp(\mat{M}) \eqqcolon \spn\set{c_1, \dots, c_m}$.
    Then define
    \[
        \mat{Q} \coloneqq
        \pmat{
            | & & |\\
            q_1 & \cdots & q_m \\
            | & & |
        }\ \text{and}\ \mat{R} \coloneqq
        \pmat{
            \inp*{q_1, c_1} & \inp*{q_1, c_2} & \dots & \inp*{q_1, c_m}\\
            0 & \inp*{q_2, c_2} & \dots & \inp*{q_2, c_m}\\
            \vdots & \vdots & \ddots & \vdots\\
            0 & 0 & \dots & \inp*{q_m, c_m}
        }.
    \]
    Thus, we can see that each column $\mat{Q} \mat{R}(i)$ of $\mat{Q} \mat{R}$ evaluates to
    \begin{align*}
        \mat{Q} \mat{R}(i)& =
        \pmat{
            | & & |\\
            q_1 & \cdots & q_m \\
            | & & |
        }
        \pmat{
            \inp*{q_1, c_1} & \inp*{q_1, c_2} & \dots & \inp*{q_1, c_m}\\
            0 & \inp*{q_2, c_2} & \dots & \inp*{q_2, c_m}\\
            \vdots & \vdots & \ddots & \vdots\\
            0 & 0 & \dots & \inp*{q_n, c_m}
        }\!\!(i)\\
        &= \sum_{j = 1}^i \inp*{q_j, c_i} q_j = \sum_{j = 1}^{i - 1} \inp*{q_j, c_i} q_j + \inp*{q_i, c_i} q_i.
    \end{align*}
    Then, we additionally see that
    \begin{align*}
        1 &\txteq{Prop.~\ref{prop:CGS-orthonorm-basis}} \inp*{q_i, q_i} \txteq{Prop.~\ref{prop:closed-form-cgs}} \inp*{\frac{c_i - \sum_{j = 1}^{i - 1} \inp{q_j, c_i}q_j}{\norm{c_i - \sum_{j = 1}^{i - 1} \inp{q_j, c_i}q_j}_2}, q_i}\\
        &= \frac{1}{\norm{c_i - \sum_{j = 1}^{i - 1} \inp{q_j, c_i}q_j}_2} \inp*{c_i - \sum_{j = 1}^{i - 1} \inp{q_j, c_i}q_j, q_i}\\
        &= \frac{1}{\norm{c_i - \sum_{j = 1}^{i - 1} \inp{q_j, c_i}q_j}_2} \left( \inp*{c_i, q_i} - \sum_{j = 1}^{i - 1} \inp*{c_i, q_j}\inp*{q_j, q_i} \right)\\
        &\txteq{Prop.~\ref{prop:CGS-orthonorm-basis}} \frac{1}{\norm{c_i - \sum_{j = 1}^{i - 1} \inp{q_j, c_i}q_j}_2} \inp*{c_i, q_i}\\
        \iff & \inp*{q_i, c_i} = \norm*{c_i - \sum_{j = 1}^{i - 1} \inp{q_j, c_i}q_j}_2.
    \end{align*}
    Combining these two relations, gives us
    \[
        \mat{Q} \mat{R}(i) = \sum_{j = 1}^i \inp*{q_j, c_i} q_j = \sum_{j = 1}^{i - 1} \inp*{q_j, c_i} q_j + \norm*{c_i - \sum_{j = 1}^{i - 1} \inp{q_j, c_i}q_j}_2 q_i.
    \]
    Note that this is just the relation from Proposition~\ref{prop:closed-form-cgs} rearranged to equal $c_i$, showing that $\mat{M} = \mat{Q} \mat{R}$.
    Further, we also get that $\mat{Q} \mat{Q}\tranconj = \mat{I}_n$, using their construction through the Gram-Schmidt procedure.
\end{proof}
We will later need this QR decomposition to compute the following as well:
\begin{definition}[Schur decomposition]\label{def:schurdecomp}
Following~\cite[127]{beilina2017numlinalg}, let $\mat{M} \in \C^{n \times n}$, $\mat{U} \in \C^{n \times n}$ and $\mat{T} \in \C^{n \times n}$.
Denote the eigenvalues of $\mat{M}$ as $\lambda_1, \dots, \lambda_n$.
Further, $\mat{U}\tranconj \mat{U} = \mat{I}_n$ and $\mat{T}$ is an upper triangle matrix with $\lambda_1, \dots, \lambda_n$ on its diagonal along with
\[
    \mat{T} = \mat{U} \mat{M} \mat{U}\tranconj,
\]
we call this collectively a \emph{Schur decomposition} of $\mat{M}$.
\end{definition}
Now we can go back to tackling our original problem of computing eigenvalues of the aggregated step matrix $\mat{H}_j$.
This is done by computing a Schur decomposition of $\mat{H}_j$.
In describing this computation, we follow~\cite[63--64]{arbenz2016lecturenotes}.
\begin{algorithm}[H]
    \caption{QR Algorithm}
    \begin{algorithmic}[1]
        \State Let $\mat{M} \in \C^{n \times n}$
        \State $\mat{M}_0 \coloneqq \mat{M}$
        \State $\mat{U}_0 \coloneqq \mat{I}$
        \For{$k = 1, 2 \dots$}
            \State $\mat{M}_{k - 1} \eqqcolon \mat{Q}_k \mat{R}_k$\Comment{Compute QR decomposition of $\mat{M}_{k - 1}$}
            \State $\mat{M}_k \coloneqq \mat{R}_k\mat{Q}_k$
            \State $\mat{U}_k \coloneqq \mat{U}_{k - 1}\mat{Q}_k$
        \EndFor
        \State $\mat{T} \coloneqq \mat{M}_\infty$
        \State $\mat{U} \coloneqq \mat{U}_\infty$
        \State \Return $\mat{T}, \mat{U}$
    \end{algorithmic}
    \label{alg:qralgorithm}
\end{algorithm}
\begin{theorem}
    Considering the QR algorithm from Algorithm~\ref{alg:qralgorithm} with input $\mat{M} \in \C^{n \times n}$, the matrices $\mat{T}$ and $\mat{U}$ form a Schur decomposition of $\mat{M}$, if $\mat{M}$ has $n$ distinct eigenvalues.
\end{theorem}
\begin{proof}
    See~\cite[Ch. 4 \& 8]{arbenz2016lecturenotes}.
\end{proof}
Of course, Algorithm~\ref{alg:qralgorithm} is not directly applicable to real-word computations due to the results being defined through limits.
However, it can be shown that convergence is achieved in $\mathcal{O}(n^4)$ in the general case.
Even better, if the input is of the same form as $\mat{H}_j$, this can be reduced to $\mathcal{O}(n^3)$.
See~\cite[63--88, 145--161]{arbenz2016lecturenotes} for a more detailed read-up of convergence speed and improvements in the case of $\mat{H}_j$.
Additionally, the distinctness of all eigenvalues is solely necessary for convergence in finite time.
Still, this does not severely limit the applicability of Algorithm~\ref{alg:qralgorithm} due to numerous numerical \enquote{tricks} which can be employed to exploit properties of floating-point arithmetic to circumvent such problems in most cases.

Now that we can compute the Schur decomposition of our aggregated step matrix, we still do not know how we can get $\pi$ with eigenvalue approximately one.
Let $\mat{M}, \mat{U}, \mat{T} \in \C^{n \times n}$ be a Schur decomposition with $\mat{T} = \mat{U} \mat{M} \mat{U}\tranconj$.
Then, we can solve the system of linear equations given by
\[
    \mat{T}v = \lambda v \iff (\mat{T} - \lambda \mat{I}_n)v = 0
\]
to compute $v \in \R^n$ as an eigenvector of $\mat{T}$ with eigenvalue $\lambda \in \R$.
Such a solution must always exist, as every eigenvalue must have an associated eigenvector.
This only yields an eigenvector of $\mat{T}$.
But as $\mat{T}$ and $\mat{M}$ are similar, $\mat{U}\tranconj v$ is an eigenvector of $\mat{M}$ with eigenvalue $\lambda$.

In our specific use-case of finding an aggregated stationary distribution $\pi$ of $\mat{\Pi}$, we choose $\pi$ such that its corresponding eigenvalue $\lambda \in \C$ has real part closest to one with its imaginary part closest to zero.
In practice, we always find such an eigenvalue where the imaginary part is exactly zero and the real part is less than the machine epsilon away from one.

\subsection{Runtime and memory complexity}\label{subsec:runtime-and-memory-complexity}
Knowing an algorithm's runtime and memory complexity is important to see where its limits may lie.
This is no different in this case, as we want to know when using an aggregation instead of the original Markov chain is sensible.

To determine the runtime, we look at standard optimized implementations.
Such a standard is given by~\cite{blackford2002updated} for general matrices and vectors, while~\cite{duff2002sparse} provides the same for sparse matrices.
This special case is important as the transition matrices of large Markov chains often appear as sparse matrices.
Of course, there are many implementations of these de facto standards, all with minor differences in performance.
However, here, we are concerned only with asymptotic performance.
Thus, we can neglect performance gains made by improved data layouts or using dedicated hardware, explicitly made to compute such operations.

As a reference for runtime and memory complexity, we will use the quick reference guide for the individual operations~\cite{blas2024reference}.
Now, determining both the runtime and memory complexity becomes a simple task of looking what operations are done in each line of Algorithm~\ref{alg:arnoldi-iteration} and how often this line is repeated.
Taking $n$ to be the state space size of the original Markov chain and $j$ the size of the aggregated step matrix, this is shown in
\begin{table}[H]
    \centering
    \begin{tabular}{ccc||c|c|c|c}
        \hline
        \multicolumn{3}{l||}{Repetitions}                                                                                 & \multicolumn{1}{l|}{Line} & Operation           & \multicolumn{1}{l|}{Floating point operations} & \multicolumn{1}{l}{Allocations} \\ \hline\hline
        \multicolumn{3}{c||}{\multirow{3}{*}{1}}                                                                          & \multirow{3}{*}{2}        & 2-norm              & $2n$                                           & $n$                             \\
        \multicolumn{3}{c||}{}                                                                                            &                           & scale vector        & $n$                                            & $n$                             \\
        \multicolumn{3}{c||}{}                                                                                            &                           & copy vector         & 0                                              & $2n$                            \\ \hline\hline
        \multicolumn{1}{c|}{\multirow{2}{*}{}} & \multicolumn{2}{c||}{\multirow{2}{*}{$j$}}                               & \multirow{2}{*}{4}        & matrix-vector mult. & $2n^2$                                         & $n^2$                           \\
        \multicolumn{1}{c|}{}                  & \multicolumn{2}{c||}{}                                                   &                           & copy vector         & 0                                              & $2n$                            \\ \hline\hline
        \multirow{5}{*}{}                      & \multicolumn{1}{c|}{\multirow{5}{*}{}} & \multirow{5}{*}{$\frac{j}{2}$} & \multirow{2}{*}{6}        & dot product         & $2n$                                           & $2n$                            \\
        & \multicolumn{1}{c|}{}                  &                                &                           & copy vector         & 0                                              & $2n$                            \\ \cline{4-7}
        & \multicolumn{1}{c|}{}                  &                                & \multirow{3}{*}{7}        & scale vector        & $n$                                            & $n$                             \\
        & \multicolumn{1}{c|}{}                  &                                &                           & vector subtraction  & $n$                                            & $n$                             \\
        & \multicolumn{1}{c|}{}                  &                                &                           & copy vector         & 0                                              & $2n$                            \\ \hline\hline
        \multicolumn{1}{c|}{\multirow{4}{*}{}} & \multicolumn{2}{c||}{\multirow{4}{*}{$j$}}                               & \multirow{2}{*}{8}        & 2-norm              & $2n$                                           & $n$                             \\
        \multicolumn{1}{c|}{}                  & \multicolumn{2}{c||}{}                                                   &                           & copy scalar         & 0                                              & 1                               \\ \cline{4-7}
        \multicolumn{1}{c|}{}                  & \multicolumn{2}{c||}{}                                                   & \multirow{2}{*}{10}       & scale vector        & $n$                                            & $n$                             \\
        \multicolumn{1}{c|}{}                  & \multicolumn{2}{c||}{}                                                   &                           & copy vector         & $n$                                            & $2n$                            \\ \hline
    \end{tabular}
    \caption{Line-by-line breakdown of memory and runtime complexity of the Arnoldi iteration from Algorithm~\ref{alg:arnoldi-iteration} with number of repetitions for each line}
    \label{tab:complexity-arnoldi}
\end{table}
Now, we can add up the number of floating point operations and the number of allocations, while taking the number of repetitions as a multiplicative factor in mind.
This yields a runtime complexity of
\begin{align*}
    &\mathcal{O}\left( 1 \cdot (2n + n + 0) + j \cdot \left( 2n^2 + 0 + 2n + 0 + n + n + \frac{j}{2} \cdot (2n + 0 + n + n + 0) \right) \right)\\
    &= \mathcal{O}\left( 3n + j \cdot \left( 2n^2 + 4n + 2jn \right) \right) = \mathcal{O}\left( 2n^2 j + 4nj + 2j^2 n \right) = \mathcal{O}\left( n^2 j + j^2 n \right).
\end{align*}
Similarly, the memory complexity ends up being the same asymptotically:
\begin{align*}
    &\mathcal{O}\left( 1 \cdot (n + n + 2n) + j \cdot \left( n^2 + 2n + n + 1 + n + 2n + \frac{j}{2} \cdot (2n + 2n + n + n+ 2n) \right) \right)\\
    &= \mathcal{O}\left( 4n + j \cdot \left( n^2 + 6n + 1 + 4jn \right) \right) = \mathcal{O}\left( n^2 j + 6nj + j + 4j^2 n \right) = \mathcal{O}\left( n^2 j + j^2 n \right).
\end{align*}
Intuitively, this is no surprise: we need to orthogonalize each new vector $q_i$ against all previous vectors $q_1, \dots, q_{i - 1}$, resulting in the number of operations increasing quadratically.
The other summand $n^2 j$ is, as we have seen above, a result of the complexity of the matrix-vector product $q_\ell\tran \mat{P}$ being computed $i$ times.

However, as already mentioned, in the context of $\mat{P}$ being a transition matrix of a Markov chain, it often ends up being sparse, which enables us to decrease the runtime and memory complexity of computing the matrix-vector product to $\mathcal{O}(n)$.
If we then repeat the above calculations, we can see that both the runtime and memory complexity of the Arnoldi iteration are reduced to
\[\mathcal{O}\left( nj + j^2 n \right) = \mathcal{O}\left( j^2 n \right).\]
Thus, if the transition matrix is sparse, increasing the aggregation size is the main concern with regard to complexity.
As such, there will be some aggregation size $j$ with respect to $n$, where computing the aggregation and approximating the transient distribution will be slower than directly computing the transient distribution.
This problem is exacerbated by $\mat{H}_j$ not being sparse, so computing each approximated transient distribution is in $\mathcal{O}(j^2)$, while computing each transient distribution is in $\mathcal{O}(n)$ due to the sparseness of $\mat{P}$.

This analysis just determined the complexity of computing an Arnoldi aggregation of a given size, so it is not applicable to real-world applications.
There, we would usually not know the best size $j$ of the Arnoldi aggregation to minimize $\norm{e_k}_1$ while keeping $j$ as small as possible.
However, we have explored in Section~\ref{subsec:determining-convergence} how measuring the \enquote{quality} of an Arnoldi aggregation might be possible by using Algorithm~\ref{alg:qralgorithm} to compute $\inp{\abs{\pi}, \abs{\mat{H}_j\mat{Q}_j - \mat{Q}_j\mat{P}} \cdot \mathbf{1}_n}$.
Naively, this is done in each step of Algorithm~\ref{alg:arnoldi-iteration} after line ten.
So with the runtime and memory complexity of Algorithm~\ref{alg:qralgorithm} being in $\mathcal{O}(j^3)$ for our $\mat{H}_j$, we end up in $\mathcal{O}(n^2 j + j^2 n + j^4)$ for non-sparse $\mat{P}$ and in $\mathcal{O}(j^2 n + j^4)$ for sparse ones.

\section{Models and methodology}\label{sec:models-and-methodology}
All models we use for evaluating Arnoldi aggregations are \acp{CTMC}.
For \acp{CTMC}, the analogue to the transition matrix $\mat{P}$ in \acp{DTMC} is the \emph{generator matrix} $\mat{Q}$.
As we have defined the Arnoldi aggregation for \acp{DTMC} only, we transform a \ac{CTMC} into a \ac{DTMC} by following~\cite[361]{stewart2009markovchain}., where we set
$
\mat{P} \coloneqq \mat{I}_n + \frac{1}{\gamma} \mat{Q},
$
for a given generator matrix $\mat{Q} \in \R^{n \times n}$ and some $\gamma \geq \max_{1 \leq i \leq n}\abs{\mat{Q}(i, i)}$.
See~\cite[Sec. 9.10]{stewart2009markovchain} for a more in-depth coverage of \acp{CTMC}.

The first model we use for evaluation is a process algebra model of the \ac{RSVP} model from~\cite{wang2008rsvp}.
There, a wireless or mobile network is modelled through a lower network channel with capacity $N$, an upper network channel with capacity $M$ and several mobile nodesf which request resources from both channels.
Specifically, we use the model with $M = 7$, $N = 5$ and three mobile nodes because of its small size of 842 states and because of an exact aggregation with 234 states found in~\cite[42]{michel2025errbndmarkovaggr}.

Further, we use a model describing the dependability of interconnected workstation clusters from~\cite{haverkort2000cluster}.
Each workstation within a cluster can randomly break down and be repaired again by a repair unit.
For brevity, we will refer to this model as a cluster model from now on.
We use a Markov chain modelling two clusters and 20 workstations per cluster for evaluation.
This results in 15,540 states with an exact aggregation of 5523 states by~\cite[41]{michel2025errbndmarkovaggr}.
Again, this specific model is used for its exact aggregation and the larger state space size when compared to the \ac{RSVP} model above.

Another model is the Lotka-Volterra model pictured in~\cite{gillespie1977lotka}, which can be used to model chemical reactions or a population with prey and predators.
In contrast to the last two models, there is no known exact aggregation with a smaller state space size in this case~\cite[39]{michel2025errbndmarkovaggr}.
In our specific application, we set the maximum population size to 100, resulting in 10,021 states.
This model is interesting, as the lack of an exact aggregation enables us to compare how Arnoldi aggregations perform on Markov chains with exact aggregations and those without.

Lastly,~\cite[39]{michel2025errbndmarkovaggr} states the same reason for choosing a prokaryotic gene expression mode from~\cite{kierzek2001gene} as for the Lotka-Volterra model.
We will refer to this as a gene model for brevity.
Specifically, we chose a model with a population size of five, which results in our largest Markov chain with 43,957 states.
Thus, we can see the behaviour of Arnoldi aggregations for larger Markov chains while the computation time is still acceptable.

To examine and compare the properties of Arnoldi aggregations in practice to the theoretical properties, we primarily use the \ac{RSVP} model because of its exact aggregation and small size, which enable shorter computation times.
Of course, we will compare Arnoldi aggregation to another aggregation method across all models.

\section{Review}\label{sec:review}
Now that we have developed the theory of Arnoldi aggregations in Chapter~\ref{ch:finding-small-aggregations}, minimized numerical instabilities in Section~\ref{subsec:floating-point-errors}, proposed a measure in Section~\ref{subsec:determining-convergence} to determine when we can terminate the aggregation process and compiled models in Section~\ref{sec:models-and-methodology}, we can realize the experiments.
Firstly, in Section~\ref{subsec:validating-the-theory}, we will look to confirm Theorems~\ref{thrm:arnoldi-aggr-is-smallest-k-1-exact} and~\ref{thrm:arnoldi-aggr-is-smallest-exact}.
Additionally, we will look at how normalizing $\tilde{p}_k$ from Section~\ref{subsec:normalization-of-the-approximated-transient-distribution} performs in practice.
Section~\ref{subsec:comparison-to-exlump-aggregations} will compare the Arnoldi aggregation to another aggregation method regarding the error $\norm{e_k}_1$ and runtime.
\subsection{Validating the theory}\label{subsec:validating-the-theory}
To assess all further results properly, we need to check whether the floating-point errors predicted in Section~\ref{subsec:floating-point-errors} manifest as described there.
As our primary concern lies with neither the error term $\norm{\mat{H}_j \mat{Q}_j - \mat{Q}_j \mat{P}}_\infty$ nor the residual $h_{j, j + 1}$ converging to zero for an exact aggregation, we will use the RSVP model from~\cite{wang2008rsvp} with $M = 7$, $N = 5$, three mobile nodes and 842 states in total.
For proper $p_0$, this model oﬀers an exact aggregation at 234 states.
For each data point, we take a sample of 100 such $p_0$.
This results in
\begin{figure}[H]
    \centering
    \begin{tikzpicture}
        \begin{axis}[
        xlabel = {Aggregated state space size $j$},
        ylabel = {$\norm{\mat{H}_j \mat{Q}_j - \mat{Q}_j \mat{P}}_\infty$},
        xmin = 1, ymin = 0, xmax = 842, ymax = 5, plot,
        width=0.9\linewidth, height=6.5cm
        ]
        \addplot coordinates {
            (1,12.683130822619011) (41,3.773238111695366) (81,3.762928392392431) (121,4.296073226935646) (161,2.2086166263646607) (201,3.6898259210675524) (241,2.6871003363724686) (281,2.776307993487043) (321,3.427858057903539) (361,2.6487240229917774) (401,2.430279964366944) (441,1.8618755195217715) (481,2.6070753609152386) (521,2.8649811553414497) (561,2.6956933683381066) (601,1.802134791161413) (641,0.6525369917002242) (681,1.8064493564394717) (721,0.9953016601520975) (761,0.891840768968981) (801,0.19453626494281007) (841,0.00011986822927490978)
        };
        \addlegendentry{CGS2}
        \addplot coordinates {
            (1,12.667310183439318) (41,3.9110813248348713) (81,3.5563180536697288) (121,4.763492417111604) (161,2.525479771792804) (201,3.889021509192359) (241,2.6106358887820162) (281,2.659682140791283) (321,3.4146887661614143) (361,2.6333898618205462) (401,2.654535527912069) (441,1.9190759190989743) (481,2.6609659004080517) (521,2.6285567265867886) (561,2.7163445886155957) (601,1.6089661277095044) (641,0.7464013120225909) (681,1.6364188057861622) (721,1.0681824708471095) (761,0.86621469759044) (801,0.16142991219134573) (841,4.8997238016761255e-6)
        };
        \addlegendentry{MGS2}
        \addplot coordinates {
            (1,16.315534111541165) (9,4.8412653788963675) (17,4.383470586954201) (25,4.5346167420827745) (33,5.23198451401072) (41,3.9785207302835657) (49,4.30182887736379) (57,6.019495169620646) (65,4.444916971186841) (73,4.858479720508327) (81,12.534639838080098)
        };
        \addlegendentry{CGS}
        \end{axis}
    \end{tikzpicture}
    \caption{$\norm{\mat{H}_j\mat{Q}_j - \mat{Q}_j\mat{P}}_\infty$ for Arnoldi aggregations of the RSVP model}
    \label{fig:staticerrosrsvp}
\end{figure}
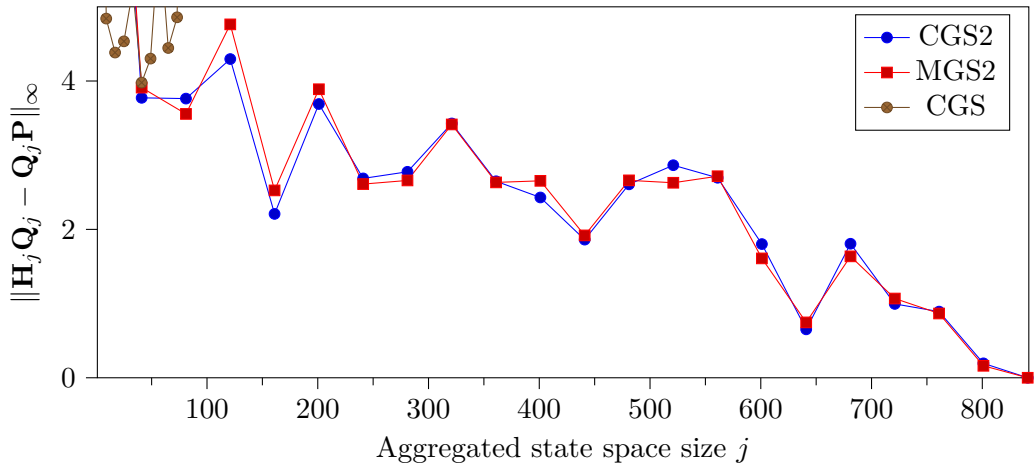
We have omitted \ac{CGSIR}, \ac{MGS} and \ac{MGSIR} as methods of orthonormalization as they barely differ from their variant \ac{CGS} and \ac{MGS2}.
Firstly, it is clear that \ac{CGS} is too unstable for aggregations exceeding 100 in size.
Further, the performance between \ac{CGS2} and \ac{MGS2} barely differs, with an average difference of $8.09\%$ from the larger value.
Thus, we will solely use \ac{CGSIR} for future experiments as it enables the best performance, comparatively.
It is also apparent from Figure~\ref{fig:staticerrosrsvp} that the Arnoldi aggregation, as predicted, fails the criterion of $\norm{\mat{H}_j \mat{Q}_j - \mat{Q}_j \mat{P}}_\infty = 0$ for an exact aggregation, which it must have found at $j \leq 234$ according to Theorem~\ref{thrm:arnoldi-aggr-is-smallest-exact}.
These numerical problems match our prediction from Section~\ref{subsec:floating-point-errors}.
Still, we will later see that an aggregation exists at with a state space size around $j = 234$ which is exact for all practical purposes.

In general, we want to see whether the attributes described in Lemma~\ref{lem:arnoldi-aggr-k0-exact} and~\ref{lem:arnoldi-aggr-k0-geq-n-exact} also hold in practice.
We continue to use the same RSVP model as above.
To first check $(j-1)$-exactness for an Arnoldi aggregation of size $j$ from Lemma~\ref{lem:arnoldi-aggr-k0-exact}, we will choose random $p_0$ at 1000 samples for each data point.
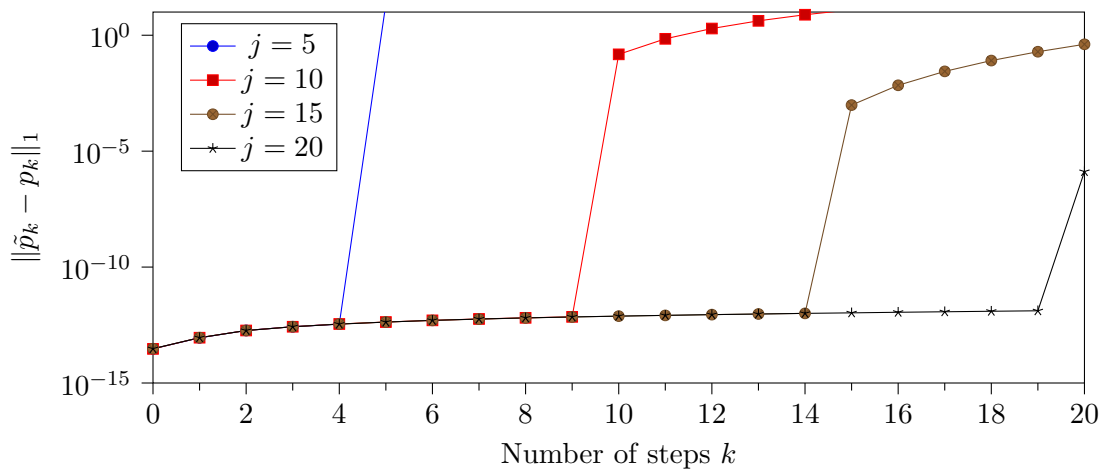
\begin{figure}[H]
    \centering
    \begin{tikzpicture}
        \begin{axis}[
        xlabel = {Number of steps $k$},
        ylabel = {$\norm{\tilde{p}_k - p_k}_1$},
        ymode=log, xmin = 0, ymin = 1e-15,
        xmax = 20, ymax = 10, plot, legend pos = north west,
        width=0.9\linewidth, height=6.5cm
        ]
        \addplot coordinates {
            (0,2.987466036213274e-14) (1,8.88744439607135e-14) (2,1.8467645406330135e-13) (3,2.673773887193835e-13) (4,3.476796704532174e-13) (5,22.718117231269844)
        };
        \addlegendentry{$j = 5$}
        \addplot coordinates {
            (0,2.961090466858519e-14) (1,8.993236880874735e-14) (2,1.842049130351002e-13) (3,2.6694891393987625e-13) (4,3.50963648152747e-13) (5,4.263810198261915e-13) (6,5.04259444548117e-13) (7,5.788226367282135e-13) (8,6.48084115314882e-13) (9,7.13287985385035e-13) (10,0.14999718776969456) (11,0.701175870155596) (12,1.9480741119205474) (13,4.182526999286293) (14,7.660702487880517) (15,12.587048804589584) (16,19.10808511459387) (17,27.318273013551583) (18,37.26382750973341) (19,48.94654342029476) (20,62.3353629372427)
        };
        \addlegendentry{$j = 10$}
        \addplot coordinates {
            (0,3.0213651685959904e-14) (1,9.039243042623956e-14) (2,1.8387090522465513e-13) (3,2.678500761702828e-13) (4,3.50304072449207e-13) (5,4.267340466981265e-13) (6,5.010857620934127e-13) (7,5.696452924542806e-13) (8,6.403265224005308e-13) (9,7.03952535157262e-13) (10,7.659470667650049e-13) (11,8.240957910349088e-13) (12,8.892292878755097e-13) (13,9.480772733279333e-13) (14,1.0065685541080732e-12) (15,0.0009709092928635403) (16,0.006936863756046161) (17,0.027612255561662174) (18,0.0808566773332774) (19,0.19437820059874678) (20,0.4064030101914703 20)
        };
        \addlegendentry{$j = 15$}
        \addplot coordinates {
            (0,3.010965916180315e-14) (1,8.994517946739052e-14) (2,1.8321746854460158e-13) (3,2.6627689944004616e-13) (4,3.4605518863875335e-13) (5,4.238016925780614e-13) (6,5.025471249516302e-13) (7,5.743064729902449e-13) (8,6.434094966450371e-13) (9,7.064668984610088e-13) (10,7.664307579676287e-13) (11,8.278055277025854e-13) (12,8.891937785528305e-13) (13,9.434192808876925e-13) (14,1.0079269993580846e-12) (15,1.060631188681854e-12) (16,1.1242686249547543e-12) (17,1.1817561648597244e-12) (18,1.243682607572495e-12) (19,1.302104218848234e-12) (20,1.3175133899176895e-6)
        };
        \addlegendentry{$j = 20$}
        \end{axis}
    \end{tikzpicture}
    \caption{$\norm{e_k}_1$ for differing $j$ and $0 \leq k \leq 20$ in Arnoldi aggregations of a RSVP model}
    \label{fig:dynamicerrorrsvp-exact-k-exact}
\end{figure}
There is a clear loss of accuracy for $k > j - 1$ steps, indeed indicating $(j-1)$-exactness.
Note that for larger $j$, the errors $\norm{e_k}_1$ for $k \leq j - 1$ are also of the same magnitude as here, but with a lesser loss of accuracy after those initial steps.

To see whether the exactness from Lemma~\ref{lem:arnoldi-aggr-k0-geq-n-exact} holds as well, we will look at larger aggregation sizes up to around $j = 234$, as we know there is an exact aggregation of this size from~\cite[40]{michel2025errbndmarkovaggr}.
Thus, we look at (approximated) transient distributions up to $k = 10^6$.
Note that for any number of steps $k$, we cannot expect $\norm{e_k}_1$ to be less than $k \cdot \macheps$, where $\macheps$ is the rounding machine epsilon.
This is because of the errors accumulating with each further step.
For IEEE-754 64-bit floating-point numbers, it is $\macheps \approx 1.11 \cdot 10^{-16}$.
At the same time, we will examine how the criterion for predicting exact aggregations from Section~\ref{subsec:determining-convergence} performs in reality.
We will use random $p_0$, but we also choose $p_0$ in such a way that it is ensured that there is an exact aggregation with 234 states.
Due to the considerable computational effort required to compute $e_{10^6}$, we take ten samples per data point.
\begin{figure}[H]
    \centering
    \begin{tikzpicture}
        \begin{axis}[
            xlabel = {Aggregated state space size},
            ylabel = {Error},
            ymode=log, xmin = 1, ymin = 1e-14, domain=0:300, legend style={cells={align=left}},
            xmax = 300, ymax = 1, plot, legend pos = north east,
            width=0.9\linewidth, height=13cm
        ]
            \addplot coordinates {
                (1,0.9999999999996783) (14,1.0590014209731492e93) (27,1.093927771271882e9) (40,12.935893732569744) (53,1.601882559284899) (66,1.0263999790864988) (79,4.440360703614884e9) (92,0.10823023233173526) (105,0.006656381653855471) (118,0.0011022447661217093) (131,0.00012873844936988325) (144,2.868205016207439e-6) (157,1.3213192209434936e-10) (170,1.3114389255293993e-12) (183,1.3070378506236906e-12) (196,1.303811571017354e-12) (209,1.305917127743144e-12) (222,1.3076426473347334e-12) (235,1.3091230464697159e-12) (248,1.30885011763483e-12) (261,1.308850117634904e-12) (274,1.308850117634904e-12) (287,1.308850117634904e-12) (300,1.308850117634904e-12)
            };
            \addlegendentry{exact $p_0$, $\norm{e_{10^4}}_1$}
            \addplot coordinates {
                (1,0.9999999999996796) (14,8.682901937131364e17) (27,2465.614499136977) (40,1.499933270633242) (53,1.0019176676563952) (66,0.641075662678966) (79,0.4098992883744737) (92,3.9380877553089775) (105,0.008280336890477236) (118,0.0037067048536584325) (131,0.0004946966456049212) (144,1.6735135945241204e-5) (157,1.8784141269801397e-7) (170,1.6130074700486288e-9) (183,1.276130098916059e-10) (196,1.3477150728000954e-12) (209,1.3494092989391435e-12) (222,1.3454557357461826e-12) (235,1.3445997750046378e-12) (248,1.3382083761360866e-12) (261,1.347495385878242e-12) (274,1.34417131793331e-12) (287,1.34417131793331e-12) (300,1.34417131793331e-12)
            };
            \addlegendentry{non-exact $p_0$, $\norm{e_{10^4}}_1$}
            \addplot coordinates {
                (1,0.9999999999971305) (14,NaN) (27,NaN) (40,2.6210701366959337e127) (53,7.382732221673995e65) (66,3.47238182454086e242) (79,4.1361829948267124e52) (92,NaN) (105,1049.964897942284) (118,0.6778038001260972) (131,0.5982002267321362) (144,0.14791531930119853) (157,2.1103031129639635e-6) (170,3.0179119464193626e-10) (183,2.0514197359653264e-10) (196,3.499224377307924e-10) (209,2.6997817469993446e-10) (222,1.4807839176906262e-10) (235,9.612420811688927e-11) (248,8.233896140165956e-11) (261,7.276852916613362e-11) (274,8.873557096644478e-11) (287,6.941767147831942e-11) (300,1.0133929029448865e-10)
            };
            \addlegendentry{exact $p_0$, $\norm{e_{10^6}}_1$}
            \addplot coordinates {
                (1,0.9999999999971305) (14,NaN) (27,NaN) (40,8.982260194890708e107) (53,1.504391780018984e73) (66,2.773861182232495e21) (79,9.616735326195112e14) (92,NaN) (105,154368.8085179561) (118,NaN) (131,31.16255372603714) (144,0.1160158113476204) (157,0.00195668393705069) (170,1.734235506987551e-5) (183,5.459926185828008e-6) (196,2.4494770614446536e-8) (209,1.709799700882245e-10) (222,1.4506467556195953e-10) (235,1.6746438345988565e-10) (248,2.2211330033513035e-10) (261,2.321736448038248e-10) (274,2.368279248713506e-10) (287,2.4114263538354774e-10) (300,2.3142362550828243e-10)
            };
            \addlegendentry{non-exact $p_0$, $\norm{e_{10^6}}_1$}
            \addplot coordinates {
                (1,0.7041328695662974) (14,0.12342073968442469) (27,0.02532684971642423) (40,0.007408188199569539) (53,0.002666716652444382) (66,0.0014533531651663705) (79,0.002564964200176851) (92,0.0005946063531233087) (105,0.00015526454138921808) (118,0.00023040410830026) (131,5.091677994027011e-5) (144,3.164163069191886e-6) (157,9.429737820802962e-11) (170,1.7513942952862348e-14) (183,1.6594556887428333e-14) (196,1.6161132854087552e-14) (209,1.57220080134689e-14) (222,1.562629761697312e-14) (235,1.5613699267473614e-14) (248,1.5612424595228052e-14) (261,1.5612177757245902e-14) (274,1.5612177760215623e-14) (287,1.5612177758025978e-14) (300,1.5612177758066687e-14)
            };
            \addlegendentry{exact $p_0$,\\ $\inp*{\abs{\pi}, \abs{\mat{H}_j \mat{Q}_j - \mat{Q}_j \mat{P}} \cdot \mathbf{1}_n}$}
            \addplot coordinates {
                (1,0.61271435355893) (14,0.13008844719682303) (27,0.019734373894162917) (40,0.006051677593685434) (53,0.002347030137527128) (66,0.0007684477820357159) (79,0.001361972498503312) (92,0.003775724460595941) (105,0.00021692615194793293) (118,0.0009337011778195679) (131,0.00018333597018766514) (144,4.3384036873767105e-6) (157,9.65952509301048e-8) (170,9.021946924875666e-10) (183,7.38542836959882e-10) (196,1.161515835346837e-12) (209,1.5337653260138767e-14) (222,1.5254605832760032e-14) (235,1.5238685743953203e-14) (248,1.522610691894044e-14) (261,1.522603265192457e-14) (274,1.52259987841452e-14) (287,1.5225998772688912e-14) (300,1.5225998773445174e-14)
            };
            \addlegendentry{non-exact $p_0$,\\ $\inp*{\abs{\pi}, \abs{\mat{H}_j \mat{Q}_j - \mat{Q}_j \mat{P}} \cdot \mathbf{1}_n}$}
            \addplot+[no marks, gray]
                {1.11 * 1e-12}
            node[pos=0.25, above] {$10^4 \cdot \macheps$};
            \addplot+[no marks, gray]
                {1.11 * 1e-10}
            node[pos=0.25, above] {$10^6 \cdot \macheps$};
        \end{axis}
    \end{tikzpicture}
    \caption{Different error terms in relation to $j$ of Arnoldi aggregations of a RSVP model}
    \label{fig:findingexactconvcriterion}
\end{figure}
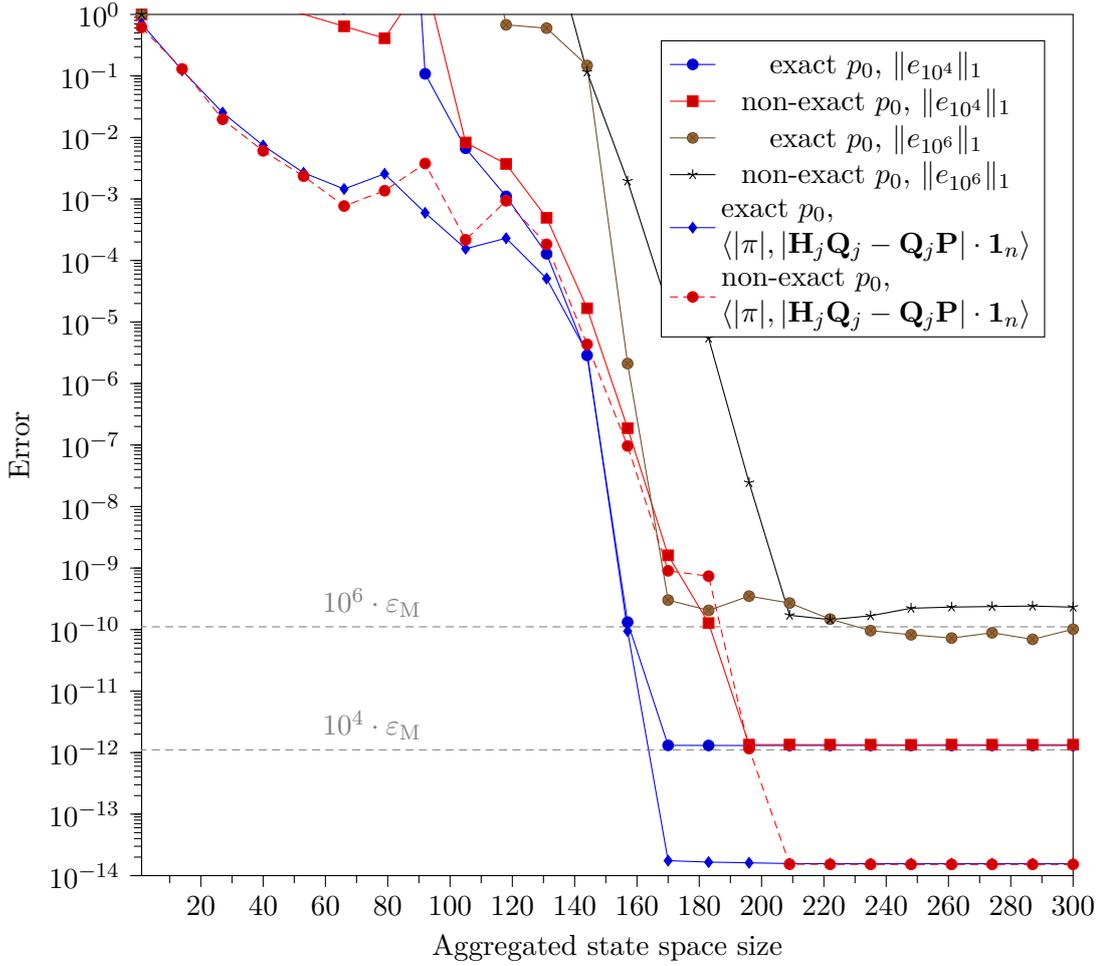
For $j \geq 175$, no further improvement in $e_k$ is made for $p_0$, which enable an exact aggregation, as we have reached machine precision.
Note that this does not necessarily mean we have found an exact aggregation because we would need to consider all possible $k$ to confirm this.
We reach an equally low error for random $p_0$, but only after $j \geq 215$.
Furthermore, we can see that $\inp{\abs{\pi}, \abs{\mat{H}_j \mat{Q}_j - \mat{Q}_j \mat{P}} \cdot \mathbf{1}_n}$ does indeed plateau at the same time as $\norm{e_k}_1$ for $k = 10^6$, hinting that we can at least predict whether the aggregation is exact in the first $10^6$ steps.
We now take a closer look at this heuristic.

Firstly, we need to confirm that the behaviour from Figure~\ref{fig:findingexactconvcriterion} is not exclusive to the RSVP model with its exact aggregations.
We will use the workstation cluster model, which models 20 workstations per cluster, resulting in a \ac{DTMC} with 15,540 states and an exact aggregation for fitting $p_0$ with 5523 states~\cite[40]{michel2025errbndmarkovaggr}.
Additionally, we will use the Lotka-Volterra model with 10,201 states.
For the Lotka-Volterra model, there is no known exact aggregation.

Although not plotted here, we achieve very similar results as in Figure~\ref{fig:findingexactconvcriterion}.
For the workstation cluster model, $\norm{e_{10^6}}_1$ and $\inp{\abs{\pi}, \abs{\mat{H}_j\mat{Q}_j - \mat{Q}_j\mat{P}} \cdot \mathbf{1}_n}$ cease to change significantly at $j \approx 400$, reducing the state space size 38-fold for random $p_0$ and for $p_0$ allowing an exact aggregation in~\cite[40]{michel2025errbndmarkovaggr}.
The result is similar for the Lotka-Volterra model, where no improvement to $\norm{e_{10^6}}_1$ and $\inp{\abs{\pi}, \abs{\mat{H}_j\mat{Q}_j - \mat{Q}_j\mat{P}} \cdot \mathbf{1}_n}$ is made after $j \approx 750$, giving us a 13-fold reduction in the state space size.
Here, we only chose random $p_0$.
For all three models, $\norm{e_{10^4}}_1 \approx 10^{-12}$, $\norm{e_{10^6}}_1 \approx 10^{-9}$ and $\inp{\abs{\pi}, \abs{\mat{H}_j\mat{Q}_j - \mat{Q}_j\mat{P}} \cdot \mathbf{1}_n} \approx 10^{-14}$ all roughly hold true once they stop improving with increased state space size.

Firstly, this shows us that we can find aggregations that perform like exact ones for up to $10^6$ steps, even if the formal condition for exactness is not fulfilled due to numerical errors.
Importantly, we have only considered $\norm{e_k}_1$ for up to $k = 10^6$, so larger inaccuracies might still arise for larger $k$.
Thus, the heuristic of $\inp{\abs{\pi}, \abs{\mat{H}_j\mat{Q}_j - \mat{Q}_j\mat{P}} \cdot \mathbf{1}_n}$ seemingly indicates that an aggregation will no longer improve when $j$ is increased.
So, while we cannot confirm Lemma~\ref{lem:arnoldi-aggr-k0-geq-n-exact}, we have found a useful heuristic for predicting when an Arnoldi aggregation will behave similarly to an exact one.
This heuristic is given through the rate of change of $\inp{\abs{\pi}, \abs{\mat{H}_j\mat{Q}_j - \mat{Q}_j\mat{P}} \cdot \mathbf{1}_n}$, with respect to $j$, nearing zero.
Additionally, we get $\tilde{p}$ through this heuristic.
Although we have again not shown it here, $\norm{\tilde{p}\tran - \tilde{p}\tran \mat{P}}_1$ ceases to change along with $\inp{\abs{\pi}, \abs{\mat{H}_j\mat{Q}_j - \mat{Q}_j\mat{P}} \cdot \mathbf{1}_n}$ at $\norm{\tilde{p}\tran - \tilde{p}\tran \mat{P}}_1 \approx 10^{-14}$.
Thus, $\tilde{p}$ provides a good approximation of a stationary distribution.

Further, we know from Figure~\ref{fig:staticerrosrsvp} that the error bound from Proposition~\ref{prop:general-error-bound} will yield inadequate bounds for the error.
As such, we now want to check how the bound from Proposition~\ref{prop:specific-error-bound} performs in practice, by looking at the difference between the error bound and the actual error.
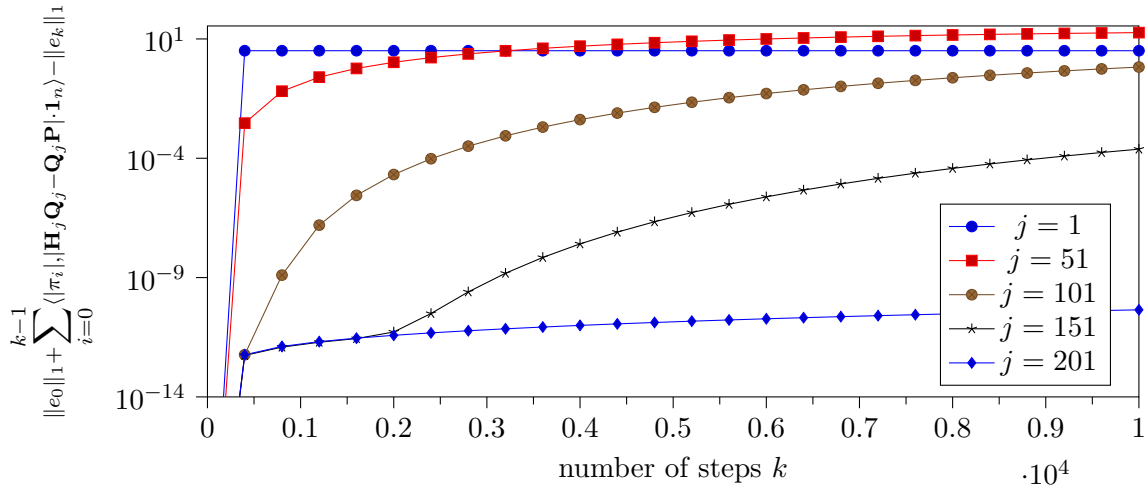
\begin{figure}[H]
    \centering
    \begin{tikzpicture}
        \begin{axis}[
        xlabel = {number of steps $k$},
        ylabel = {${\scriptstyle \norm{e_0}_1 + {\displaystyle\sum_{i = 0}^{k - 1}} \inp{\abs{\pi_i}, \abs{\mat{H}_j \mat{Q}_j - \mat{Q}_j \mat{P}} \cdot \mathbf{1}_n} - \norm{e_k}_1}$},
        ymode=log, xmin = 0, ymin = 1e-14,
        xmax = 10000, ymax = 35, plot, legend pos = south east,
        width=0.9\linewidth, height=6.5cm
        ]
        \addplot coordinates {
            (0.0,1e-25) (400.0,3.209570959644669) (800.0,3.209570959644692) (1200.0,3.209570959644715) (1600.0,3.2095709596447364) (2000.0,3.209570959644758) (2400.0,3.2095709596447786) (2800.0,3.2095709596447968) (3200.0,3.2095709596448145) (3600.0,3.2095709596448305) (4000.0,3.209570959644846) (4400.0,3.20957095964486) (4800.0,3.2095709596448723) (5200.0,3.2095709596448847) (5600.0,3.2095709596448962) (6000.0,3.209570959644905) (6400.0,3.2095709596449136) (6800.0,3.209570959644922) (7200.0,3.209570959644929) (7600.0,3.209570959644936) (8000.0,3.2095709596449424) (8400.0,3.2095709596449487) (8800.0,3.209570959644954) (9200.0,3.2095709596449593) (9600.0,3.2095709596449638) (10000.0,3.209570959644968)
        };
        \addlegendentry{$j = 1$}
        \addplot coordinates {
            (0.0,1e-25) (400.0,0.002956773795413428) (800.0,0.0644970416092984) (1200.0,0.2505674039780908) (1600.0,0.5811865822684588) (2000.0,1.0529858713765394) (2400.0,1.6523648382042457) (2800.0,2.36173564165444) (3200.0,3.1625157983145664) (3600.0,4.036558852827591) (4000.0,4.96682837797443) (4400.0,5.937704913112175) (4800.0,6.93509864116552) (5200.0,7.94646629955426) (5600.0,8.960775418720148) (6000.0,9.96843272822489) (6400.0,10.961200522956757) (6800.0,11.932101961617168) (7200.0,12.875322171532597) (7600.0,13.786107779000197) (8000.0,14.660667088238224) (8400.0,15.496075447430728) (8800.0,16.2901829034211) (9200.0,17.041524625738184) (9600.0,17.749239416672076) (10000.0,18.412992730938583)
        };
        \addlegendentry{$j = 51$}
        \addplot coordinates {
            (0.0,1e-25) (400.0,5.85482792827059e-13) (800.0,1.2671731562957077e-9) (1200.0,1.573962653869359e-7) (1600.0,2.7894162141880283e-6) (2000.0,2.08111887832645e-5) (2400.0,9.549122187461983e-5) (2800.0,0.0003212747063261207) (3200.0,0.000872156391883284) (3600.0,0.002023982976237758) (4000.0,0.004168686533432121) (4400.0,0.007818926295075367) (4800.0,0.01360319252293028) (5200.0,0.022253141188527496) (5600.0,0.03458454954937572) (6000.0,0.05147483912952013) (6400.0,0.0738389376539302) (6800.0,0.10260533305225238) (7200.0,0.13869361199423724) (7600.0,0.18299441667819022) (8000.0,0.23635213538605332) (8400.0,0.2995517681030016) (8800.0,0.3733076443515429) (9200.0,0.45825608508155574) (9600.0,0.5549506725270272) (10000.0,0.6638599989955146)
        };
        \addlegendentry{$j = 101$}
        \addplot coordinates {
            (0.0,1e-25) (400.0,5.449578536565926e-13) (800.0,1.2199885668596097e-12) (1200.0,1.949762767136002e-12) (1600.0,2.78289068163491e-12) (2000.0,5.239164634587024e-12) (2400.0,3.127888356385234e-11) (2800.0,2.510916577974461e-10) (3200.0,1.519940257058623e-9) (3600.0,6.979077705256555e-9) (4000.0,2.5790441780776303e-8) (4400.0,8.035615171372783e-8) (4800.0,2.183860364023512e-7) (5200.0,5.308083096341535e-7) (5600.0,1.1758815186676715e-6) (6000.0,2.4090335035793576e-6) (6400.0,4.617055950008738e-6) (6800.0,8.354912983439538e-6) (7200.0,1.4382534129408111e-5) (7600.0,2.3699092618401133e-5) (8000.0,3.7572494233649685e-5) (8400.0,5.756197033160351e-5) (8800.0,8.553231133185647e-5) (9200.0,0.0001236588482636703) (9600.0,0.0001744229058559644) (10000.0,0.0002405980102135506)
        };
        \addlegendentry{$j = 151$}
        \addplot coordinates {
            (0.0,1e-25) (400.0,5.822089914417355e-13) (800.0,1.3031673085106596e-12) (1200.0,2.0710842867026534e-12) (1600.0,2.893872269876558e-12) (2000.0,3.810489985916907e-12) (2400.0,4.815254314601078e-12) (2800.0,5.923887692737385e-12) (3200.0,7.143244247726137e-12) (3600.0,8.468399720887725e-12) (4000.0,9.920650525085925e-12) (4400.0,1.1471891769512638e-11) (4800.0,1.3135124365725011e-11) (5200.0,1.4903985985680947e-11) (5600.0,1.6781167897320624e-11) (6000.0,1.877881860770462e-11) (6400.0,2.087157814928798e-11) (6800.0,2.307738916996836e-11) (7200.0,2.5386301121092934e-11) (7600.0,2.7800025732967822e-11) (8000.0,3.0332348231810974e-11) (8400.0,3.2969619820755736e-11) (8800.0,3.571001643189153e-11) (9200.0,3.8548970349921005e-11) (9600.0,4.149600570121275e-11) (10000.0,4.454242392815773e-11)
        };
        \addlegendentry{$j = 201$}
        \end{axis}
    \end{tikzpicture}
    \caption{Error bounds in the RSVP model for different $j$ and $k$}
    \label{fig:rsvpbndcmp}
\end{figure}
For larger $j$, the approximations get much better.
Additionally, as can be seen through the aggregations with $j = 151$ and $j = 201$ for $1 \leq k \leq 0.2 \cdot 10^4$, the approximation from Proposition~\ref{prop:specific-error-bound} stops changing up to some $k$ for large enough $j$.
This invariance is also observed for smaller and larger $k$ and $j$ on all different models from Section~\ref{sec:models-and-methodology}.

As a last point, we will look at how normalizing $\tilde{p}_k$ can help to increase its accuracy.
Once again, we will use the same RSVP model as before, with random $p_0$ and $j = 150$ and a sample size of 1000 per data point.
\begin{figure}[H]
    \centering
    \begin{tikzpicture}
        \begin{axis}[
        xlabel = {number of steps $k$},
        ylabel = {$\norm{\tilde{p}_k - p_k}_1$},
        ymode=log, xmin = 0, ymin = 1e-3,
        xmax = 200000, ymax = 0.1, plot, legend pos = south east,
        width=0.9\linewidth, height=6.5cm
        ]
        \addplot coordinates {
            (0,0.0) (10000,2.537279813962671e-5) (20000,0.0013320738429553863) (30000,0.005575035868948396) (40000,0.011572673416499573) (50000,0.018251607635968913) (60000,0.025188869014703328) (70000,0.032242467201582234) (80000,0.039366691864927306) (90000,0.04654713986699468) (100000,0.053852021566231686) (110000,0.061062441261347244) (120000,0.06839600582490436) (130000,0.07578033116130624) (140000,0.08321571909369196) (150000,0.09070251826089082) (160000,0.09824106772513716) (170000,0.10583172357249224) (180000,0.1134748454506605) (190000,0.12117079583785485) (200000,0.1289199398275935)
        };
        \addlegendentry{Never normalize $\tilde{p}_k$}
        \addplot coordinates {
            (0,0.0) (10000,4.298194306911224e-5) (20000,0.002695751446101786) (30000,0.01049857234732662) (40000,0.01756877090274203) (50000,0.020570403878014013) (60000,0.020199532297681923) (70000,0.01795297078985981) (80000,0.021333775779624083) (90000,0.025259632401285904) (100000,0.02920134298877995) (110000,0.03315840029971157) (120000,0.037130719493092976) (130000,0.04111832355149234) (140000,0.045121262532450856) (150000,0.04913959330275832) (160000,0.05317337454570795) (170000,0.0572226713721627) (180000,0.06128753808020577) (190000,0.06536803403983552) (200000,0.06946421934342437)
        };
        \addlegendentry{Normalize $\tilde{p}_k$ as in Section~\ref{subsec:normalization-of-the-approximated-transient-distribution}}
        \addplot coordinates {
            (0,0.0) (10000,7.367624040086022e-5) (20000,0.005085910908381393) (30000,0.020204482438770013) (40000,0.03482554969193824) (50000,0.043566977258760554) (60000,0.04769999205821289) (70000,0.04941580187626979) (80000,0.05007424844312767) (90000,0.05031452654713874) (100000,0.05039932690074875) (110000,0.05042858347644811) (120000,0.05043852050414913) (130000,0.050441859143748144) (140000,0.0504429723830407) (150000,0.050443341619568934) (160000,0.050443463633786455) (170000,0.05044350384894726) (180000,0.05044351707985791) (190000,0.05044352142728976) (200000,0.05044352285454105)
        };
        \addlegendentry{Always normalize $\tilde{p}_k$}
        \end{axis}
    \end{tikzpicture}
    \caption{Normalization of $\tilde{p}_k$ under different conditions.}
    \label{fig:normtest}
\end{figure}
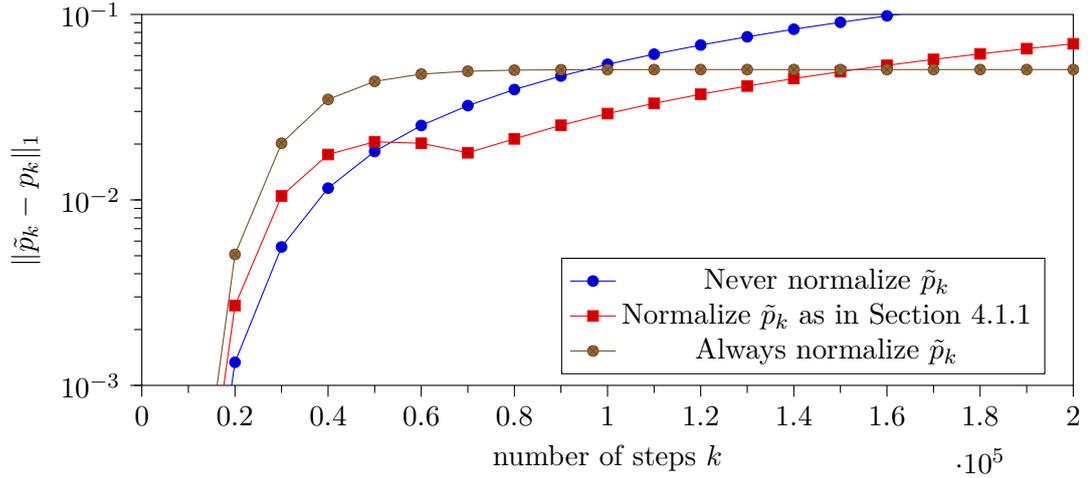
We can already see that the propositions in Section~\ref{subsec:normalization-of-the-approximated-transient-distribution} improve $\tilde{p}_k$, especially for large $k$.
On the other hand, for smaller $k$, we have a worsening effect, either caused by floating-point errors or by some of the postulated conjectures being false.
Additionally, the three conjectures are essential, as the two proven theorems can be applied only very rarely in practice.
The same behaviour is observed again to a more extreme degree for $\tilde{p}_k$ if we choose to normalize it in all cases.
Although not shown here, the improvements made by always normalizing are much more pronounced when looking at $k > 2 \cdot 10^5$, as the error stays at around $0.05$, while the error term slowly diverges in the other two cases.
Further, $\tilde{p}_k$ does not gain or lose meaningful accuracy after normalization for larger state space sizes because $\tilde{p}_k$ is a good approximation of $p_k$; thus, $\norm{\tilde{p}_k}_1$ is close to one.
Thus, normalizing $\tilde{p}_k$ is mainly useful if it is less accurate due to using a small aggregation.
Nonetheless, continuing from here, we will never normalize $\tilde{p}_k$ because sometimes the error worsens, and we want to compare Arnoldi aggregations to other aggregations where normalization does not occur either.

\subsection{Comparison to Exlump aggregations}\label{subsec:comparison-to-exlump-aggregations}
In Section~\ref{subsec:validating-the-theory}, we confirm that our implementation of the Arnoldi aggregation behaves as expected.
Thus, we now compare it to the Exlump Algorithm from~\cite[Alg. 3]{michel2025errbndmarkovaggr}.
There, a stricter definition of aggregation is used, where we partition the original state space into the aggregated one.
Intuitively, we can view the disaggregation matrix $\mat{A}$ as a matrix where $\mat{A}(i, j)$ is the probability that we are in the $j$-th state under the assumption that we are in the $i$-th state of the aggregated state space.
Different ways of choosing such an $\mat{A}$ are shown in~\cite[4, 39]{michel2025errbndmarkovaggr}.
We choose $\mat{A}(i, j)$ as the reciprocal of the size of the $i$-th aggregate and will call Exlump aggregations with such $\mat{A}$ \emph{uniform}.
Furthermore, Exlump aggregations are computed independent of a starting vector $p_0$, making it perform differently for different $p_0$.
Specifically, if $p_0$ is uniform among the states within every aggregate, the initial error $\norm{e_0}_1$ is especially small.
We call such $p_0$ \emph{exact}.
We start by comparing the error made by both aggregation methods for the RSVP, cluster, and gene model.
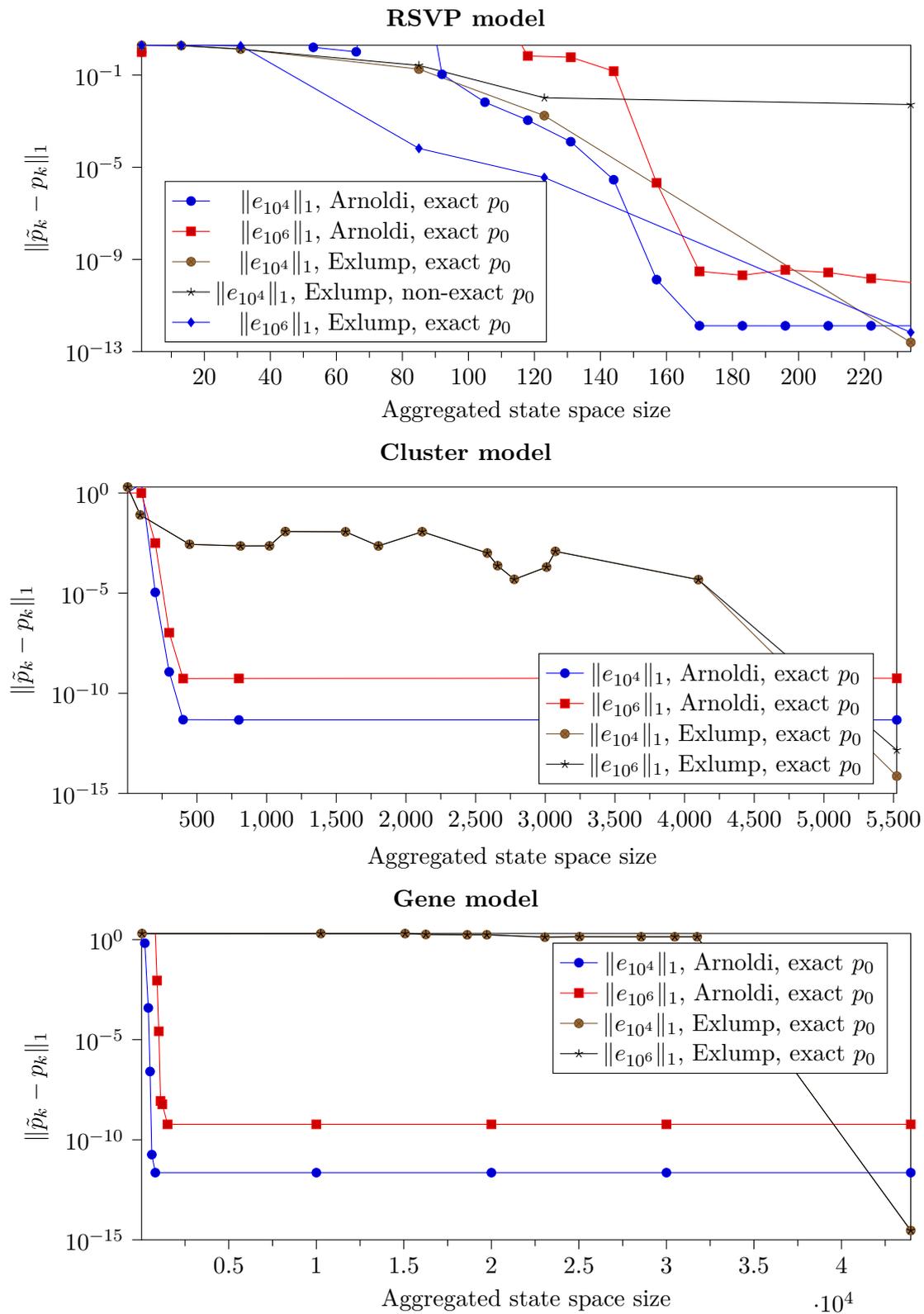
\begin{figure}[H]
    \centering
    \textbf{RSVP model}\par\medskip
    \begin{tikzpicture}
        \begin{axis}[
        xlabel = {Aggregated state space size},
        ylabel = {$\norm{\tilde{p}_k - p_k}_1$},
        ymode=log, xmin = 1, ymin = 1e-13, domain=0:300, legend style={cells={align=left}},
        xmax = 234, ymax = 2, plot, legend pos = south west,
        width=0.9\linewidth, height=6.5cm
        ]
        \addplot coordinates {
            (1,0.9999999999996783) (14,1.0590014209731492e93) (27,1.093927771271882e9) (40,12.935893732569744) (53,1.601882559284899) (66,1.0263999790864988) (79,4.440360703614884e9) (92,0.10823023233173526) (105,0.006656381653855471) (118,0.0011022447661217093) (131,0.00012873844936988325) (144,2.868205016207439e-6) (157,1.3213192209434936e-10) (170,1.3114389255293993e-12) (183,1.3070378506236906e-12) (196,1.303811571017354e-12) (209,1.305917127743144e-12) (222,1.3076426473347334e-12) (235,1.3091230464697159e-12) (248,1.30885011763483e-12) (261,1.308850117634904e-12) (274,1.308850117634904e-12) (287,1.308850117634904e-12) (300,1.308850117634904e-12)
        };
        \addlegendentry{$\norm{e_{10^4}}_1$, Arnoldi, exact $p_0$}
        \addplot coordinates {
            (1,0.9999999999971305) (14,NaN) (27,NaN) (40,2.6210701366959337e127) (53,7.382732221673995e65) (66,3.47238182454086e242) (79,4.1361829948267124e52) (92,NaN) (105,1049.964897942284) (118,0.6778038001260972) (131,0.5982002267321362) (144,0.14791531930119853) (157,2.1103031129639635e-6) (170,3.0179119464193626e-10) (183,2.0514197359653264e-10) (196,3.499224377307924e-10) (209,2.6997817469993446e-10) (222,1.4807839176906262e-10) (235,9.612420811688927e-11) (248,8.233896140165956e-11) (261,7.276852916613362e-11) (274,8.873557096644478e-11) (287,6.941767147831942e-11) (300,1.0133929029448865e-10)
        };
        \addlegendentry{$\norm{e_{10^6}}_1$, Arnoldi, exact $p_0$}
        \addplot coordinates {
            (1,1.949928382850986) (13,1.9050753657561208) (31,1.3513770222377024) (85,0.18207581740237608) (123,0.001733212465603741) (234,2.483729145080202e-13)
        };
        \addlegendentry{ $\norm{e_{10^4}}_1$, Exlump, exact $p_0$}
        \addplot coordinates {
            (1,1.9498808980075593) (13,1.9098973279096165) (31,1.3409567392481854) (85,0.26378462546805925) (123,0.010335033836082565) (234,0.005262650500545417)
        };
        \addlegendentry{$\norm{e_{10^4}}_1$, Exlump, non-exact $p_0$}
        \addplot coordinates {
            (1,1.9776035540749521) (13,1.9724002080337637) (31,1.8633400696449895) (85,6.583851230470933e-5) (123,3.5792770755141195e-6) (234,6.804311490889649e-13)
        };
        \addlegendentry{$\norm{e_{10^6}}_1$, Exlump, exact $p_0$}
        \end{axis}
    \end{tikzpicture}
    \textbf{Cluster model}\par\medskip
    \begin{tikzpicture}
        \begin{axis}[
        xlabel = {Aggregated state space size},
        ylabel = {$\norm{\tilde{p}_k - p_k}_1$},
        ymode=log, xmin = 3, ymin = 1e-15, legend style={cells={align=left}},
        xmax = 5523, ymax = 2, plot, legend pos = south east,
        width=0.9\linewidth, height=6.5cm
        ]
        \addplot coordinates {
            (1,0.9999999999997736) (101,2.9960676944945974) (201,1.0974347226627747e-5) (301,1.1665471161805482e-9) (401,4.7764182672623035e-12) (801,4.655221907811145e-12) (3000,4.655221907811145e-12) (5523,4.655221907811145e-12)
        };
        \addlegendentry{$\norm{e_{10^4}}_1$, Arnoldi, exact $p_0$}
        \addplot coordinates {
            (1,0.9999999999997214) (101,0.9999999999997214) (201,0.003193343448207156) (301,1.0636349656083822e-7) (401,5.426755841233736e-10) (801,5.474716192697255e-10) (3000,5.580831224640294e-10) (5523,5.601280670949976e-10)
        };
        \addlegendentry{$\norm{e_{10^6}}_1$, Arnoldi, exact $p_0$}
        \addplot coordinates {
            (3,1.996267487746774) (93,0.08103739715454) (447,0.002751336696423843) (813,0.002283314385808007) (1021,0.0022998681863000453) (1135,0.011948128429298623) (1567,0.011638325911909523) (1803,0.00228743294468368) (2117,0.011701582406766284) (2584,0.0010107578319757582) (2658,0.00023701249356537527) (2778,4.897748750512534e-5) (3010,0.0002005197398484969) (3073,0.0012253542411099266) (4099,4.723290446820298e-5) (5523,7.276165492753359e-15)
        };
        \addlegendentry{$\norm{e_{10^4}}_1$, Exlump, exact $p_0$}
        \addplot coordinates {
            (3,1.9962674417491288) (93,0.08103626034232529) (447,0.002751336684639088) (813,0.0022833143002223754) (1021,0.0022998675979318825) (1135,0.011948128555891401) (1567,0.011638325850456214) (1803,0.0022874326423915296) (2117,0.011701582159996388) (2584,0.0010107575552173362) (2658,0.0002370122029990647) (2778,4.897761123273923e-5) (3010,0.00020051943143081547) (3073,0.001225354205629224) (4099,4.723293679567371e-5) (5523,1.4543713377705557e-13)
        };
        \addlegendentry{$\norm{e_{10^6}}_1$, Exlump, exact $p_0$}
        \end{axis}
    \end{tikzpicture}
    \textbf{Gene model}\par\medskip
    \begin{tikzpicture}
        \begin{axis}[
        xlabel = {Aggregated state space size},
        ylabel = {$\norm{\tilde{p}_k - p_k}_1$},
        ymode=log, xmin = 20, ymin = 1e-15, legend style={cells={align=left}},
        xmax = 43957, ymax = 2, plot, legend pos = north east,
        width=0.9\linewidth, height=6.5cm
        ]
        \addplot coordinates {
            (1,0.9999999999999444) (201,0.6587687932467388) (401,0.0003855514649099467) (501,2.573163693666238e-7) (601,1.8229725718925325e-11) (801,2.290723929352115e-12) (10000,2.29145345565224e-12) (20000,2.29136852853893e-12) (30000,2.29136437438633e-12) (43957,2.29136839814596e-12)
        };
        \addlegendentry{$\norm{e_{10^4}}_1$, Arnoldi, exact $p_0$}
        \addplot coordinates {
            (1,0.99999999999884) (401,130.99441088094216) (801,2.3713831499619022) (901,0.009014472373013217) (1001,2.6196155723189905e-5) (1101,8.694450146679285e-9) (1201,5.926427299167131e-9) (1501,5.926424503754961e-10) (10000,5.926424573529503e-10) (20000,5.926424575234568e-10) (30000,5.926424564017603e-10) (43957,5.9264245769198454e-10)
        };
        \addlegendentry{$\norm{e_{10^6}}_1$, Arnoldi, exact $p_0$}
        \addplot coordinates {
            (9,1.9961241470164677) (42,1.996182001490757) (10257,1.9939769913780414) (15072,1.9931514373096355) (16255,1.803599246729235) (18620,1.7311750713546123) (19738,1.7402383080063302) (23058,1.3188373676417149) (25035,1.3855001493612977) (28557,1.3728751636567667) (30476,1.3761130324038422) (31755,1.3727051844015754) (43957,2.950509031464231e-15)
        };
        \addlegendentry{$\norm{e_{10^4}}_1$, Exlump, exact $p_0$}
        \addplot coordinates {
            (9,1.9961241470164677) (42,1.996182001490757) (10257,1.9939769913780414) (15072,1.9931514373096355) (16255,1.803599246729235) (18620,1.7311750713546123) (19738,1.7402383080063302) (23058,1.3188373676417149) (25035,1.3855001493612977) (28557,1.3728751636567667) (30476,1.3761130324038422) (31755,1.3727051844015754) (43957,2.950509031464231e-15)
        };
        \addlegendentry{$\norm{e_{10^6}}_1$, Exlump, exact $p_0$}
        \end{axis}
    \end{tikzpicture}
    \caption{$\norm{e_k}_1$ in different models for Arnoldi and Exlump aggregations}
    \label{fig:errorcomp}
\end{figure}
In the RSVP model, both aggregation methods perform similarly for an aggregated state space size $j \geq 150$.
Before this point, the error made by the Arnoldi aggregation is larger or equal in all cases.
For Exlump, the error is bounded above by two.
This bound is explained by $\mat{\Pi}$ acting like a Markov chain and by $\pi_k$ and $\tilde{p}_k$ still being probability distributions~\cite[3]{michel2025errbndmarkovaggr}, as the 1-norm of the difference of two probability distributions can be at most two.
Although not shown here, $\norm{e_{10^6}}_1$ in Exlump aggregations with non-exact $p_0$ is close to $\norm{e_{10^6}}_1$ for exact $p_0$.
This closeness contrasts the exact and non-exact $p_0$ at $k = 10^4$.
This effect is caused by $\tilde{p}_{10^6}$ and $p_{10^6}$ having approximated the same eigenvector or stationary distribution of $\mat{P}$ up to an error of $10^{-8}$.
For $k = 10^4$, this convergence has not yet happened.

If we look at the larger cluster model, the behaviour is different.
The Arnoldi aggregation consistently outperforms the Exlump aggregation, except for very small or large $j$.
In the case of small $j$, it is for the same reason mentioned beforehand.
Still, both methods do not yield usable aggregations for $1 \leq j \leq 100$.
However, the exact aggregation found by Exlump at $j = 5523$ performs better than the Arnoldi aggregation, which has already stopped improving at $j \approx 400$.
Even further, $\norm{e_k}_1$ is the same for both $k = 10^4$ and $k = 10^6$ in the Exlump aggregation.
This can also be observed for more $k$.
In contrast, $\norm{e_k}_1$ in the Arnoldi aggregation is generally bounded through $k \cdot \macheps \leq \norm{e_k}_1$, as already seen in Figure~\ref{fig:findingexactconvcriterion}.
This difference between aggregation methods is again explained by $\pi_0$ converging more rapidly towards an eigenvector in Exlump aggregations than in Arnoldi aggregations.
For example, $\norm{\pi_{10^4}\tran - \pi_{10^4}\tran \mat{\Pi}}_1$ is roughly $10^{-28}$ in Exlump aggregations and $10^{-12}$ in Arnoldi aggregations for $j = 2778$.
So, computing $\pi_{k + 1}\tran = \pi_k\tran \mat{\Pi}$ is like applying $\mat{I}$ to $\pi_k\tran$, where practically no error is made.
Thus, computing $\tilde{p}_k$ in Exlump aggregations is numerically more stable, giving us $\abs{\norm{e_{10^4}}_1 - \norm{e_{10^6}}_1} \approx 10^{-10}$.
We will continue to observe this effect.
Lastly, using non-exact $p_0$ yields the same result.
For Exlump aggregations, this is explained by the stationary distribution of the cluster model being close to uniform, which random $p_0$ are as well.

Finally, we look at the gene model, for which the Exlump method finds no exact aggregation.
Hence, Exlump finds no sensible aggregation except for $j = n = 43,957$.
In contrast, Arnoldi aggregations quickly improve with larger $j$ until $j \approx 1100$.
From this aggregated state space size onward, we again reach the typical error observed for the other models with Arnoldi aggregations.
The same behaviour was observed in the Lotka-Volterra model, where Exlump fails to yield usable aggregations until $j = n = 10,201$, while the Arnoldi aggregations gave the same small errors as above for $j \geq 600$.

Although not pictured here, the behaviour of the Lotka-Volterra model concerning Exlump and Arnoldi aggregations mirrors that of the gene model, although extended to its larger state space size.
This similarity hints at this behaviour being expected for Markov chains where Exlump can not find an exact aggregation.

In summary, Exlump outperforms Arnoldi at the same aggregated state space size for small Markov chains with exact aggregations.
It is the other way around for larger Markov chains with exact aggregations, except for when Exlump finds this exact aggregation.
Exlump's superiority over Arnoldi at exact aggregations is due to Exlump's numerically more stable computation of $\tilde{p}_k$.
This instability of Arnoldi implies that if $k \geq 10^{11}$, then Arnoldi aggregations will always yield $\norm{e_k}_1 \geq 1$, making it useless for such $k$.
This problem does not arise in Exlump aggregations.
Lastly, for Markov chains where Exlump finds no exact aggregation, it fails to find any usable aggregations, except if $j = n$, while Arnoldi eventually converges much earlier than usual.

Now that we have compared the errors made by both aggregation methods, we must also compare their runtime to assess their usability.
We will compute $\tilde{p}_{10^5}$ and $p_{10^5}$ respectively for differing aggregated state space sizes.
We will also further measure the runtime piecewise for Arnoldi aggregations to determine the expensive parts.
Thus, we measure once without computing the convergence criterion from Section~\ref{subsec:determining-convergence} with Algorithm~\ref{alg:qralgorithm} and then include it in another run.
Similarly, we exclude the computation time for $\tilde{p}_{10^5}$ in one run and include it in another one.
As Exlump is implemented in Python and Arnoldi in Julia, we compare their runtimes to the runtime needed to compute $p_{10^5}$ in their respective languages, as not to compare the speed of Python and Julia.
Again, we use the same three models as before.
\begin{figure}[H]
    \centering
    \begin{tikzpicture}
        \begin{axis}[
            xlabel = {Aggregated state space size},
            ylabel = {Runtime in nanseconds}, domain=0:300,
            ymode=log, xmin =1, ymin = 1e4, legend style={cells={align=left}},
            xmax = 235, ymax = 1e10, plot, legend pos = south east,
            width=0.9\linewidth, height=6.5cm
        ]
            \addplot coordinates {
                (1.0,16167.0) (14.0,244874.99999999997) (27.0,683708.0) (40.0,1.327875e6) (53.0,2.132958e6) (66.0,3.118167e6) (79.0,4.288208e6) (92.0,5.648042e6) (105.0,7.1655e6) (118.0,8.897708e6) (131.0,1.0921125e7) (144.0,1.2932e7) (157.0,1.5126333e7) (170.0,1.7543666e7) (183.0,2.0087167e7) (196.0,2.2875292e7) (209.0,2.5942334e7) (222.0,2.9546458e7) (235.0,3.2515125e7) (248.0,4.1434625e7) (261.0,3.9724833e7) (274.0,4.3316375e7) (287.0,4.774925e7) (300.0,5.2196708e7)
            };
            \addlegendentry{Alg.~\ref{alg:arnoldi-iteration}}
            \addplot coordinates {
                (1.0,28167.0) (14.0,1.417333e6) (27.0,6.905583e6) (40.0,1.8450959e7) (53.0,3.9315125e7) (66.0,7.1557625e7) (79.0,1.2844470899999999e8) (92.0,2.06478291e8) (105.0,3.114835e8) (118.0,4.4077325e8) (131.0,6.1829525e8) (144.0,8.84453e8) (157.0,1.11042825e9) (170.0,1.361418583e9) (183.0,1.690980875e9) (196.0,2.211813958e9) (209.0,3.411024208e9) (222.0,3.279998667e9) (235.0,3.818750042e9) (248.0,4.413651417e9) (261.0,6.387936083e9) (274.0,6.506113958e9) (287.0,8.91052725e9) (300.0,8.947452458e9)
            };
            \addlegendentry{Alg.~\ref{alg:arnoldi-iteration} + Alg.~\ref{alg:qralgorithm}}
            \addplot coordinates {
                (1.0,3.186125e6) (14.0,2.3111291e7) (27.0,4.3066084e7) (40.0,4.466575e7) (53.0,7.4442042e7) (66.0,1.17590542e8) (79.0,1.66460541e8) (92.0,2.07320792e8) (105.0,2.76531333e8) (118.0,3.46874208e8) (131.0,4.83037583e8) (144.0,6.13090917e8) (157.0,7.38502584e8) (170.0,8.64359417e8) (183.0,1.022389583e9) (196.0,1.1495105e9) (209.0,1.511606833e9) (222.0,1.464653667e9) (235.0,1.658823291e9) (248.0,1.836984792e9) (261.0,2.159743083e9) (274.0,2.384947e9) (287.0,2.452161958e9) (300.0,2.636279583e9)
            };
            \addlegendentry{$\tilde{p}_{10^5}$ with Alg.~\ref{alg:arnoldi-iteration}}
            \addplot coordinates {
                (1.0,3.198708e6) (14.0,2.4334583e7) (27.0,4.93595e7) (40.0,6.2896458e7) (53.0,1.14603333e8) (66.0,1.90650959e8) (79.0,2.97788209e8) (92.0,4.08959792e8) (105.0,5.97906917e8) (118.0,8.02907083e8) (131.0,1.26577625e9) (144.0,1.486441625e9) (157.0,1.7811855e9) (170.0,2.210081917e9) (183.0,3.0760695e9) (196.0,4.616175416e9) (209.0,5.213340125e9) (222.0,6.080823666e9) (235.0,7.721852333e9) (248.0,6.158582292e9) (261.0,7.244801417e9) (274.0,1.1122885e10) (287.0,9.587029875e9) (300.0,1.1162920875e10)
            };
            \addlegendentry{$\tilde{p}_{10^5}$ with Alg.~\ref{alg:arnoldi-iteration} + Alg.~\ref{alg:qralgorithm}}
            \addplot+[no marks, gray, dashed]
                {4.64735667e8}
            node[pos=0.15, above] {Computing $p_{10^5}$};
        \end{axis}
    \end{tikzpicture}
    \begin{tikzpicture}
        \begin{axis}[
            xlabel = {Aggregated state space size},
            ylabel = {Runtime in nanseconds}, domain=0:234,
            ymode=log, xmin = 1, ymin = 1e7, legend style={cells={align=left}},
            xmax = 235, ymax = 1e11, plot, legend pos = south east,
            width=0.9\linewidth, height=6.5cm
        ]
            \addplot coordinates {
                (234,170905113.22021484) (233,302190303.80249023) (182,230839967.72766113) (123,205775260.92529297) (121,224488973.6175537) (110,195050954.8187256) (87,163579940.79589844) (83,165996074.67651367) (31,47124862.67089844) (26,50933122.634887695) (20,89674949.6459961) (16,48655033.111572266) (1,48215779.35736731)
            };
            \addlegendentry{Exlump}
            \addplot coordinates {
                (234,570801973.3428955) (233,688514709.4726562) (182,608697891.2353516) (123,541729927.0629883) (121,552939891.8151855) (110,524688005.4473877) (87,1029531002.0446777) (83,486077308.65478516) (31,342112064.36157227) (26,342994213.10424805) (20,323736906.05163574) (1,303632351.47033436)
            };
            \addlegendentry{$\tilde{p}_{10^5}$ with Exlump}
            \addplot+[no marks, gray, dashed]
                {6337805747}
            node[pos=0.5, above] {Computing $p_{10^5}$};
        \end{axis}
    \end{tikzpicture}
    \caption{Runtime of computing Arnoldi and Exlump aggregations of the RSVP model}
    \label{fig:runtimersvpcomp}
\end{figure}
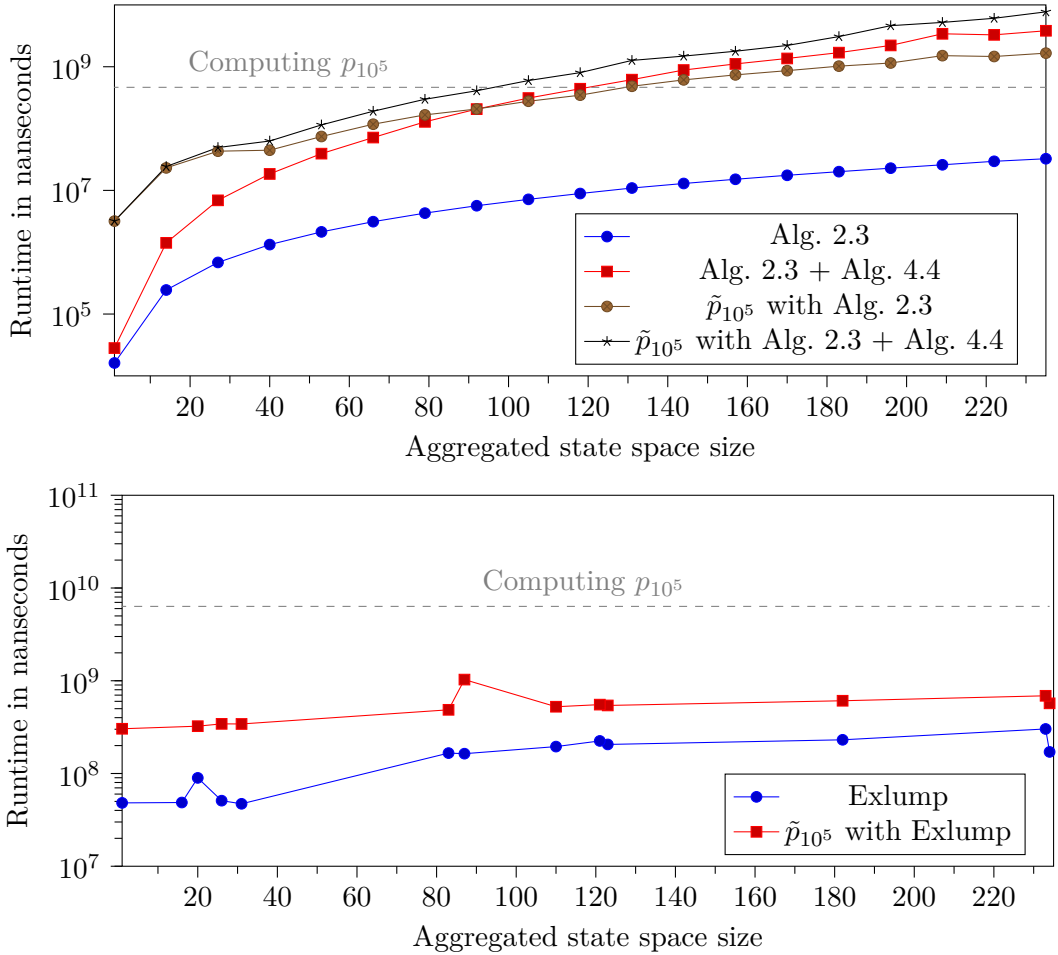
Firstly, the convergence criterion via Algorithm~\ref{alg:qralgorithm} causes the most significant hit to the runtime, accounting for up to 80\% of the total runtime.
Then, 19\% is taken to compute $\tilde{p}_{10^5}$.
If we include the convergence criterion, we can see in Figure~\ref{fig:runtimersvpcomp} that Exlump consistently outperforms Arnoldi for the same $j$, considering Figure~\ref{fig:errorcomp}, with Exlump being consistently faster than just computing $p_{10^5}$.
Even if we exclude Algorithm~\ref{alg:qralgorithm}, making the Arnoldi method faster than naively computing $p_{10^5}$ for $j \leq 130$, Exlump is still faster at a similar, sometimes even better, error $\norm{e_k}_1$.
Even though Exlump's $\mat{\Pi}$ is sparse, and $\mat{H}_j$ is not, again, about 19\% of the total runtime is needed for Exlump to compute $\tilde{p}_{10^5}$.
This may be due to the small aggregated state space size, where the runtime difference given by sparse vector-matrix multiplication is not yet fully exploited or because of the aggregation being computed being relatively fast.
\begin{figure}[H]
    \centering
    \begin{tikzpicture}
        \begin{axis}[
            xlabel = {Aggregated state space size},
            ylabel = {Runtime in nanseconds}, domain=0:505,
            ymode=log, xmin = 1, ymin = 1e5, legend style={cells={align=left}},
            xmax = 505, ymax = 1e11, plot, legend pos = south east,
            width=0.9\linewidth, height=6.5cm
        ]
            \addplot coordinates {
                (1.0,283167.0) (22.0,7.914249999999999e6) (43.0,1.8899417e7) (64.0,3.5338583e7) (85.0,6.3126959e7) (106.0,8.48735e7) (127.0,1.15376875e8) (148.0,1.67903042e8) (169.0,2.13173625e8) (190.0,2.56115541e8) (211.0,3.74552667e8) (232.0,4.31593833e8) (253.0,5.0098699999999994e8) (274.0,6.903775e8) (295.0,7.24549041e8) (316.0,8.7795125e8) (337.0,9.57630625e8) (358.0,1.299554625e9) (379.0,1.236357083e9) (400.0,1.37194575e9) (421.0,1.528603083e9) (442.0,1.782771083e9) (463.0,1.911957709e9) (484.0,2.163060625e9) (505.0,2.470805834e9)
            };
            \addlegendentry{Alg.~\ref{alg:arnoldi-iteration}}
            \addplot coordinates {
                (1.0,531334.0) (22.0,1.09449833e8) (43.0,4.49423375e8) (64.0,1.381803e9) (85.0,2.839345e9) (106.0,6.428480625e9) (127.0,8.351942624999999e9) (148.0,1.3520093125e10) (169.0,2.0127933959e10) (190.0,2.6863216667e10) (211.0,3.6253416583e10) (232.0,4.9224772667e10) (253.0,5.6920589375e10) (274.0,7.6336695292e10) (295.0,9.1691585e10) (316.0,1.02351404917e11)
            };
            \addlegendentry{Alg.~\ref{alg:arnoldi-iteration} + Alg.~\ref{alg:qralgorithm}}
            \addplot coordinates {
                (1.0,3.928208e6) (22.0,4.1529458e7) (43.0,8.4052625e7) (64.0,1.54313458e8) (85.0,2.42091541e8) (106.0,3.78725792e8) (127.0,5.1614574999999994e8) (148.0,8.20015e8) (169.0,1.030604833e9) (190.0,1.346032167e9) (211.0,1.992986e9) (232.0,2.018570167e9) (253.0,2.607595792e9) (274.0,2.927998708e9) (295.0,3.393051125e9) (316.0,3.951233542e9) (337.0,4.396471584e9) (358.0,4.832830709e9) (379.0,5.717839958e9) (400.0,6.444834792e9) (421.0,7.101320125e9) (442.0,7.215499667e9) (463.0,8.323566916000001e9) (484.0,9.102346625e9) (505.0,9.866025708e9)
            };
            \addlegendentry{$\tilde{p}_{10^5}$ with Alg.~\ref{alg:arnoldi-iteration}}
            \addplot coordinates {
                (1.0,3.658167e6) (22.0,1.23741875e8) (43.0,5.1856995800000006e8) (64.0,1.431290625e9) (85.0,3.026162125e9) (106.0,5.388370458e9) (127.0,8.796733375e9) (148.0,1.4170757125e10) (169.0,2.1365119334e10) (190.0,2.9054022666e10) (211.0,3.4930075e10) (232.0,5.0213824792e10) (253.0,5.8556894875e10) (274.0,7.182273575e10) (295.0,8.6034242541e10) (316.0,1.02940723333e11)
            };
            \addlegendentry{$\tilde{p}_{10^5}$ with Alg.~\ref{alg:arnoldi-iteration} + Alg.~\ref{alg:qralgorithm}}
            \addplot+[no marks, gray, dashed]
                {1.000360925e10}
            node[pos=0.125, above] {Computing $p_{10^5}$};
        \end{axis}
    \end{tikzpicture}
    \begin{tikzpicture}
        \begin{axis}[
            xlabel = {Aggregated state space size},
            ylabel = {Runtime in nanseconds}, domain=0:5523,
            ymode=log, xmin = 1, ymin = 1e8, legend style={cells={align=left}},
            xmax = 5523, ymax = 1e12, plot, legend pos = south east,
            width=0.9\linewidth, height=6.5cm
        ]
            \addplot coordinates {
                (5523,83314050912.85706) (4389,120294332981.10962) (3339,89052643060.6842) (3011,108575690984.72595) (2824,98519396305.08423) (2658,62364818096.16089) (2539,58140552997.58911) (2258,70762601852.41699) (1803,89563687801.36108) (1531,102784945964.81323) (1205,96981308937.07275) (919,35332195997.23816) (493,22250364065.170288) (427,20193339109.420776) (209,8075666904.449463) (50,5679363012.313843) (27,2677618741.9891357) (17,1083217144.0124512) (9,738503217.6971436) (3,537408828.7353516)
            };
            \addlegendentry{Exlump + I/O}
            \addplot coordinates {
                (5523,88728263854.98047) (4389,124613713026.04675) (3339,91810101985.9314) (3011,111375453948.97461) (2824,101483062982.5592) (2658,64589128971.09985) (2539,60579787969.58923) (2258,73266474962.2345) (1803,91525496006.01196) (1531,104852325201.03455) (1205,96969364881.51557) (919,34730596780.77698) (493,20454180002.212524) (427,20143305063.24768) (209,7190715074.539185) (50,4480700969.696045) (27,2837113142.01355) (17,1324820995.3308105) (9,1009343862.5335693) (3,774098873.1384277)
            };
            \addlegendentry{$\tilde{p}_{10^5}$ with Exlump + I/O}
            \addplot+[no marks, gray, dashed]
                {81227878808}
            node[pos=0.115, above] {Computing $p_{10^5}$};
        \end{axis}
    \end{tikzpicture}
    \caption{Runtime of computing Arnoldi and Exlump aggregations of the cluster model}
    \label{fig:runtimeclustercomp}
\end{figure}
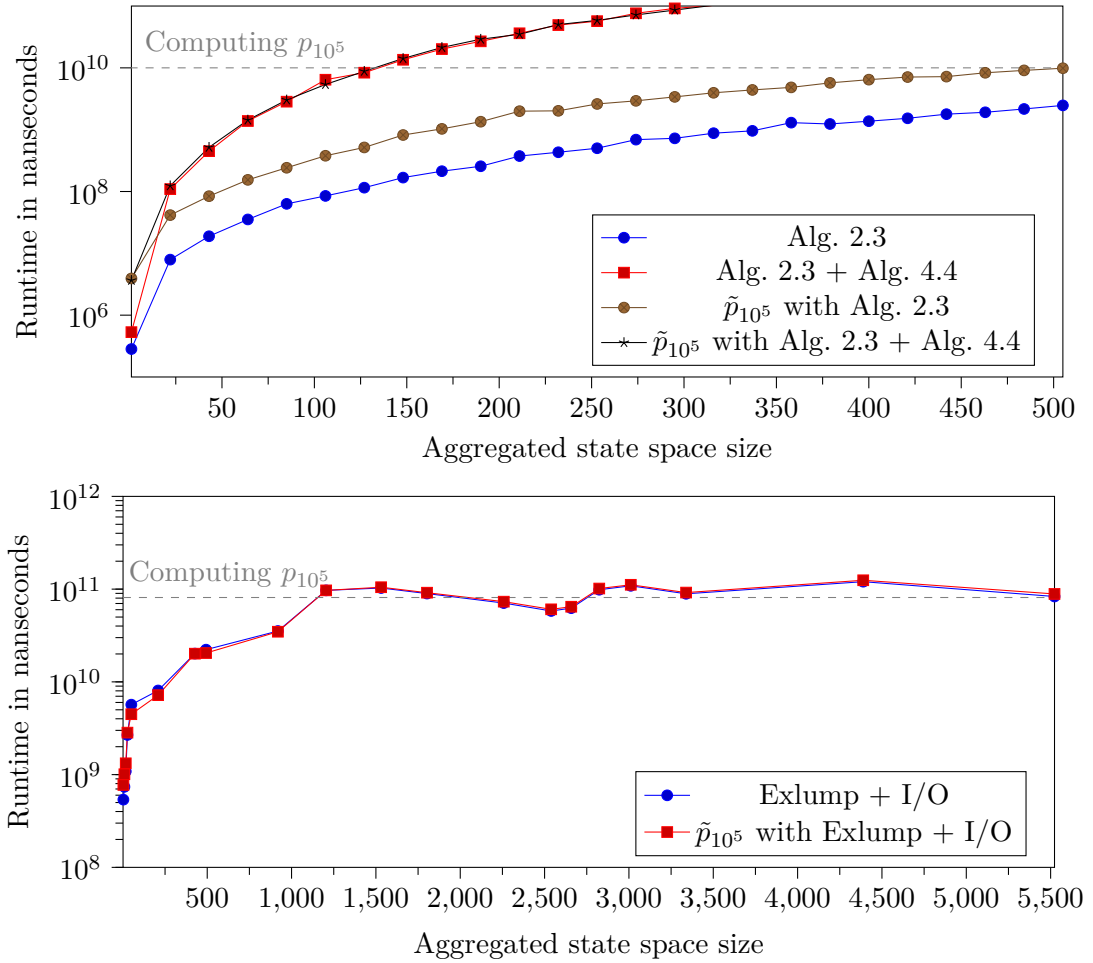
A similar picture emerges in the runtime analysis of the cluster model.
In the Arnoldi method, we cease to achieve usable runtimes after $j \approx 130$.
Due to the larger state space sizes, we have that just computing the Arnoldi aggregation with the convergence criterion takes about 97\% of the total time, while computing $\tilde{p}_{10^5}$ takes 1\%.
On the other hand, for the Exlump method, computing $\tilde{p}_{10^5}$ takes 7\% of the total time.

As the state space size is much larger here than in the RSVP model, the sparse vector-matrix multiplication is faster than the normal one.
Thus, the effect of $\tilde{p}_{10^5}$ needing, relatively, more time in Exlump than in Arnoldi is explained through the comparatively faster computation of an Exlump aggregation.
The convergence criterion will generally dominate the runtime, as it is in $\mathcal{O}(j^4)$, while the Arnoldi iteration and computation of $\tilde{p}_k$ are in $\mathcal{O}(j^2)$.

Lastly, we look at the gene model.
From Figure~\ref{fig:errorcomp}, we already know that the only usable Exlump aggregation is the same size as the original Markov chain, which will always result in a runtime worse than just using $\mat{P}$ outright.
As such, we will just look at Arnoldi aggregations, to see whether we can achieve a usable speedup in this case.

Although the Lotka-Volterra model also lacks an exact aggregation, we chose the gene model for this purpose.
This is because, as we know from Figure~\ref{fig:errorcomp}, a good Arnoldi aggregation is found earlier in the gene model than in the full state space size.
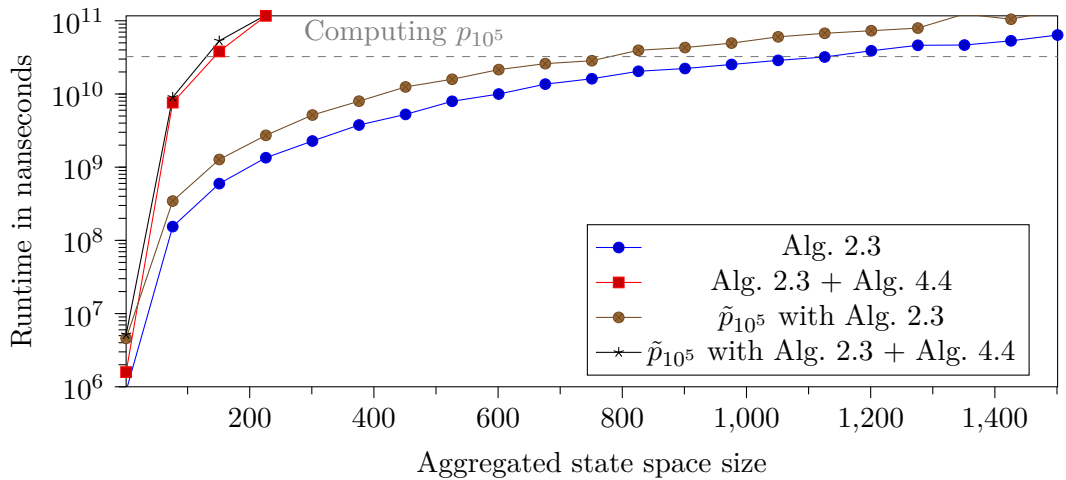
\begin{figure}[H]
    \centering
    \begin{tikzpicture}
        \begin{axis}[
            xlabel = {Aggregated state space size},
            ylabel = {Runtime in nanseconds}, domain=0:1501,
            ymode=log, xmin = 1, ymin = 1e6, legend style={cells={align=left}},
            xmax = 1501, ymax = 1.16893025292e11, plot, legend pos = south east,
            width=0.9\linewidth, height=6.5cm
        ]
            \addplot coordinates {
                (1,877708.0) (76,1.54160958e8) (151,5.9600975e8) (226,1.347271916e9) (301,2.276919e9) (376,3.76796225e9) (451,5.26298625e9) (526,7.948147792e9) (601,9.984156584e9) (676,1.3627496542e10) (751,1.6156159959e10) (826,2.0428208458e10) (901,2.2274822833e10) (976,2.5303289875e10) (1051,2.8800052708e10) (1126,3.2028526417000004e10) (1201,3.8946550542e10) (1276,4.6280826e10) (1351,4.660551475e10) (1426,5.325985075e10) (1501,6.3921290125e10)
            };
            \addlegendentry{Alg.~\ref{alg:arnoldi-iteration}}
            \addplot coordinates {
                (1,1.588208e6) (76,7.620807416e9) (151,3.8087862709e10) (226,1.16893025292e11)
            };
            \addlegendentry{Alg.~\ref{alg:arnoldi-iteration} + Alg.~\ref{alg:qralgorithm}}
            \addplot coordinates {
                (1,4.580625e6) (76,3.44494833e8) (151,1.273896833e9) (226,2.723861375e9) (301,5.155652917e9) (376,7.977081875e9) (451,1.2486718792e10) (526,1.5898223208e10) (601,2.1526933875e10) (676,2.5981260125e10) (751,2.8473287625e10) (826,3.9546000125e10) (901,4.3083307875e10) (976,4.9539523084e10) (1051,6.036477925e10) (1126,6.74903875e10) (1201,7.3185328125e10) (1276,7.9510071167e10) (1351,1.23776516042e11) (1426,1.04997321625e11) (1501,1.40035275041e11)
            };
            \addlegendentry{$\tilde{p}_{10^5}$ with Alg.~\ref{alg:arnoldi-iteration}}
            \addplot coordinates {
                (1.0,5.166417e6) (76.0,9.02639025e9) (151.0,5.32050945e10) (226.0,1.1977875925e11)
            };
            \addlegendentry{$\tilde{p}_{10^5}$ with Alg.~\ref{alg:arnoldi-iteration} + Alg.~\ref{alg:qralgorithm}}
            \addplot+[no marks, gray, dashed]
                {3.2470448625000004e10}
            node[pos=0.3, above] {Computing $p_{10^5}$};
        \end{axis}
    \end{tikzpicture}
    \caption{Runtime of computing Arnoldi aggregations of the gene model}
    \label{fig:runtimegene}
\end{figure}
Again, if we include the convergence criterion, we fail to compute a usable aggregation quickly enough.
After $j \approx 150$, computing $\tilde{p}_{10^5}$ takes longer than computing $p_{10^5}$.
When ignoring the convergence criterion, aggregations up to $j \approx 800$ are computed fast enough to be useful.
There, we get the following errors and speedups, when compared to the time needed to compute $p_k$:
\begin{table}[H]
    \centering
    \begin{tabular}{|c||c|c|}
        \hline
        $j$ & $\approx\frac{\text{Runtime for}\ \tilde{p}_{10^5}}{\text{Runtime for}\ p_{10^5}}$ & $\approx\norm{e_{10^5}}$ \\ \hline\hline
        500 & $0.10$                                                                           & $2 \cdot 10^{-1}$                 \\ \hline
        600 & $0.66$                                                                           & $8 \cdot 10^{-2}$                 \\ \hline
        700 & $0.82$                                                                           & $6 \cdot 10^{-3}$                 \\ \hline
        800 & $1.01$                                                                           & $1 \cdot 10^{-3}$                \\ \hline
    \end{tabular}
    \caption{Speedup and $\norm{e_{10^5}}_1$ in Algorithm~\ref{alg:arnoldi-iteration} with the gene model}
    \label{tab:runtimeerrorgene}
\end{table}
The Arnoldi method with the convergence criterion is too slow to achieve significant or even any speedups.
On the other hand, Exlump can give useful aggregations in some cases, such as in Figure~\ref{fig:runtimersvpcomp} or, to a more limited degree, in Figure~\ref{fig:runtimeclustercomp}.
Thus, Exlump is ahead of Arnoldi if an exact aggregation can be found.
When removing the costly convergence criterion from Arnoldi, we can sometimes find a useful aggregation of a Markov chain where Exlump failed; see Figures~\ref{fig:errorcomp} and~\ref{fig:runtimegene}.
Even then, the needed runtime limits the usability.

\chapter{Conclusion}\label{ch:conclusion}
To conclude this thesis, we summarize all important results in Section~\ref{sec:results}.
This summary includes important theoretical properties in exact arithmetic of Arnoldi aggregations, the behaviour of Arnoldi aggregations in floating-point arithmetic, its runtime, and an evaluation of computed Arnoldi aggregations.
Then, in Section~\ref{sec:outlook}, we highlight important gaps and shortcomings of our current work while offering some possible approaches to fixing them.
These mainly concern the runtime of our implementation.
\section{Results}\label{sec:results}
Theorem~\ref{thrm:arnoldi-aggr-is-smallest-k-1-exact} shows that an Arnoldi aggregation of size $j$ is $(j - 1)$-exact with minimal state space size.
In Theorem~\ref{thrm:arnoldi-aggr-is-smallest-exact}, we have also seen that if an Arnoldi aggregation is exact, it is again of minimal state space size.
Both results are shown in exact arithmetic only.
Then, in Section~\ref{sec:error-bounds-on-approximated-transient-distributions} an example is provided to show that improving the error bounds presented in~\cite{michel2025errbndmarkovaggr} is not possible for Arnoldi aggregations in general and specifically for \ac{NCD} Markov chains.

Moving on to practical results, we have shown in Section~\ref{subsec:normalization-of-the-approximated-transient-distribution} and Figure~\ref{fig:normtest} that normalizing $\tilde{p}_k$ with respect to the 1-norm helps increase its accuracy, especially if the error is substantial.
Further, we have proven two conditions under which we know normalization helps.
However, these apply very rarely, so we have provided another three postulations, which remain unproven but can be applied more commonly.
Then, in Section~\ref{subsec:validating-the-theory}, we see that Theorem~\ref{thrm:arnoldi-aggr-is-smallest-k-1-exact} holds in floating-point arithmetic, whereas Theorem~\ref{thrm:arnoldi-aggr-is-smallest-exact} does not.
To determine when an Arnoldi aggregation has reached a sufficient level of accuracy, we introduce and experimentally validate a new convergence criterion via $\inp{\abs{\pi}, \abs{\mat{H}_j \mat{Q}_j - \mat{Q}_j \mat{P}} \cdot \mathbf{1}_1}$.

Lastly, we compare Arnoldi aggregations to Exlump aggregations from~\cite{michel2025errbndmarkovaggr} for four Markov chain models.
Exlump outperforms Arnoldi for smaller models where Exlump can find an exact aggregation smaller than the original Markov chain, as shown in Figures~\ref{fig:errorcomp} and~\ref{fig:runtimersvpcomp}.
If the model is larger but Exlump can still find such an exact aggregation, Arnoldi yields a smaller error but at a substantially larger runtime, making Exlump the more viable method.
However, if Exlump fails to find an exact aggregation substantially smaller than the original Markov chain, it is outperformed by Arnoldi, which can find aggregation with a much smaller error, at the price of a long runtime.
Figures~\ref{fig:runtimersvpcomp},~\ref{fig:runtimeclustercomp} and~\ref{fig:runtimegene} show that the computation of $\pi$ for our convergence criterion mainly contributes to this runtime.
This runtime makes it slower to compute an Arnoldi aggregation and the approximated transient distribution $\tilde{p}_k$, especially for larger $k$.
Although, if $k$ is smaller, there can be cases where using the Arnoldi aggregation with a small error is faster than naively computing $p_k$.

Another area where Arnoldi can outperform Exlump is the aggregation of non-sparse Markov chains.
This shift in performance is caused by $\mat{\Pi}$ retaining the sparseness property of $\mat{P}$ in Exlump, which is not the case in Arnoldi.
Thus, if $\mat{P}$ is sparse, the approximated transient distribution can be computed much faster in Exlump.
This advantage vanishes if $\mat{P}$ is no longer sparse.

\section{Outlook}\label{sec:outlook}
There multiple aspects which can still be expanded upon.
In general, for all kinds of aggregations, the idea of normalizing $\tilde{p}_k$ from Section~\ref{subsec:normalization-of-the-approximated-transient-distribution} must be further worked upon to be useful in practice.
This work includes a (dis-)proof of Conjectures~\ref{conj:unproven1},~\ref{conj:unproven2} and~\ref{conj:unproven3}.
Still, these bounds to $\tilde{p}_k$ for~(\ref{eq:normalization-helps}) are far from tight, as we have seen in Figure~\ref{fig:normtest}.
Thus, improvements to the existing bounds are necessary as well.

Furthermore, improving Arnoldi aggregations is still possible in many aspects.
Firstly, we have formally introduced the Arnoldi aggregation only for \ac{DTMC}s in Chapter~\ref{ch:finding-small-aggregations}.
However, we have observed similar properties and performance in experiments when using \ac{CTMC}s instead.
This application to \ac{CTMC}s can also be formalized and experimented upon to investigate such aggregations.
A second way to extend the theory for Arnoldi aggregations is to look for specific Markov chains where the error bound can be improved, in contrast to Section~\ref{sec:aggregations-under-error-bounds}.

Lastly, and most importantly, the applicability to real-world problems must be improved, mainly by decreasing the runtime, as seen in Section~\ref{subsec:comparison-to-exlump-aggregations}.
This is done the easiest by adapting the current implementation.
Currently, we compute $\inp{\abs{\pi}, \abs{\mat{H}_j \mat{Q}_j - \mat{Q}_j \mat{P}} \cdot \mathbf{1}_n}$ for every aggregation size.
Thus, it is possible to compute it only every $i$ steps, reducing the runtime of this, by far, most intensive part by a factor of $\frac{1}{i}$.
As this method would inevitably lead to larger aggregated state space sizes, one must determine $i$ to strike a good balance between terminating as early as possible and computing $\inp{\abs{\pi}, \abs{\mat{H}_j \mat{Q}_j - \mat{Q}_j \mat{P}} \cdot \mathbf{1}_n}$ as rarely.
Nonetheless, the overall runtime complexity remains the same, as it only changes by a factor of $\frac{1}{i}$, meaning that finding another, faster convergence criterion is still very much needed for a larger aggregated state space size.

One way to decrease the runtime is to use a faster implementation of the Gram-Schmidt procedure.
We use a simple non-parallelized version of \ac{CGSIR}.
However, it is, for example as shown in~\cite{lingen2000parallelgs}, possible to parallelize all versions of the Gram-Schmidt procedure, although to different degrees.
Further, another version of the Gram-Schmidt procedure, besides the ones we know so far (\ac{CGS}, \ac{MGS},\dots) might yield better results at fewer necessary floating-point operations for large vectors.
One such method is presented with the \enquote{randomized Gram-Schmidt procedure} from~\cite{balabanov2022rgs}.
Still, it is unclear how this method can be adapted to be used in the Arnoldi iteration.

Another approach to decrease the runtime is to keep the aggregated state space size $j$ small.
This reduction is possible through so-called restarting.
In general, this means reducing an Arnoldi relation (see Lemma~\ref{lem:arnoldi-relation}) of size $j$ to some other Arnoldi relation of size $j - \ell$ and then expanding the smaller relation again to size $j$.
We repeat this so-called restarting until meeting some criterion.
The two most common methods of restarting are shown in~\cite[195--215]{arbenz2016lecturenotes} and~\cite{stewart2002krylovschur}.
In a rough prototype, we applied the algorithm from~\cite{stewart2002krylovschur} to our gene model.
Across different, sensible values for the number of restarts and aggregated state space size, $\pi_0$, as chosen in Definition~\ref{def:arnoldi-aggr} and chosen as an approximate solution of $\pi_0 \tran \mat{A} = p_0$, yielded bad results for $\norm{e_k}_1$ in general.
So, it is unclear which starting vector, if any, yields low initial errors.

Further, the above algorithms were developed to yield good approximations of eigenvalues and -vectors, which does not imply good properties for an aggregation.
As such, finding a restarting method that is purpose-fit for aggregating Markov chains may be necessary.
This method should ideally resolve the problems above while maintaining a low error, similar to the one in Figure~\ref{fig:errorcomp}.

\backmatter

\end{document}